\documentclass[11pt]{amsart}
\usepackage{lmodern}
\usepackage[T1]{fontenc}
\usepackage{microtype}
\usepackage{amssymb}
\usepackage{amsthm}
\usepackage{amscd}
\usepackage{hyperref}
\usepackage{cite}
\usepackage{color}

\def\lN{\mathop{\mathcal{L}_{\mathbb{N}}}\nolimits}	 		
\def\ln{\mathop{\mathcal{L}_{[n]}}\nolimits}					
\def\lS{\mathop{\mathcal{L}_{S}}\nolimits}					
\def\Lip{\mathop{\mathrm{Lip}}\nolimits}					
\def\equalinlaw{\mathop{=_{\mathcal{D}}}\nolimits}			
\def\Xbf{\mathop{\mathbf{X}_{}}\nolimits}					
\def\age{\mathop{\text{age}}\nolimits}					

\newcommand{\xnorm}[1]{ \Vert #1 \Vert }
\def\Nb{\mathop{\mathbb{N}_{}}\nolimits}					
\def\im{\mathop{\text{im}}\nolimits}						
\def\partitionsN{\mathop{\mathcal{P}_{\mathbb{N}}}\nolimits}		
\def\ar{\mathop{\mathrm{ar}}\nolimits}	
\def\XS{\mathop{\mathcal{X}_S}\nolimits}	
\def\XN{\mathop{\mathcal{X}_{\Nb}}\nolimits}	
\def\Xn{\mathop{\mathcal{X}_{[n]}}\nolimits}		
\def\Borel{\mathop{\subseteq_{\text{Borel}}}\nolimits}	

\def\idn{\mathop{\text{id}_{[n]}}\nolimits}		
\def\idN{\mathop{\text{id}_{\Nb}}\nolimits}		
\def\id{\mathop{\text{id}_{}}\nolimits}
\def\rng{\mathop{\text{support}}\nolimits}		
\def\graphsN{\mathop{\mathcal{G}_{\Nb}}\nolimits}	
\def\graphsn{\mathop{\mathcal{G}_{[n]}}\nolimits}	
\def\coag{\mathop{\text{Coag}}\nolimits}		
\def\idn{\mathop{\text{id}_{[n]}}\nolimits}
\def\idN{\mathop{\text{id}_{\Nb}}\nolimits}
\def\0{\mathop{\mathbf{0}_{}}\nolimits}

\def\1{\mathop{\mathbf{1}_{\mathbb{N}}^{\mathcal{L}}}\nolimits}
\def\finite{\mathop{\subset_f}\nolimits}
\def\E{\mathop{\mathcal{E}_{}}\nolimits}
\def\odap{\mathord{<}\omega}

\newtheorem{theorem}{Theorem}[section]
\newtheorem{lemma}[theorem]{Lemma}
\newtheorem{prop}[theorem]{Proposition}
\newtheorem{cor}[theorem]{Corollary}

\theoremstyle{definition}
\newtheorem{definition}[theorem]{Definition}

\newtheorem{remark}[theorem]{Remark}

\newtheorem{example}[theorem]{Example}

\newtheorem{claim}{Claim}

 \newenvironment{claimproof}{\begin{proof}}{\end{proof}}

\def \dom{\operatorname{dom}}
\def\cod{\operatorname{cod}}
\def \rng{\operatorname{rng}}

\makeatletter

\def\dotminussym#1#2{%
  \setbox0=\hbox{$\m@th#1-$}%
  \kern.5\wd0%
  \hbox to 0pt{\hss\hbox{$\m@th#1-$}\hss}%
  \raise.6\ht0\hbox to 0pt{\hss$\m@th#1.$\hss}%
  \kern.5\wd0}

\mathchardef\mhyphen="2D

\allowdisplaybreaks[2]

\begin{document}

\title{The structure of combinatorial Markov processes}
\author{Harry Crane and Henry Towsner}
\address {Department of Statistics \& Biostatistics, Rutgers University, 110 Frelinghuysen Avenue, Piscataway, NJ 08854, USA}
\email{hcrane@stat.rutgers.edu}
\urladdr{\url{http://stat.rutgers.edu/home/hcrane}}
\address {Department of Mathematics, University of Pennsylvania, 209 South 33rd Street, Philadelphia, PA 19104-6395, USA}
\email{htowsner@math.upenn.edu}
\urladdr{\url{http://www.math.upenn.edu/~htowsner}}
\thanks{H.\ Crane is partially supported by NSF grants CAREER-DMS-1554092, CNS-1523785, and DMS-1308899.}
\thanks{ H.\ Towsner is partially supported by NSF grant DMS-1340666.}

\date{\today}

\begin{abstract}
Every exchangeable Feller process taking values in a suitably nice combinatorial state space can be constructed by a system of iterated random Lipschitz functions.
In discrete time, the construction proceeds by iterated application of independent, identically distributed functions, while in continuous time the random functions occur as the atoms of a time homogeneous Poisson point process.
We further show that every exchangeable Feller process projects to a Feller process in an appropriate limit space, akin to the projection of partition-valued processes into the ranked-simplex and graph-valued processes into the space of graph limits.
Together, our main theorems establish common structural features shared by all exchangeable combinatorial Feller processes, regardless of the dynamics or resident state space, thereby generalizing behaviors previously observed for exchangeable coalescent and fragmentation processes as well as other combinatorial stochastic processes.
If, in addition, an exchangeable Feller process evolves on a state space satisfying the $n$-disjoint amalgamation property for all $n\geq1$, then its jump measure can be decomposed explicitly in the sense of L\'evy--It\^o--Khintchine.
\end{abstract}

\maketitle

\section{Introduction}\label{section:introduction}

Many combinatorial Markov processes exhibit common mathematical behaviors regardless of the resident state space.
Initial observations come from the field of mathematical population genetics, where Kingman \cite{Kingman1982} introduced his {\em coalescent process} as a model for the evolution of ancestral lineages looking backwards in time.
The simple dynamics of Kingman's coalescent---each pair of lineages merges independently at unit exponential rate---produce the family of {\em $[n]$-coalescents}, where for each $n\geq1$ the $[n]$-coalescent is a Markov process on partitions of $[n]:=\{1,\ldots,n\}$.
Each  $[n]$-coalescent is {\em exchangeable}, that is, invariant under relabeling $[n]$ by any permutation, and together the family of all $[n]$-coalescents is {\em sampling consistent}, that is, the process obtained by removing all elements in $[n]\setminus[m]$ from an $[n]$-coalescent behaves as an $[m]$-coalescent.
{\em Kingman's coalescent} is the projective limit of $[n]$-coalescents to an exchangeable Feller process on partitions of $\Nb:=\{1,2,\ldots\}$.

Since its inception, coalescent theory has played a significant role in population genetics and stochastic process theory.
Pitman \cite{Pitman1999a} and Schweinsberg \cite{Schweinsberg2000} later expanded upon Kingman's coalescent by allowing multiple blocks to merge at once.
Together, Kingman's coalescent and multiple merger coalescents describe the behavior of all exchangeable, consistent coalescent processes, recalling the representation of L\'evy processes \cite[Chapter 1]{BertoinLevy} as an independent superposition of Brownian motion with drift, compound Poisson process, and pure jump martingale.
There are analogous decompositions for homogeneous fragmentation processes \cite{Bertoin2001a} and exchangeable fragmentation-coalescent processes \cite{Berestycki2004}, which combine the dynamics of coalescence and fragmentation.
See \cite{Bertoin2006,Pitman2005} for many other deep connections between random partitions and stochastic process theory.

The transitions of coalescent, fragmentation, and fragmentation-coalescent processes are specified most compactly in terms of the {\em coagulation} and {\em fragmentation operators}, which determine a class of Lipschitz continuous functions on the space of partitions.
The ensuing descriptions of these processes by an iterated composition of Lipschitz continuous functions imply the Feller property, and ultimately Poissonian structure.
As our main theorems show, L\'evy--It\^o structure occurs much more generally as a consequence the exchangeability and sampling consistency conditions without any further assumption on the dynamics or state space.
These outcomes establish a link between exchangeable Feller processes on a broad class of combinatorial spaces and systems of iterated random functions, which arise much more broadly in statistics and applied probability problems on more general state spaces; see \cite{DiaconisFreedmanIterate} for a general survey.
Exchangeable combinatorial stochastic processes also play an important role in certain Bayesian nonparametrics and hidden Markov modeling applications; see, for example, \cite{Crane2016ESF,BNP}.

Given the common behaviors exhibited by these processes with different dynamics on different state spaces, it is natural to ask the extent to which the observed behavior reflects a universal property of combinatorial Markov processes.
This consideration strips away specific attributes of the aforementioned processes.
For example, the coalescent and fragmentation-type processes are defined to have specific semigroup behavior, while cut-and-paste and graph-valued processes are defined on state spaces with sufficiently tractable structure.

The discussion establishes shared elements of combinatorial stochastic processes obeying minimal regularity conditions, namely exchangeability, c\`adl\`ag sample paths, and the Feller property, but which can otherwise evolve on exotic state spaces.
Our main theorems prove a generic representation of discrete and continuous time Markov processes that evolve on a {\em Fra\"{i}ss\'e space}, 
culminating in a refined description for processes on a space with the additional {\em $\odap$-disjoint amalgamation property} ($\odap$-DAP).
We borrow the terms Fra\"{i}ss\'e space and $\odap$-disjoint amalgamation property from the model theory literature on homogeneous structures.
The former term describes a broad class of combinatorial spaces satisfying the basic property that all elements of the state space embed into a single universal object, called the {\em Fra\"iss\'e limit}.
By the Feller property, the dynamics of these processes are determined by the behavior at the Fra\"{i}ss\'e limit, mimicking the structure of L\'evy processes, which are characterized by their behavior at the origin, and eliciting a nice representation of the infinitesimal jump rates.
Though the terminology of Fra\"{i}ss\'e limits and other notions from model theory may be unfamiliar to many readers, the setting is quite natural and merely extends many comfortable ideas to a more general setting.
Each of the familiar examples mentioned above, whether processes valued in the space of set partitions, graphs, or hypergraphs, and many other processes of natural interest in applications, such as those evolving on orderings, $k$-colorings, or more intricate structures, evolve on a Fra\"iss\'e space.

Though our main theorems emphasize universal structural properties among combinatorial Markov processes, our main discussion also highlights key differences between certain common spaces that arise.
The nature of our representation boils down to the way in which substructures fit together to form larger structures, a notion we make precise in due course.
A well known but poorly understood illustration of this discrepancy is seen by comparing the representations of exchangeable coalescent and fragmentation processes \cite{Bertoin2006} to those of cut-and-paste \cite{Crane2014AOP} and graph-valued processes \cite{Crane2014GraphsI}.
Comparing Theorems \ref{thm:continuous-Lambda} and \ref{thm:Levy-Ito} reveals this distinction as part of a larger phenomenon of combinatorial stochastic processes.

Ultimately, our main results draw a connection between the L\'evy--It\^o--Khintchine representation for L\'evy processes and infinitely divisible distributions in probability theory \cite[Chapter 1]{BertoinLevy}, Fra\"isse's theorem and countably categorical structures in model theory \cite[Chapter 6]{HodgesShorter}, and graph limits and limits of more general combinatorial structures in combinatorial theory \cite{AroskarCummings2014,LovaszSzegedy2006}.

\subsection{Outline}

We organize the rest of the article as follows. 
We provide preliminary definitions, notation, and exposition in Section \ref{section:preliminaries}.
We summarize our most general theorems in Section \ref{section:summary}, delaying the more nuanced L\'evy--It\^o representation to Section \ref{section:Levy-Ito}.
In Sections \ref{section:Fraisse} and \ref{section:relative exchangeability}, we collect the necessary background on combinatorial state spaces and relatively exchangeable structures.
In Section \ref{section:discrete}, we prove Theorem \ref{thm:discrete} and other theorems for discrete time Markov chains.
In Section \ref{section:continuous time}, we prove Theorem \ref{thm:continuous-Lambda} for continuous time processes.
In Section \ref{section:Levy-Ito}, we prove the L\'evy--It\^o--Khintchine representation for combinatorial Markov processes, our main result.
In Section \ref{section:limits}, we discuss the projection of exchangeable combinatorial Feller processes to a Feller process in an appropriate space of limit objects.

\section{Preliminaries}\label{section:preliminaries}

To establish results in the above generality, we draw on concepts from model theory and first-order logic.
Understanding that the reader may come from any one of several backgrounds, we often spell things out more explicitly than a specialist would require.
The conditions imposed, though abstract in appearance, are actually quite natural and nonrestrictive, as many examples make apparent throughout the text.
We avoid specialized terminology as much as possible.

\subsection{Combinatorial structures}\label{section:combinatorial structures}

A \emph{signature} is a finite set of relation symbols $\mathcal{L}=\{R_1,\ldots,R_r\}$ together with, for every $j=1,\ldots,r$, a positive integer $\ar(R_j)$, called the \emph{arity} of $R_j$.  
For any set $S$, an {\em $\mathcal{L}$-structure (over $S$)} is a collection $\mathfrak{M}=(S,\mathcal{R}_1,\ldots,\mathcal{R}_r)$, where $\mathcal{R}_j\subseteq S^{\ar(R_j)}$ for each $j\in[1,r]:=[r]$.  
For any $\mathcal{L}$-structure $\mathfrak{M}=(S,\mathcal{R}_1,\ldots,\mathcal{R}_r)$, we write $\dom\mathfrak{M}:=S$ to denote the {\em domain} or {\em universe} of $\mathfrak{M}$ and $R_j^{\mathfrak{M}}:=\mathcal{R}_j$, $j\in[r]$, to denote the {\em interpretation} of $R_j$ in $\mathfrak{M}$.
For any $\mathfrak{M}=(\Nb,R_1^{\mathfrak{M}},\ldots,R_r^{\mathfrak{M}})$, we also write 
\[R_j^{\mathfrak{M}}(\vec x):=\left\{\begin{array}{cc} 1, & \vec x\in R_j^{\mathfrak{M}},\\ 0,& \text{otherwise.}\end{array}\right.\]

We write $\mathcal{L}_S$ to denote the set of all $\mathcal{L}$-structures $\mathfrak{M}$ for which $\dom\mathfrak{M}=S$.
Specifically, $\lN$ denotes $\mathcal{L}$-structures with $\dom\mathfrak{M}=\mathbb{N}$ and $\ln$ denotes $\mathcal{L}$-structures with $\dom\mathfrak{M}=[n]$.
Writing $|S|$ to denote the {\em cardinality} of a set $S$, we call $\mathfrak{M}$ {\em finite} if $|{\dom\mathfrak{M}}|<\infty$.
If $\dom\mathfrak{M}$ is countable then we call $\mathfrak{M}$ {\em countable} and without loss of generality we assume $\dom\mathfrak{M}=\Nb$.

\begin{example}[Common examples]\label{common examples}
The concept of an $\mathcal{L}$-structure generalizes many common combinatorial structures, for example, subsets, partitions, orderings, directed and undirected graphs, $k$-ary hypergraphs, and composite structures.
\begin{itemize}
	\item {\bf Sets}: a subset $A\subseteq S$ can be represented as an $\mathcal{L}$-structure $\mathfrak{M}=(\Nb,\mathcal{R})$ with $\mathcal{L}=\{R\}$ having $\ar(R)=1$ and $\mathcal{R}=A$.
	\item {\bf Partitions}: a {\em partition} $\pi$ of $S$ is a collection of nonempty, disjoint subsets $\{B_1,B_2,\ldots\}$, called \emph{blocks}, satisfying $\bigcup_{i\geq1} B_i=S$.
We can represent $\pi$ as an $\mathcal{L}$-structure $\mathfrak{M}=(S,\mathcal{R})$ with $\mathcal{L}=\{R\}$ and $\ar(R)=2$ by 
\[(i,j)\in \mathcal{R}\quad\Longleftrightarrow\quad i\text{ and }j\text{ are in the same block of }\pi.\]
	\item {\bf Graphs}: a {\em directed graph} $G$ with vertex set $S$ and edges $E\subseteq S\times S$ is also an $\{R\}$-structure with $\ar(R)=2$.
In this case, $\mathfrak{M}=(S,\mathcal{R})$ has
\[(i,j)\in \mathcal{R}\quad\Longleftrightarrow\quad G\text{ has an edge from }i \text{ to }j\quad\Longleftrightarrow\quad (i,j)\in E.\]
(An {\em undirected graph} can be described similarly with the additional symmetry constraint: $(i,j)\in \mathcal{R}$ if and only if $(j,i)\in \mathcal{R}$.)
	\item {\bf Orderings}: an {\em ordering} of $S$ is a binary relation $\prec_S$ such that for all $i,j\in S$ with $i\neq j$
	\begin{itemize}
		\item[(i)] either $i\prec_S j$ or $j\prec_S i$, but not both, and
		\item[(ii)] $i\prec_S j$ and $j\prec_S k$ implies $i\prec_S k$.
	\end{itemize}
 We can represent $\prec_S$ as an $\mathcal{L}$-structure $\mathfrak{M}=(S, \mathcal{R})$ having $\ar(R)=2$ such that 
	\[(i,j)\in \mathcal{R}\quad\Longleftrightarrow\quad i\prec_S j.\]
	\item {\bf Graphs with partition structure}: consider $\mathcal{L}=\{R_1,R_2\}$ with $\ar(R_1)=\ar(R_2)=2$.
	Let $\mathcal{R}_1,\mathcal{R}_2\subseteq\Nb\times\Nb$ be such that $(\Nb,\mathcal{R}_1)$ is an undirected graph and $(\Nb,\mathcal{R}_2)$ is a partition of $\Nb$.  Then $\mathfrak{M}=(\Nb,\mathcal{R}_1,\mathcal{R}_2)$ is an $\mathcal{L}$-structure representing a graph $(\Nb,\mathcal{R}_1)$ with community structure described by the partition $(\Nb,\mathcal{R}_2)$. 
	Graphs with a partition of vertices are often referred to as {\em networks with community structure} in the networks literature.
\end{itemize}

Note that, although partitions and graphs have the same signature, Markov processes behave differently on these two spaces.
It is instructive to keep these two examples in mind as we continue the general exposition.
We discuss the differences between these cases further beginning in Section \ref{section:model theory}. 
\end{example}

Every injection $\phi:S'\rightarrow S$ determines a map $\mathcal{L}_S\to\mathcal{L}_{S'}$, $\mathfrak{M}\mapsto\mathfrak{M}^{\phi}:=(S',\mathcal{R}_1^{\phi},\ldots,\mathcal{R}_r^{\phi})$, with
\begin{equation}\label{eq:phi-image}\mathcal{R}_j^{\phi}(s_1,\ldots,s_{\ar(R_j)})=\mathcal{R}_j(\phi(s_1),\ldots,\phi(s_{\ar(R_j)}))\end{equation}
for each $(s_1,\ldots,s_{\ar(R_j)})\in\Nb^{\ar(R_j)}$.
In particular, every permutation $\sigma:S\rightarrow S$ determines a {\em relabeling} of $\mathfrak{M}\in\mathcal{L}_{S}$ and, when $S'\subset S$, the inclusion map, $i\mapsto i$, determines the {\em restriction of $\mathfrak{M}$ to $\mathcal{L}_{S'}$} by
\[\mathfrak{M}|_{S'}:=(S',\mathcal{R}_1\cap S'^{\ar(R_1)},\ldots,\mathcal{R}_r\cap S'^{\ar(R_r)}).\]
For $\mathcal{L}$-structures $\mathfrak{N},\mathfrak{M}$, we call $\phi:\dom\mathfrak{N}\rightarrow\dom\mathfrak{M}$ an {\em embedding of $\mathfrak{N}$ into $\mathfrak{M}$}, denoted $\phi:\mathfrak{N}\to\mathfrak{M}$, if $\mathfrak{M}^{\phi}=\mathfrak{N}$.
We write $\mathfrak{N}\subseteq\mathfrak{M}$ to denote that $\mathfrak{N}$ is an {\em embedded substructure of $\mathfrak{M}$}, that is, $\dom\mathfrak{N}\subseteq\dom\mathfrak{M}$ and $\mathfrak{M}|_{\dom\mathfrak{N}}=\mathfrak{N}$.
Two $\mathcal{L}$-structures $\mathfrak{M}$ and $\mathfrak{N}$ are {\em isomorphic}, written $\mathfrak{M}\cong\mathfrak{N}$, if there is a bijection $\phi:\dom\mathfrak{N}\to\dom\mathfrak{M}$ such that $\mathfrak{M}^{\phi}=\mathfrak{N}$ and $\mathfrak{N}^{\phi^{-1}}=\mathfrak{M}$.

We equip $\lN$ with the product discrete topology induced by the ultrametric
\begin{equation}\label{eq:ultrametric}
d_{\lN}(\mathfrak{M},\mathfrak{M}'):=1/(1+\sup\{n\in\Nb:\,\mathfrak{M}|_{[n]}=\mathfrak{M}'|_{[n]}\}),\quad \mathfrak{M},\mathfrak{M}'\in\lN,
\end{equation}
with the convention that $1/\infty=0$.
Under this metric, $\lN$ is complete, separable, and compact. 
We equip $\lN$ with the Borel $\sigma$-field generated by the restriction maps $\cdot|_{[n]}:\lN\rightarrow\ln$, $n\in\Nb$, and we write $A\Borel\lN$ to denote that $A$ is a Borel subset.

\subsection{Combinatorial Markov processes}\label{section:CMP}

We are primarily interested in stochastic processes on subspaces of $\mathcal{L}$-structures that behave nicely with respect to the natural actions of {\em relabeling} and {\em restriction}, as for the coalescent and other processes mentioned in Section \ref{section:introduction}.
We focus specifically on processes with the Markov property.

In the following definition and throughout the article, $\mathcal{L}$ is a signature, $S$ is an at most countable set, $T$ is either $\mathbb{Z}_+:=\{0,1,\ldots\}$ (discrete time) or $\mathbb{R}_+:=[0,\infty)$ (continuous time), and $\XS\subseteq\lS$ is a set of $\mathcal{L}$-structures which is closed under isomorphism and is equipped with the trace of the Borel $\sigma$-field on $\lN$ and the topology induced by the ultrametric in \eqref{eq:ultrametric}.

\begin{definition}[Markov property]\label{defn:Markov}
A family of $\mathcal{X}_S$-valued random structures $(\mathfrak{X}_t)_{t\in T}$ satisfies the {\em Markov property} if the $\sigma$-fields $\sigma\langle\mathfrak{X}_s\rangle_{s<t}$ and $\sigma\langle\mathfrak{X}_s\rangle_{s>t}$ are conditionally independent given $\sigma\langle\mathfrak{X}_t\rangle$ for all $t\in T$, where $\sigma\langle\cdot\rangle$ is the $\sigma$-field generated by $\cdot$.
\end{definition}

\begin{definition}[Combinatorial Markov process]\label{defn:CMP}
A {\em combinatorial Markov process} on $\XS$ is a collection $\Xbf=\{\Xbf_{\mathfrak{M}}:\,\mathfrak{M}\in\XS\}$, where each $\Xbf_{\mathfrak{M}}=(\mathfrak{X}_t)_{t\in T}$ is $\XS$-valued, has $\mathfrak{X}_0=\mathfrak{M}$, and satisfies the Markov property with a common time homogeneous transition law
\begin{equation}\label{eq:tps}
P_{s}(x,A):=\mathbb{P}\{\mathfrak{X}_{t+s}\in A\mid \mathfrak{X}_t=x,\,\mathfrak{X}_0=\mathfrak{M}\},\quad s,t\in T,\quad A\Borel\XN,
\end{equation}
for all $\mathfrak{M}\in\XS$.
\end{definition}

\begin{remark}[Notation and terminology]
Combinatorial Markov processes exhibit different behaviors in discrete and continuous time.
We avoid confusion by calling $\mathbf{X}$ a {\em Markov chain} when time is discrete and a {\em Markov process} when time is continuous.
To further distinguish these cases, we index time by $m=0,1,\ldots$ in discrete time and $t\in[0,\infty)$ in continuous time.
When speaking generically about discrete and continuous time processes, as we do in this section, we employ the notation and terminology of the continuous time case.
\end{remark}

\begin{definition}[Exchangeable processes]\label{defn:exchangeable CMP}
A combinatorial Markov process $\mathbf{X}=\{\mathbf{X}_{\mathfrak{M}}:\,\mathfrak{M}\in\XS\}$ is {\em exchangeable} if its transition law \eqref{eq:tps}  satisfies
\begin{equation}\label{eq:exch-tps}
P_s(x,A)=P_s(x^{\sigma},A^{\sigma}),\quad x\in\XS,\,A\Borel\XS,
\end{equation}
for all $s\in T$ and all permutations $\sigma:S\rightarrow S$, where $A^{\sigma}:=\{x'^{\sigma}:\,x'\in A\}$.
\end{definition}

When $S$ is countable, without loss of generality $S=\Nb$, we assume the further minimal condition that every $\mathbf{X}_{\mathfrak{M}}\in\mathbf{X}$ has {\em c\`adl\`ag sample paths} with probability 1, that is, the map $t\mapsto \mathfrak{X}_t$ is right continuous and has left limits.
In the product discrete topology, the c\`adl\`ag sample paths property is equivalent to the condition that every embedded process $\mathbf{X}_{\mathfrak{M}}^{S}:=(\mathfrak{X}_t|_S)_{t\in T}$, with $S\subset\Nb$ finite, stays in each state it visits for a strictly positive hold time with probability 1.
When $S$ is finite or $T=\mathbb{Z}_+$, the c\`adl\`ag paths assumption is implicit in the Markov assumption of Definition \ref{defn:CMP}.

The stronger property of {\em projectivity}, or {\em sampling consistency}, is a common assumption in statistical applications.

\begin{definition}[Projective Markov property]\label{defn:projective}
We say that $\mathbf{X}=\{\mathbf{X}_{\mathfrak{M}}:\,\mathfrak{M}\in\XS\}$ exhibits the {\em projective Markov property}, or is {\em consistent under subsampling}, if $\mathbf{X}_{\mathfrak{M}}^{S'}:=(\mathfrak{X}_t|_{S'})_{t\in T}$ satisfies the Markov property on $\mathcal{X}_{S'}$ for every $S'\subseteq S$ and $\mathfrak{M}\in\XS$.
\end{definition}

Under the projective Markov property, $\mathbf{X}$ determines a combinatorial Markov process $\mathbf{X}^{S'}:=\{\mathbf{X}_{\mathfrak{S}}: \mathfrak{S}\in\mathcal{X}_{S'}\}$ on $\mathcal{X}_{S'}$, for every $S'\subseteq S$, by taking $\mathbf{X}_{\mathfrak{S}}\equalinlaw\mathbf{X}_{\mathfrak{M}}^{S'}$ for any $\mathfrak{M}\in\mathcal{X}_{S}$ with $\mathfrak{M}|_{S'}=\mathfrak{S}$.
Since $\Xn$ is finite for every $n\in\Nb$, the projective Markov property implies that each $\Xbf_{\mathfrak{M}}$ has c\`adl\`ag sample paths with probability 1.

Our main theorems pertain to processes that evolve on certain {\em Fra\"iss\'e spaces} $\XN\subseteq\lN$, which we define formally in Definition \ref{defn:Fraisse space} as a natural generalization of a {\em combinatorial state space}.
The key feature of a Fra\"iss\'e space is that it contains a universal representative $\mathfrak{F}\in\XN$ such that every $\mathfrak{M}\in\XN$ embeds into $\mathfrak{F}$.
In Theorem  \ref{prop:Feller} we prove that the projective Markov property is equivalent to the Feller property for exchangeable processes on Fra\"iss\'e spaces.

The {\em Markov semigroup} of an $\XS$-valued Markov process $\mathbf{X}$ is a family of operators $(\mathbf{P}_t)_{t\in T}$ which acts on bounded, measurable functions $g:\XS\to\mathbb{R}$ by
\begin{equation}\label{eq:semigroup}
\mathbf{P}_tg(\mathfrak{M}):=\mathbb{E}(g(\mathfrak{X}_t)\mid \mathfrak{X}_0=\mathfrak{M}),\quad \mathfrak{M}\in\XS.
\end{equation}

\begin{definition}[Feller property]\label{defn:Feller}
An $\XS$-valued Markov process $\mathbf{X}=\{\mathbf{X}_{\mathfrak{M}}: \mathfrak{M}\in\XS\}$ with semigroup $(\mathbf{P}_t)_{t\in T}$ has the {\em Feller property} if, for all bounded, measurable $g:\XS\rightarrow\mathbb{R}$,
\begin{itemize}
	\item $\mathfrak{M}\mapsto\mathbf{P}_tg(\mathfrak{M})$ is continuous for all $t\in T$ and
	\item $\lim_{t\downarrow0}\mathbf{P}_tg(\mathfrak{M})=g(\mathfrak{M})$ for all $\mathfrak{M}\in\XS$.
\end{itemize}
\end{definition}

Most important for our purposes is that the infinitesimal generator, which determines the behavior of $\mathbf{X}$, always exists for Feller processes.
Our main theorems characterize this infinitesimal generator under general conditions on the state space $\XN$.

From here on we refer to $\mathbf{X}$ interchangeably as an {\em exchangeable, projective Markov process}, an {\em exchangeable, consistent Markov process}, or an {\em exchangeable Feller process}.

\subsection{Lipschitz continuous functions}\label{section:Lipschitz}
A function $F:\XS\rightarrow\XS$ is {\em Lipschitz continuous} if 
\begin{equation}\label{eq:Lipschitz}
d_{\XS}(F(\mathfrak{M}),F(\mathfrak{M}'))\leq d_{\XS}(\mathfrak{M},\mathfrak{M}'),\quad\text{for all }\mathfrak{M},\mathfrak{M}'\in\XS,\end{equation}
where $d_{\XS}(\cdot,\cdot)$ is the ultrametric on $\mathcal{X}_S\subseteq\mathcal{L}_S$ induced from \eqref{eq:ultrametric}.  For any $S\subseteq\Nb$, we write $\Lip(\XS)$ to denote the set of Lipschitz continuous functions $\XS\rightarrow\XS$.

\begin{remark}
  The Lipschitz condition is natural in the context of projective Markov processes.
  In our characterization of discrete time Markov chains (Theorem \ref{thm:discrete}), Lipschitz continuous functions $F$ act as random transition operators that relate the state $\mathfrak{X}_m=\mathfrak{M}$ at time $m$ to $\mathfrak{X}_{m+1}=F(\mathfrak{M})$ at time $m+1$.  
 The ultrametric  property of \eqref{eq:ultrametric} implies that $F(\mathfrak{M})|_{[n]}$ depends only on $\mathfrak{M}|_{[n]}$, so that the Lipschitz continuous functions are those which can determine the restriction of $(\mathfrak{X}_t)_{t\in T}$ to $\Xn$ at time $m+1$ by considering only the restriction to $\Xn$ at time $m$, as is needed to fulfill the projective Markov property.
 The role of Lipschitz continuous functions in describing continuous time processes (Theorem \ref{thm:continuous-Lambda}) is more intricate but analogous to the discrete time case.
\end{remark}

We define the restriction of $F\in\Lip(\XN)$ to a function $F^{[n]}:\Xn\to\Xn$, $n\in\Nb$, by
\begin{equation}\label{eq:restrict-Lip}
F^{[n]}(\mathfrak{S}):=F(\mathfrak{M})|_{[n]},\quad\mathfrak{S}\in\Xn,
\end{equation}
where $\mathfrak{M}\in\XN$ is any structure such that $\mathfrak{M}|_{[n]}=\mathfrak{S}$.
Thus, $\Lip(\XN)$ also comes equipped with the product discrete topology and Borel $\sigma$-field induced by the analog of \eqref{eq:ultrametric}:
\begin{equation}\label{eq:d-Lip}
d_{\Lip(\XN)}(F,F'):=1/(1+\sup\{n\in\Nb:\,F^{[n]}=F'^{[n]}\}),\quad F,F'\in\Lip(\XN).\end{equation}

Since $\mathcal{X}_S$ is closed under isomorphism, any permutation $\sigma:S\to S$ acts on $F\in\Lip(\XS)$ by {\em conjugation}, $F\mapsto\sigma F\sigma^{-1}$, defined by $(\sigma F\sigma^{-1})(\mathfrak{M}):=F(\mathfrak{M}^{\sigma})^{\sigma^{-1}}$.
\begin{definition}[Conjugation invariance]\label{defn:conjugation invariant}
We call $F\in\Lip(\XS)$ {\em conjugation invariant} if $\sigma F\sigma^{-1}$ is Lipschitz continuous for all permutations $\sigma:S\to S$.
\end{definition}

\begin{remark}
As we prove in Proposition \ref{prop:conj-inv},  conjugation invariance is a natural strengthening of the Lipschitz condition, saying that $F(\mathfrak{M})|_S$ depends only on
  $\mathfrak{M}|_S$ for
  any subset $S$, whereas Lipschitz continuity only
  ensures this when $S=[n]$ for some $n\in\Nb$.  Conjugation invariance is the strongest reasonable
  requirement to ensure that the transition functions are determined {\em locally}. 
\end{remark}

\begin{definition}[Exchangeable measure]\label{defn:exch-measure}
We call a measure $\mu$ on $\Lip(\XN)$ {\em exchangeable} if
\begin{equation}\label{eq:exch-mu}
\mu(\{F\in\Lip(\XN): F^{[n]}\in A\})=\mu(\{F\in\Lip(\XN): \sigma F^{[n]}\sigma^{-1}\in A\})
\end{equation}
for all $A\Borel\Lip(\Xn)$ and all permutations $\sigma:[n]\to[n]$, for all $n\in\Nb$.
Specifically, a probability measure $\mu$ is exchangeable if $F\sim\mu$ implies $F(\mathfrak{M}^{\sigma})\equalinlaw F(\mathfrak{M})^{\sigma}$ for all permutations $\sigma:\Nb\rightarrow\Nb$ and all $\mathfrak{M}\in\XN$, where $\equalinlaw$ denotes {\em equality in distribution}.
\end{definition}

\subsection{Notation}\label{section:notation}

We adopt the following notational conventions: $\mathcal{L}$ and $\mathcal{L}'$ always are signatures, $\mathfrak{M}$ is always an $\mathcal{L}$-structure, and $\mathfrak{X}$ is always a random $\mathcal{L}$-structure.
In general, we use fraktur letters, $\mathfrak{M}$, $\mathfrak{N}$, $\mathfrak{S}$, $\mathfrak{T}$, to denote structures with the base set indicated by the corresponding plain Roman letters, $M$, $N$, $S$, $T$, respectively.
When no confusion will result, we often write $\vec s\in S$ to denote that $\vec s=(s_1,\ldots,s_k)$ is a tuple in $S$ without specifying its length.

Capital letters at the end of the alphabet $\mathfrak{X},\mathfrak{Y},\mathfrak{Z}$ denote random structures and $\mathbf{X}=(\mathfrak{X}_t)_{t\in T}$ denotes a family of random structures.
We write $\mathfrak{X}\equalinlaw\mathfrak{X}^*$ to mean that $\mathfrak{X}$ and $\mathfrak{X}^*$ are {\em equal in distribution}, and for processes we write $\mathbf{X}\equalinlaw \mathbf{X}^{*}$ to mean that $\mathbf{X}$ and $\mathbf{X}^*$ have all of the same finite-dimensional distributions.
For $\vec x=(x_1,\ldots,x_k)$, we write $\rng\vec x=\{x_1,\ldots,x_k\}$ for the set of distinct elements in $\vec x$ and we write $\vec y\sqsubseteq\vec x$ to indicate that $\vec y$ occurs as a subsequence of $\vec x$, that is, there are indices $1\leq i_1<\cdots <i_{m}\leq k$ such that $\vec y=(x_{i_1},\ldots,x_{i_m})$.

\section{Summary of main theorems} \label{section:summary}
We now state two of our main theorems, saving many technical details for later.
We delay a formal statement of our most precise theorems, including the L\'evy--It\^o--Khintchine representation and properties of the projection into a space of limit objects, until Theorems \ref{thm:conjugation-discrete}, \ref{thm:Levy-Ito}, \ref{thm:limit process}, and \ref{thm:disc}.

Throughout the section, $\mathcal{L}$ is a fixed signature and $\XN\subseteq\lN$ is closed under isomorphism and satisfies the conditions of a Fra\"{i}ss\'e space, which we define in Definition \ref{defn:Fraisse space}.
Model theorists will recognize a Fra\"{i}ss\'e space as a collection of countable combinatorial structures satisfying the hereditary property (Definition \ref{defn:HP}), joint embedding property (Definition \ref{defn:JEP}), and disjoint amalgamation property (Definition \ref{defn:DAP}).
We discuss the motivation and consequences of these properties throughout Section \ref{section:Fraisse}.

\subsection{Discrete time Markov chains}\label{section:discrete time-summary}

Given any probability measure $\mu$ on $\Lip(\XN)$, we construct the {\em standard $\mu$-process} $\mathbf{X}_{\mu}^*:=\{\mathbf{X}^*_{\mathfrak{M},\mu}:\,\mathfrak{M}\in\XN\}$ on $\XN$ by taking $F_1,F_2,\ldots$ independent, identically distributed (i.i.d.)~ from $\mu$ and putting $\mathbf{X}^*_{\mathfrak{M},\mu}:=(\mathfrak{X}^*_m)_{m\in\mathbb{Z}_+}$ with $\mathfrak{X}^*_0=\mathfrak{M}$ and 
\begin{equation}\label{eq:mu-chain}
\mathfrak{X}^*_{m+1}=F_{m+1}(\mathfrak{X}^*_m)=(F_{m+1}\circ F_m\circ\cdots\circ F_1)(\mathfrak{M}),\quad m\geq0.
\end{equation}

\begin{theorem}\label{thm:discrete}
Let $\mathbf{X}$ be a discrete time, exchangeable Markov chain on a Fra\"iss\'e space $\XN$.
Then $\mathbf{X}$ has the Feller property if and only if there exists an exchangeable probability measure $\mu$ on $\Lip(\XN)$ such that $\mathbf{X}\equalinlaw\mathbf{X}^*_{\mu}$, where $\mathbf{X}^*_{\mu}$ is the standard $\mu$-process constructed in \eqref{eq:mu-chain}.
\end{theorem}

If $\XN$ satisfies a stronger model theoretic property called {\em $\odap$-disjoint amalgamation} (Definition \ref{defn:n-DAP}), we can specify $\mu$ in Theorem \ref{thm:discrete} to concentrate on conjugation invariant functions $\XN\to\XN$.
In this case, the measure $\mu$ in Theorem \ref{thm:discrete} admits a more explicit description, which we describe further in Theorem \ref{thm:conjugation-discrete}.

\subsection{Continuous time Markov processes}\label{section:continuous-summary}
Although Theorem \ref{thm:discrete} serves as a precursor to our main theorems about continuous time processes, the continuous time case does not follow directly from the discrete time results above.
More subtle behaviors ensue because of the possibility that countably many small jumps bunch together in arbitrarily small time intervals.

We construct the {\em standard $\Lambda$-process} $\mathbf{X}_{\Lambda}^*:=\{\mathbf{X}_{\mathfrak{M},\Lambda}^*:\,\mathfrak{M}\in\XN\}$ by specifying a measure $\Lambda$ on $\Lip(\XN)$ that satisfies
\begin{equation}\label{eq:regularity-Lambda}
\Lambda(\{\idN\})=0\quad\text{and}\quad\Lambda(\{F\in\Lip(\XN):\,F^{[n]}\neq\idn\})<\infty\text{ for all }n\in\Nb,\end{equation}
where $\text{id}_S$ denotes the identity $\XS\rightarrow\XS$, and letting $\mathbf{\Phi}:=\{(t,F_t)\}\subseteq [0,\infty)\times\Lip(\XN)$ be a Poisson point process with intensity $dt\otimes\Lambda$, where $dt$ denotes Lebesgue measure on $[0,\infty)$.
We then build $\mathbf{X}_{\Lambda}^*$ as the projective limit of its finite state space processes $\mathbf{X}^{*[n]}_{\Lambda}:=\{\mathbf{X}^{*[n]}_{\mathfrak{S},\Lambda}:\,\mathfrak{S}\in\Xn\}$, where each $\mathbf{X}^{*[n]}_{\mathfrak{S},\Lambda}=(\mathfrak{X}^{*[n]}_t)_{t\in[0,\infty)}$ has
\begin{itemize}	
	\item  $\mathfrak{X}^{*[n]}_0=\mathfrak{S}$, 
	\item $\mathfrak{X}_t^{*[n]}=F_t^{[n]}(\mathfrak{X}_{t-}^{*[n]})$, if $t$ is an atom time of $\Phi$ with $F_t^{[n]}\neq\idn$, and
	\item $\mathfrak{X}_t^{*[n]}=\mathfrak{X}_{t-}^{*[n]}$ otherwise,
\end{itemize}
for $\mathfrak{X}^{*[n]}_{t-}:=\lim_{s\uparrow t}\mathfrak{X}^{*[n]}_s$ denoting the state of $\mathbf{X}^{*[n]}_{\mathfrak{S},\Lambda}$ just before time $t$.
For every $n\in\Nb$, $\mathbf{X}^{*[n]}_{\mathfrak{S},\Lambda}$ is clearly Markovian by its construction from $\mathbf{\Phi}$, for which the righthand side of \eqref{eq:regularity-Lambda} implies that only finitely many jumps occur in bounded intervals with probability 1.
By the construction of $\mathbf{X}^{*[n+1]}_{\Lambda}$ and $\mathbf{X}^{*[n]}_{\Lambda}$ from the same Poisson process, it is immediate that the restriction of each $\mathbf{X}^{*[n+1]}_{\mathfrak{S}',\Lambda}$ to $\Xn$ coincides exactly with $\mathbf{X}^{*[n]}_{\mathfrak{S}'|_{[n]},\Lambda}$, so that there is a well defined projective limit process $\mathbf{X}^*_{\Lambda}$ on $\XN$.

\begin{theorem}\label{thm:continuous-Lambda}
Let $\mathbf{X}$ be a continuous time, exchangeable Markov process on a Fra\"iss\'e space $\XN$.
Then $\mathbf{X}$ has the Feller property if and only if there exists an exchangeable measure $\Lambda$ on $\Lip(\XN)$ satisfying \eqref{eq:regularity-Lambda} such that $\mathbf{X}\equalinlaw\mathbf{X}^*_{\Lambda}$.
\end{theorem}

Just as in discrete time, when $\XN$ also satisfies the $\odap$-disjoint amalgamation property, we can choose $\Lambda$ to concentrate on the subspace of conjugation invariant functions on $\XN$.
In this case, Theorem \ref{thm:Levy-Ito} refines Theorem \ref{thm:continuous-Lambda} with our L\'evy--It\^o--Khintchine representation of exchangeable combinatorial Feller processes, which decomposes $\Lambda$ into mutually singular measures indexed by partitions of integers $k=1,\ldots,\max_j\ar(R_j)$.
Theorems \ref{thm:limit process} and \ref{thm:disc} go on to characterize the behavior induced by projecting $\mathbf{X}$ into a suitable space of limiting objects.
These latter outcomes require much more technical notation and definitions, and so we delay their formal statements until later.
The next two sections prepare the key elements of model theory and exchangeable random structures needed to prove our main theorems.

\section{Fra\"iss\'e spaces}\label{section:Fraisse}

Throughout this section, $\mathbf{X}=\{\mathbf{X}_{\mathfrak{M}}: \mathfrak{M}\in\XN\}$ is an exchangeable, projective Markov process on $\XN\subseteq\lN$.

\subsection{Model theoretic properties}\label{section:model theory}

The connections to the Feller property in our main theorems indicate that the constituents of $\{\mathbf{X}_{\mathfrak{M}}:\ \mathfrak{M}\in\XN\}$ fit together to produce a jointly continuous flow $(\mathfrak{M},t)\mapsto\mathfrak{X}_{\mathfrak{M},t}$, where here we write $\mathfrak{X}_{\mathfrak{M},t}$ to denote the state of $\mathbf{X}_{\mathfrak{M}}$ at time $t$.
(We usually suppress the dependence on $\mathfrak{M}$ when no confusion will result.)
As we show, the structure of any such process depends on the structure of the state space, which is most naturally described in terms of model theoretic and first-order logical properties.
Here we distill the main concepts, which can be found in most model theory texts, for example, \cite{HodgesModelTheory,HodgesShorter}.

The projective Markov property facilitates the study of $\mathbf{X}$ through its restrictions to finite substructures.
In this direction, we identify an $\mathcal{L}$-structure $\mathfrak{M}=(\Nb,\mathcal{R}_1,\ldots,\mathcal{R}_r)$ with its collection of finite substructures, called the {\em age}.

\begin{definition}[Age of a structure]\label{defn:age}
The {\em age} of $\mathfrak{M}$, denoted $\age(\mathfrak{M})$, is the set of all finite $\mathcal{L}$-structures embedded in $\mathfrak{M}$, that is,
\[\age(\mathfrak{M}):=\{\mathfrak{S}\in\bigcup_{n\in\Nb}\ln:\,\text{there exists an embedding }\phi:\mathfrak{S}\to\mathfrak{M}\}.\]
\end{definition}

\begin{example}
The partition $\mathbf{0}_{\Nb}:=\{\{1\},\{2\},\ldots\}$ of $\Nb$ into singletons has $\age(\mathbf{0}_{\Nb})=\{\mathbf{0}_{[n]}: n\in\Nb\}$ consisting only of partitions of finite sets into singletons.
The age of a partition $\pi=\{B_1,\ldots,B_k\}$ of $\Nb$ for which each block $B_1,\ldots,B_k$ is infinite consists of all finite partitions with at most $k$ blocks.
The {\em countable universal partition} $\mathfrak{F}_{\Pi}=\{B_1,B_2,\ldots\}$ has infinitely many blocks of infinite size so that every finite partition embeds into $\mathfrak{F}_{\Pi}$ and $\age(\mathfrak{F}_{\Pi})$ contains all finite partitions.
\end{example}

\begin{remark}
The age of a structure is usually defined as the set of all finite substructures of $\mathfrak{M}$, perhaps identified if they are isomorphic.
For our purposes, it is more appropriate to specify the domain of structures with cardinality $n$ to be precisely $[n]$ and to regard $\mathfrak{S},\mathfrak{S}'\in\age(\mathfrak{M})$ as different even when they are isomorphic.

Take, for example, the structure $\mathfrak{M}$ corresponding to the partition of $\{1,2,3\}$ into equivalence classes $\{1,3\}$ and $\{2\}$, written $\mathfrak{M}=\{1,3\}/\{2\}$ for simplicity.
Each of the partitions $\{1,2\}/\{3\}$, $\{1,3\}/\{2\}$, $\{1\}/\{2,3\}$, $\{1,2\}$, $\{1,3\}$, $\{2,3\}$, $\{1\}/\{2\}$, $\{1\}/\{3\}$, $\{2\}/\{3\}$, $\{1\}$, $\{2\}$, and $\{3\}$ embeds into $\mathfrak{M}$.
Typically, the isomorphic structures $\{1,2\}/\{3\}$, $\{1,3\}/\{2\}$, and $\{1\}/\{2,3\}$ would be identified with their equivalence class, but for our purposes these structures are distinct.
Instead, the exchangeability assumption eliminates the need to distinguish among isomorphic structures with different base sets, such as $\{1,2\}$, $\{1,3\}$, and $\{2,3\}$, allowing us to eliminate any structure whose domain is not an initial segment of $\Nb$.
With this, we obtain
\[\age(\{1,3\}/\{2\})=\{\{1,2\}/\{3\}, \{1,3\}/\{2\}, \{1\}/\{2,3\}, \{1,2\}, \{1\}/\{2\}, \{1\}\},\]
which partitions into the possible states of any finite substructure of $\mathfrak{M}$:
\begin{eqnarray*}
\lefteqn{\age(\{1,3\}/\{2\})=}\\
&&=\{\{1,2\}/\{3\}, \{1,3\}/\{2\}, \{1\}/\{2,3\}\}\cup\{\{1,2\},\{1\}/\{2\}\}\cup\{\{1\}\}.\end{eqnarray*}

\end{remark}

We extend our definition of the age of a structure to the entire space $\XN\subseteq\lN$ by putting $\age(\XN):=\bigcup_{\mathfrak{M}\in\XN}\age(\mathfrak{M})$.
With $\age_n(\mathfrak{M}):=\age(\mathfrak{M})\cap\ln$ denoting those elements of $\age(\mathfrak{M})$ of size $n$, and likewise for $\age_n(\XN)$, we immediately see that $\age_n(\XN)=\Xn$ is the state space of the process $\mathbf{X}$ restricted to $\Xn$.
For reasons laid out below, the age of any $\XN$ harboring an exchangeable, projective Markov process satisfies the following natural structural properties.

\begin{definition}[Hereditary property]\label{defn:HP}
A collection of finite structures $K$ has the {\em hereditary property} (HP) if $\mathfrak{S}\in K$ and $\mathfrak{T}\subseteq\mathfrak{S}$ implies $\mathfrak{T}\in K$.
In this case, we say that $K$ is {\em closed under substructures}.
\end{definition}

\begin{definition}[Joint embedding property]\label{defn:JEP}
  A collection of finite structures $K$ has the {\em joint embedding property} (JEP) if, for all $\mathfrak{S},\mathfrak{T}\in K$, there exists $\mathfrak{U}\in K$ such that $\mathfrak{S}$ and $\mathfrak{T}$ both embed into $\mathfrak{U}$.
\end{definition}

\begin{definition}[Disjoint amalgamation property]\label{defn:DAP}
A collection of finite structures $K$ has the {\em disjoint amalgamation property} (DAP)
 if for any $\mathfrak{S},\mathfrak{T},\mathfrak{T}'\in K$ and embeddings $\phi:\mathfrak{S}\rightarrow\mathfrak{T}$ and $\phi':\mathfrak{S}\rightarrow\mathfrak{T}'$ there exist $\mathfrak{U}\in K$ and embeddings $\psi:\mathfrak{T}\rightarrow\mathfrak{U}$ and $\psi':\mathfrak{T}'\rightarrow\mathfrak{U}$ such that $\psi\circ\phi=\psi'\circ\phi'$ and $\im(\psi\circ\phi)=\im(\psi)\cap\im(\psi')$, where $\im(\phi)=\{t\in \dom\mathfrak{T}:\,\exists s\in\dom\mathfrak{S}\,(\phi(s)=t)\}$ is the {\em image} of $\phi$.
\end{definition}


\begin{example}\label{example:graphs-JEP}
For a concrete illustration of Definition \ref{defn:HP}-\ref{defn:DAP}, consider the space $\graphsN$ of countable undirected graphs.
The hereditary property is plain since $\age(\graphsN)$ contains all finite subgraphs.
For the joint embedding property, we can embed any $\mathfrak{G}\in\mathcal{G}_{[m]}$ and $\mathfrak{G}'\in\mathcal{G}_{[m']}$ into a common graph $\mathfrak{G}''\cong \mathfrak{G}\cup \mathfrak{G}'\in\mathcal{G}_{[m+m']}$ by putting $\mathfrak{G}''|_{[m]}=\mathfrak{G}$ and $\mathfrak{G}''|_{[m+m']\setminus[m]}\cong \mathfrak{G}'$ with no edges between the vertex sets $[m]$ and $[m+m']\setminus[m]$.
For disjoint amalgamation, suppose $\mathfrak{G}$ embeds into both $\mathfrak{H}$ and $\mathfrak{H}'$.  Then we can embed both $\mathfrak{H}$ and $\mathfrak{H}'$ into a larger structure $\mathfrak{H}''$ by performing a similar operation as for joint embedding above but with the modification that the pieces of $\mathfrak{H}$ and $\mathfrak{H}'$ corresponding to $\mathfrak{G}$ are the same part of $\mathfrak{H}''$.
\end{example}

Though abstract in appearance, the hereditary, joint embedding, and disjoint amalgamation properties align closely with the exchangeability and projectivity properties of combinatorial Markov processes $\mathbf{X}=\{\mathbf{X}_{\mathfrak{M}}: \mathfrak{M}\in\XN\}$.
By the projectivity property, we can alternatively regard $\mathbf{X}$ as a collection of finite state space Markov processes $\{\mathbf{X}_{\mathfrak{S}}: \mathfrak{S}\in\bigcup_{n\in\Nb}\Xn\}$, akin to the age of a structure defined above, so that $\mathbf{X}_{\mathfrak{S}}^{\dom\mathfrak{S}'}=(\mathfrak{X}_t|_{\dom\mathfrak{S}'})_{t\geq0}\equalinlaw\mathbf{X}_{\mathfrak{S}'}$ for every substructure $\mathfrak{S}'\subseteq\mathfrak{S}$.
The joint embedding property reflects our assumption that no two processes defined on structures with starting states $\mathfrak{S}\in\Xn$ and $\mathfrak{S}'\in\mathcal{X}_{[n']}$, $n,n'\in\Nb$, are mutually exclusive, that is, there are $\mathfrak{T}\in\age(\XN)$ and embeddings $\phi:\mathfrak{S}\to\mathfrak{T}$ and $\phi':\mathfrak{S}'\to\mathfrak{T}$ such that $\mathbf{X}^{\phi}_{\mathfrak{T}}\equalinlaw\mathbf{X}_{\mathfrak{S}}$ and $\mathbf{X}_{\mathfrak{T}}^{\phi'}\equalinlaw\mathbf{X}_{\mathfrak{S}'}$.
The disjoint amalgamation property ensures that any two processes $\mathbf{X}_{\mathfrak{T}}$ and $\mathbf{X}_{\mathfrak{T}'}$ for which there are embeddings $\phi:\mathfrak{S}\to\mathfrak{T}$ and $\phi':\mathfrak{S}\to\mathfrak{T}'$ can both be embedded into a process $\mathbf{X}_{\mathfrak{U}}$ by $\psi:\mathfrak{T}\to\mathfrak{U}$ and $\psi':\mathfrak{T}'\to\mathfrak{U}$ in such a way that the behaviors of $\mathbf{X}_{\mathfrak{U}}^{\psi}\equalinlaw\mathbf{X}_{\mathfrak{T}}$ and $\mathbf{X}_{\mathfrak{U}}^{\psi'}\equalinlaw\mathbf{X}_{\mathfrak{T}'}$ are coupled so that $\mathbf{X}_{\mathfrak{U}}^{\psi\circ\phi}=\mathbf{X}_{\mathfrak{U}}^{\psi'\circ\phi'}$ but the conditional distribution of $\mathbf{X}_{\mathfrak{U}}^{\psi}$ given $\mathbf{X}_{\mathfrak{U}}^{\psi'}$ depends only on $\mathbf{X}_{\mathfrak{U}}^{\psi\circ\phi}$.

A key outcome is that collections of structures satisfying HP, JEP, and DAP can be embedded into a single structure with a nice homogeneity property.

\begin{definition}[Ultrahomogeneity]\label{defn:ultrahomogeneity}
An $\mathcal{L}$-structure $\mathfrak{M}$ is \emph{ultrahomogeneous} if every embedding $\phi:\mathfrak{M}|_S\rightarrow\mathfrak{M}$, with $S\subseteq \dom\mathfrak{M}$ finite, extends to an automorphism $\overline{\phi}:\mathfrak{M}\rightarrow\mathfrak{M}$.
\end{definition}

\begin{theorem}[Fra\"iss\'e's theorem, \cite{HodgesShorter}, Theorem 6.1.2]\label{thm:Fraisse}
Let $\mathcal{L}$ be a signature and let $K$ be a nonempty collection of finite $\mathcal{L}$-structures that has HP, JEP, and DAP.
Then there exists an $\mathcal{L}$-structure $\mathfrak{F}_K$ that is ultrahomogeneous, unique up to isomorphism, and has $\age(\mathfrak{F}_K)=K$.
Furthermore, if $\mathfrak{N}$ is any countable structure with $\age(\mathfrak{N})\subseteq K$, then there is an embedding of $\mathfrak{N}$ into $\mathfrak{F}_K$.
If $K$ is countable, then $\dom\mathfrak{F}_K$ is countable and can be taken as $\Nb$.
\end{theorem}

Theorem \ref{thm:Fraisse} holds most generally by replacing DAP with the weaker {\em amalgamation property} (AP), which drops the requirement that $\im(\psi\circ\phi)=\im(\psi)\cap\im(\psi')$ from Definition \ref{defn:DAP}.  
In the model theory literature, any set $K$ of finite structures which is closed under isomorphism and satisfies HP, JEP, and AP is called a \emph{Fra\"iss\'e class} and the unique (up to isomorphism) structure $\mathfrak{F}_K$ from Theorem \ref{thm:Fraisse} is called the {\em Fra\"iss\'e limit of $K$}.  

Within our context of exchangeable, projective Markov processes, the critical distinction between AP and DAP is that DAP implies that any two structures $\mathfrak{T},\mathfrak{T}'\in K$ can be embedded into a larger structure $\mathfrak{U}\in K$ without identifying any elements that are not already identified, while AP allows for the possibility that additional elements of $\mathfrak{T}$ and $\mathfrak{T}'$ must be identified when embedding into $\mathfrak{U}$.
A space $\XN$ with AP but not DAP need not have sufficient structure to house an exchangeable, projective Markov process, as the next example shows.

\begin{example}\label{ex:DAP-fail}
As an example of a class of objects for which AP holds but DAP fails, let $\mathcal{L}=\{R\}$ have $\ar(R)=1$ and let $K$ consist of all structures $([n],\mathcal{R})$, $n\geq1$, such that $\mathcal{R}\subseteq[n]$ is either a singleton or empty.
The collection $K$ is the age of all countable $\mathcal{L}$-structures $\mathfrak{M}$ for which $R^{\mathfrak{M}}\subseteq\Nb$ is either a singleton or is empty.

Let $l>m>n\geq1$, $\mathfrak{S}=([n],\emptyset)$, $\mathfrak{T}=([m],\{j\})$, and $\mathfrak{T}'=([l],\{j'\})$ for $j\in[m]$ and $j'\in[l]$.
Any $\mathfrak{U}$ satisfying the conditions of Definition \ref{defn:DAP} must have $\mathfrak{T}$ and $\mathfrak{T}'$ as embedded substructures; thus, for some $k\geq l$ we must have $\mathfrak{U}=([k],\{j''\})$, $j''\in[k]$.
In this case, $\psi:\mathfrak{T}\to\mathfrak{U}$ and $\psi':\mathfrak{T}'\to\mathfrak{U}$ must have $\psi(j)=j''$ and $\psi'(j')=j''$.
Any embeddings $\phi:\mathfrak{S}\to\mathfrak{T}$ and $\phi':\mathfrak{S}\to\mathfrak{T}'$ must have $\im(\phi)\cap\{j\}=\emptyset=\im(\phi')\cap\{j'\}$, but $\im(\psi\circ\phi)\cap\{j''\}=\emptyset$ and $\im(\psi)\cap\im(\psi')\supseteq\{j''\}$ prevents $\im(\psi\circ\phi)=\im(\psi)\cap\im(\psi')$.
Thus, $K$ satisfies AP but violates DAP.
By Theorem \ref{thm:Fraisse}, there is a unique (up to isomorphism) $\mathcal{L}$-structure $\mathfrak{F}$ into which every $\mathfrak{S}\in K$ embeds.
Any such $\mathfrak{F}$ is isomorphic to $(\Nb,\{1\})$.

The collection $\XN$ containing only $(\Nb,\emptyset)$ and structures of the form $(\Nb,\{i\})$, $i\in\Nb$, is a pathological state space for exchangeable, projective Markov processes.
By exchangeability of the transition probabilities \eqref{eq:exch-tps} and countable additivity of probabilities, the one step transition probability of any such Markov process must satisfy
\[P(\mathfrak{M},A)=\left\{\begin{array}{cc} p,& \mathfrak{M}\in A,\ (\Nb,\emptyset)\not\in A,\\ 1-p, & \mathfrak{M}\not\in A,\ (\Nb,\emptyset)\in A, \\ 1, & \mathfrak{M}\in A,\ (\Nb,\emptyset)\in A,\\ 0,& \text{otherwise},\end{array}\right.\quad \mathfrak{M}\in\XN,\quad A\Borel\XN,\]
for some $0\leq p\leq 1$,
meaning that every $\mathbf{X}_{\mathfrak{M}}=(\mathfrak{X}_m)_{m\in\mathbb{Z}_+}$ remains in its initial state $\mathfrak{X}_0=\mathfrak{M}$ for a Geometrically distributed number of steps until it is absorbed in the empty structure $(\Nb,\emptyset)$.
Moreover, there exists no exchangeable measure $\nu$ with full support on $\XN$ such that we can generate a fully exchangeable process $\mathbf{X}_{\nu}=(\mathfrak{X}_m)_{m\in\mathbb{Z}_+}$ by first drawing $\mathfrak{X}_0\sim\nu$ and then putting $\mathbf{X}_{\nu}=\mathbf{X}_{\mathfrak{M}}$ on the event $\mathfrak{X}_0=\mathfrak{M}$.
\end{example}

Example \ref{ex:DAP-fail} reveals a pathology in spaces that satisfy HP, JEP, and AP but not DAP.
Processes evolving on such spaces may exhibit trivial behaviors and present technical nuisances.
We rule out both possibilities by restricting our attention to Markov processes on spaces of countable structures that satisfy DAP in addition to HP and JEP.
We call any such space a {\em Fra\"{i}ss\'e space} to distinguish from the notion of a Fra\"{i}ss\'e class, which need only satisfy AP.
We sometimes refer to a Fra\"{i}ss\'e space as a {\em Fra\"iss\'e class with DAP} to avoid confusion for readers familiar with the customary definition.

\begin{definition}[Fra\"iss\'e space]\label{defn:Fraisse space}
We call $\XN\subseteq\lN$ a {\em Fra\"iss\'e space} if $\age(\XN)$ is a Fra\"{i}ss\'e class with DAP, that is, $\age(\XN)$ satisfies HP, JEP, and DAP.
\end{definition}

Definition \ref{defn:Fraisse space} captures the key properties satisfied by combinatorial state spaces of practical interest while ruling out pathological cases, as in Example \ref{ex:DAP-fail}.
As an important corollary, we observe that Fra\"{i}ss\'e spaces are exactly those spaces whose Fra\"{i}ss\'e limit can be randomly generated by an exchangeable measure.

\begin{theorem}[Ackerman, Freer \& Patel \cite{AFP}, Corollary 1.3]\label{thm:AFP}
Let $\mathfrak{F}$ be a Fra\"iss\'e limit of some Fra\"iss\'e class with DAP.
Then there exists an exchangeable probability measure $\mu$ on $\lN$ such that $\mathfrak{N}\cong\mathfrak{F}$ for $\mu$-almost every $\mathfrak{N}\in\lN$.
\end{theorem}

\begin{example}\label{ex:universal-set}
For an easy illustration of Theorem \ref{thm:AFP}, let $\mathcal{L}=\{R\}$ have $\ar(R)=1$ and let $\XN=\lN$ correspond to the set of all subsets of $\Nb$.
An exchangeable Fra\"iss\'e limit $\mathfrak{F}=(\Nb,R^{\mathfrak{F}})$ can be constructed by putting $R^{\mathfrak{F}}(n)=\xi_n$ for each $n\in\Nb$, where $\xi_1,\xi_2,\ldots$ are i.i.d.\ Bernoulli random variables with $\mathbb{P}\{\xi_n=1\}=1/2$, $n\in\Nb$.
\end{example}

\begin{example}\label{example:Rado}
Let $\graphsN$ be the space of countable undirected graphs.
By Example \ref{example:graphs-JEP}, $\age(\graphsN)$ satisfies HP, JEP, and DAP and, by Theorem \ref{thm:Fraisse}, there exists a unique (up to isomorphism) ultrahomogeneous graph whose age coincides with $\age(\graphsN)$.
This countable universal graph is a quintessential example of a Fra\"iss\'e limit and is called the {\em Rado graph} \cite{Rado1964}, denoted $\mathfrak{F}_{\Gamma}$.
Theorem \ref{thm:AFP} guarantees us the ability to generate an exchangeable copy of the Rado graph.

In the case of undirected graphs, the family of Erd\H{o}s--R\'enyi measures with parameter $p\in(0,1)$ produces an exchangeable copy of the Rado graph with probability 1.
To see this, take $p=1/2$ so that every edge is present independently with probability $1/2$.
The Erd\H{o}s--R\'enyi measure assigns probability $2^{-\binom{n}{2}}>0$ to every $\mathfrak{S}\in\graphsn$, $n\geq1$.
The Borel--Cantelli lemma implies that every $\mathfrak{S}$ embeds into the random graph with probability 1.
Since there are countably many finite graphs, this holds for all $\mathfrak{S}\in\bigcup_{n\in\Nb}\graphsn$ and the Erd\H{o}s--R\'enyi graph is isomorphic to the Rado graph with probability 1.
\end{example}

Beyond classifying the behavior of all exchangeable, projective Markov processes defined on Fra\"{i}ss\'e spaces, our discussion focuses on fundamental differences between processes that evolve on certain spaces. 
Comparing processes on the spaces $\partitionsN$ of countable partitions and $\graphsN$ of countable graphs provides an illuminating example of this distinction.
Both $\age(\partitionsN)$ and $\age(\graphsN)$ satisfy HP, JEP, and DAP.
By Theorem \ref{thm:Fraisse}, $\partitionsN$ and $\graphsN$ can be represented by structures $\mathfrak{F}_{\Pi}$ and $\mathfrak{F}_{\Gamma}$, respectively, into which every element of the state space embeds, allowing us to study the transition behavior of Feller processes on these spaces by considering the transition behavior out of their respective Fra\"iss\'e limits.
For reasons laid bare in Section \ref{section:relative exchangeability}, the space of countable graphs $\graphsN$ satisfies a stronger amalgamation property than $\partitionsN$, permitting a more precise description of processes on these spaces.

\begin{definition}[$n$-DAP]\label{defn:n-DAP}
Let $K$ be a collection of finite structures that is closed under isomorphism.
For $n\geq1$, we say that $K$ satisfies the \emph{$n$-disjoint amaglamation property} ($n$-DAP) if for every collection $(\mathfrak{S}_i)_{1\leq i\leq n}$ of structures with $\mathfrak{S}_i\in K$, $\dom\mathfrak{S}_i=[n]\setminus\{i\}$ for each $1\leq i\leq n$, and $\mathfrak{S}_i|_{[n]\setminus\{i,j\}}=\mathfrak{S}_j|_{[n]\setminus\{i,j\}}$ for all $1\leq i,j\leq n$ there exists $\mathfrak{S}\in K$ with $\dom\mathfrak{S}=[n]$ such that $\mathfrak{S}|_{[n]\setminus\{i\}}=\mathfrak{S}_i$ for every $1\leq i\leq n$.

We say that $K$ satisfies \emph{$\odap$-DAP} if it satisfies $n$-DAP for all $n\geq1$.
\end{definition}

\begin{remark}
Note that 2-DAP and DAP are identical conditions.
\end{remark}
Under $n$-disjoint amalgamation, if we specify structures on each proper subset of $[n]$ in a way that is pairwise compatible, then we can unify these structures into a single structure on all of $[n]$. 
(Note that DAP only implies that $(\mathfrak{S}_i)_{1\leq i\leq n}$ can be unified into some large enough structure.)
 By slight abuse of terminology, if $K$ is a collection of finite structures not closed under isomorphism, then we say that $K$ has $\odap$-DAP if its closure under isomorphism has $\odap$-DAP.

We realize the probabilistic significance of $\odap$-DAP by considering the conditional distribution of a random substructure $\mathfrak{X}_{t+s}|_S$ given $\mathfrak{X}_t$ in an exchangeable combinatorial Feller process $(\mathfrak{X}_t)_{t\in T}$.
Under $\odap$-DAP and exchangeability, the conditional distribution of $\mathfrak{X}_{t+s}|_S$ given $\mathfrak{X}_t$ depends only on $\mathfrak{X}_t|_S$, that is, the dependence is localized to the corresponding substructure of $\mathfrak{X}_t$.
In the absence of $\odap$-DAP, the projective Markov property only implies that the conditional distribution of $\mathfrak{X}_{t+s}|_S$ given $\mathfrak{X}_t$ depends on $\mathfrak{X}_t|_{[\max S]}$.
This distinction is captured in Theorems \ref{thm:CTasym} and \ref{thm:CT2}.

Many of the most interesting applications of our results are to symmetric structures, including partitions and undirected graphs.
Restricting to this case simplifies some of the statements.

\begin{definition}[Symmetric structures]
An $\mathcal{L}$-structure $\mathfrak{M}=(S,R_1^{\mathfrak{M}},\ldots,R_r^{\mathfrak{M}})$ is \emph{symmetric} if each of its relations $R_j^{\mathfrak{M}}$ is symmetric, that is, 
\[ R_j^{\mathfrak{M}}(x_1,\ldots,x_{\ar(R_j)})=R_j^{\mathfrak{M}}(x_{\sigma(1)},\ldots,x_{\sigma(\ar(R_j))})\]
for all permutations $\sigma:[\ar(R_j)]\to[\ar(R_j)]$, for every $j=1,\ldots,r$.
A combinatorial state space $\mathcal{X}_S$ is \emph{symmetric} if every $\mathfrak{M}\in\mathcal{X}_S$ is symmetric.
\end{definition}


\begin{lemma}
  The Fra\"iss\'e limit of a Fra\"iss\'e class $K$ is symmetric if and only if every $\mathfrak{S}\in K$ is symmetric.
\end{lemma}

\subsection{Examples}

Following are several examples describing familiar Fra\"iss\'e spaces and their Fra\"{i}ss\'e limits.

\begin{example}[Colorings]\label{ex:coloring}
Let $\mathcal{L}=\{R, B\}$ have $\ar(R)=\ar(B)=1$.  
Here we may regard an $\mathcal{L}$-structure as a $4$-coloring, where the four colors are represented by the four possible combinations of $R$ and $B$: for each $i\in\Nb$ and $\mathfrak{M}\in\mathcal{L}_{\Nb}$, we have $(R^{\mathfrak{M}}(i),B^{\mathfrak{M}}(i))\in\{(0,0),(0,1),(1,0),(1,1)\}$.

If $K$ is the collection of all finite $\mathcal{L}$-structures, its Fra\"iss\'e limit is simply the countably infinite structure in which there are infinitely many elements of each color.  This example is symmetric and has $\odap$-DAP.
Such a Fra\"{i}ss\'e limit can be generated exchangeably by choosing among each of the four combinations with equal probability, independently for each $i\in\Nb$.

More generally, one can represent  a $k$-coloring for any fixed $k\geq1$ by taking $\mathcal{L}=\{R_1,\ldots,R_k\}$ with $\ar(R_i)=1$, $1\leq i\leq k$, and requiring of each $\mathfrak{M}$ that $x\in R_i^{\mathfrak{M}}$ holds of exactly one $i=1,\ldots,k$, for every $x\in\dom\mathfrak{M}$.
In this case, generating an exchangeable version of the Fra\"{i}ss\'e limit requires coordination across relations.
\end{example}

\begin{example}[Graphs]\label{ex:graph}
Let $\mathcal{L}=\{R\}$ have $\ar(R)=2$ and let $K_d$ be the collection of all finite irreflexive $\mathcal{L}$-structures, which may be regarded as the collection of finite directed graphs.
For $\mathfrak{S}\in K_d$ and $a,b\in \dom\mathfrak{S}$,
\[(a,b)\in R^{\mathfrak{S}}\quad\text{if and only if}\quad \mathfrak{S}\text{ has a directed edge from }a\text{ to }b.\]
The Fra\"iss\'e limit of $K_d$ is the universal directed graph, which has a natural random representation by taking a countable set and letting each directed edge be present independently with probability $1/2$ (or any probability $p\in(0,1)$).

A natural subset of $K_d$ is $K_g$, the collection of finite irreflexive symmetric $\mathcal{L}$-structures corresponding to the set of finite \emph{undirected} graphs.  The Fra\"iss\'e limit is the Rado graph discussed in Example \ref{example:Rado} above.

Another subset of $K_d$ is the collection $K_t$ of finite {\em tournaments}, that is, finite $\mathcal{L}$-structures $\mathfrak{S}$ in which exactly one of the pairs $(i,j)$ and $(j,i)$ is present in $R^{\mathfrak{S}}$ for each pair $i,j$ with $i\neq j$.  The Fra\"iss\'e limit of $K_t$ is the universal tournament, which can be randomly generated by independently selecting an orientation $i\rightarrow j$ or $j\rightarrow i$ for each pair $i,j$.

Clearly $K_g$ is symmetric but $K_d$ and $K_t$ are not.  All three have $\odap$-DAP.
\end{example}

\begin{example}[Partitions]\label{ex:partition}
Let $\mathcal{L}=\{R\}$ have $\ar(R)=2$.  Let $K_p$ be the collection of all finite structures $\mathfrak{S}$ in which $R^{\mathfrak{S}}$ is an equivalence relation, that is, a reflexive, symmetric, and transitive relation.  
The Fra\"iss\'e limit is a structure with infinitely many equivalence classes of infinite size.

A subset of $K_p$ is $K_{p,2}$, the collection of all structures $\mathfrak{S}$ in which $R^{\mathfrak{S}}$ is an equivalence relation with at most 2 equivalence classes.  The Fra\"iss\'e limit is an infinite structure with two infinite equivalence classes.

Both $K_p$ and $K_{p,2}$ are symmetric but do not have $3$-DAP.
For example, take $\mathfrak{S}_1$ to be the structure with equivalence classes $\{2\}$ and $\{3\}$, $\mathfrak{S}_2$ to be the structure with equivalence class $\{1,3\}$, and $\mathfrak{S}_3$ to be the structure with equivalence class $\{1,2\}$.
Then the result of taking the union of these three structures fails to be an equivalence relation since $\mathfrak{S}_2$ and $\mathfrak{S}_3$ require equivalence classes $\{1,3\}$ and $\{1,2\}$, respectively, and, thus, the transitivity axiom requires the class $\{2,3\}$ while $\mathfrak{S}_1$ requires classes $\{2\}$ and $\{3\}$, a contradiction.
\end{example}

\begin{example}[Hypergraphs]\label{ex:hypergraph}
  Let $\mathcal{L}=\{S\}$ have $\ar(S)=n$.  Let $K$ be the collection of all finite irreflexive symmetric $\mathcal{L}$-structures, that is, structures $\mathfrak{S}$ for which $(a_1,\ldots,a_n)\in S^{\mathfrak{S}}$ implies that $a_1,\ldots,a_n$ are pairwise distinct and $(a_{\sigma(1)},\ldots,a_{\sigma(n)})\in S^{\mathfrak{S}}$ for every permutation $\sigma:[n]\rightarrow[n]$.  Then $K$ is the collection of finite $n$-ary undirected hypergraphs, which is symmetric and has $\odap$-DAP.  (Undirectedness is ensured by the closure under permutations and $n$-arity by the fact that $a_1,\ldots,a_n$ must be pairwise distinct.)

The Fra\"iss\'e limit is the natural generalization of the Rado graph to a hypergraph, a countable $n$-ary hypergraph in which every finite hypergraph appears, which can be produced randomly by choosing a hypergraph on $\mathbb{N}$ for which every hyperedge is present independently with probabiltiy $1/2$.

\end{example}

\begin{example}[Partitions on pairs]\label{ex:pairpartition}
Let $\mathcal{L}=\{S\}$ have $\ar(S)=4$.  Let $K$ be the collection of all finite structures $\mathfrak{S}$ such that $S^{\mathfrak{S}}$ is an equivalence relation on pairs from $\dom\mathfrak{S}$, that is,
\begin{itemize}
\item for all $i,j\in\dom\mathfrak{S}$, $(i,j,i,j)\in S^{\mathfrak{S}}$,
\item for all $i,j,i',j'\in\dom\mathfrak{S}$, if $(i,j,i',j')\in S^{\mathfrak{S}}$ then $(i',j',i,j)\in S^{\mathfrak{S}}$, and
\item for all $i,j,i',j',i'',j''\in\dom\mathfrak{S}$, if $(i,j,i',j')\in S^{\mathfrak{S}}$ and $(i',j',i'',j'')\in S^{\mathfrak{S}}$ then $(i,j,i'',j'')\in S^{\mathfrak{S}}$.
\end{itemize}

The Fra\"iss\'e limit is an infinite structure with an equivalence relation on pairs with infinitely many infinite equivalence relations and some additional universality properties.
For instance, for any finite list of equivalence classes $C_1,\ldots,C_d$, there is an $i$ and values $j_1,\ldots,j_d$ such that $(i,j_k)\in C_k$ for each $k\leq d$.

This class fails to have $n$-DAP for $n=3,4,5$.
\end{example}

\begin{example}[Colored graphs]\label{ex:colorgraph}
Let $\mathcal{L}=\{P,R\}$ have $\ar(P)=1$ and $\ar(R)=2$.  Let $K$ be the collection of all finite structures $\mathfrak{S}$ such that $R^{\mathfrak{S}}$ is irreflexive and symmetric, that is, $K$ consists of graphs with a 2-coloring on the vertices.  The Fra\"iss\'e limit is the Rado graph together with a 2-coloring of the vertices such that every possible colored finite graph appears as a subgraph.  This can be generated by labeling each of the vertices in an Erd\H{o}s--R\'enyi graph independently and uniformly in $\{0,1\}$.
  This example is symmetric and has $\odap$-DAP.
\end{example}

\begin{example}[Orderings]
Let $\mathcal{L}=\{R\}$ have $\ar(R)=2$.
Let $K_o$ be the collection of all finite structures $\mathfrak{S}$ such that $R^{\mathfrak{S}}$ is a total ordering of $\dom\mathfrak{S}$, that is, every $\mathfrak{S}=([n],R^{\mathfrak{S}})\in K$ satisfies, for all $i,j\in\Nb$, $i\neq j$, either $(i,j)\in R^{\mathfrak{S}}$ or $(j,i)\in R^{\mathfrak{S}}$ but not both.
The Fra\"{i}ss\'e limit $\mathfrak{F}$ is the dense total ordering of $\Nb$, which can be generated exchangeably by taking $\xi_1,\xi_2,\ldots$ i.i.d.\ Uniform$[0,1]$ and putting $(i,j)\in R^{\mathfrak{F}}$ if and only if $\xi_i<\xi_j$, for all $i\neq j$.
Here, $K_o$ is asymmetric and fails to have $3$-DAP.
\end{example}

Further elaborations include all manner of combinations of the above examples, including edge colored graphs and hypergraphs, graphs with vertex partitions, hypergraphs with edge partitions, and so on.

\subsection{Equivalence of Feller and projective Markov properties}

\begin{theorem}\label{prop:Feller}
Let $\mathbf{X}=\{\mathbf{X}_{\mathfrak{M}}:\,\mathfrak{M}\in\XN\}$ be an exchangeable Markov process on a Fra\"iss\'e space $\XN$.
Then the following are equivalent.
\begin{itemize}
	\item[(i)] $\mathbf{X}$ has the projective Markov property.
	\item[(ii)] $\mathbf{X}$ has the Feller property.
\end{itemize}
\end{theorem}

\begin{proof}Below we write $\XN\subseteq\lN$ to denote the state space of an exchangeable Markov process $\Xbf$ and $\Xn:=\{\mathfrak{M}|_{[n]}:\,\mathfrak{M}\in\XN\}\subseteq\ln$ to denote the state space of $\mathbf{X}^{[n]}$, the restriction of $\mathbf{X}$ to a process on $[n]$-labeled structures, for each $n\in\Nb$.
\begin{proof}[(i) $\Rightarrow$ (ii)]
Consider the set
\[\mathcal{C}(\XN):=\{g:\XN\rightarrow\mathbb{R}:\, \exists n\in\mathbb{N}\text{ such that }\mathfrak{M}|_{[n]}=\mathfrak{M}'|_{[n]}\Longrightarrow g(\mathfrak{M})=g(\mathfrak{M}')\}.\]
By the Stone--Weierstrass theorem for compact Hausdorff spaces, $\mathcal{C}(\XN)$ is dense in the space of bounded, continuous functions, and it suffices to prove the Feller property for $g\in\mathcal{C}(\XN)$.

Take any $g\in\mathcal{C}(\XN)$ and let $n\in\mathbb{N}$ be such that $g(\mathfrak{M})$ only depends on $\mathfrak{M}|_{[n]}$.
We can, therefore, specify $g':\Xn\to\mathbb{R}$ so that $g(\mathfrak{M})=g'(\mathfrak{M}|_{[n]})$ for every $\mathfrak{M}\in\XN$.
Under the topology induced by \eqref{eq:ultrametric}, all elements of $\Xn$ are isolated, making $g'$ continuous by default.

Suppose $\mathfrak{X}_0=\mathfrak{M}$.
By projectivity of $\mathbf{X}$, $\mathbf{X}^{[n]}_{\mathfrak{M}}$ has the Markov property with initial state $\mathfrak{M}|_{[n]}$.
As $\Xn$ is a finite state space, $\Xbf^{[n]}_{\mathfrak{M}}$ must have a strictly positive hold time in its initial state with probability 1.
In particular, $\mathfrak{X}_t|_{[n]}\rightarrow_P \mathfrak{X}_0|_{[n]}=\mathfrak{M}|_{[n]}$ as $t\downarrow0$, where $\rightarrow_P$ denotes {\em convergence in probability}.
By the projective Markov property, continuity of $g$, and the bounded convergence theorem, we observe
\begin{eqnarray*}
\lim_{t\downarrow0}\mathbf{P}_tg(\mathfrak{M})&=&\lim_{t\downarrow0}E(g(\mathfrak{X}_t)\mid \mathfrak{X}_0=\mathfrak{M})\\
&=&E(\lim_{t\downarrow0}g'(\mathfrak{X}_t|_{[n]})\mid \mathfrak{X}_0|_{[n]}=\mathfrak{M}|_{[n]})\\
&=&E(g'(\lim_{t\downarrow0} \mathfrak{X}_t|_{[n]})\mid \mathfrak{X}_0|_{[n]}=\mathfrak{M}|_{[n]})\\
&=&g'(\mathfrak{M}|_{[n]})\\
&=&g(\mathfrak{M}),\end{eqnarray*}
establishing the first part of the Feller property.

For the second part, we must show that $\mathfrak{M}\mapsto\mathbf{P}_tg(\mathfrak{M})$ is continuous for all $t>0$.
Let $g\in\mathcal{C}(\XN)$ and $(\mathfrak{M}_n)_{n\in\Nb}$ be a sequence in $\XN$ that converges to $\mathfrak{M}\in\XN$ under the product discrete topology.
Then, for every $k\in\mathbb{N}$, there is an $N_k\in\mathbb{N}$ such that $\mathfrak{M}_n|_{[k]}=\mathfrak{M}|_{[k]}$ for all $n\geq N_k$.
In particular, we can choose $k$ so that $g(\mathfrak{M}_n)=g(\mathfrak{M})$ for all $n\geq N_k$.
Continuity follows by continuity of $g$ and the bounded convergence theorem.
\end{proof}

\begin{proof}[(ii) $\Rightarrow$ (i)]

 For finite $S\subset\Nb$ and $\mathfrak{S}\in\XS$, we define $\psi_{\mathfrak{S}}:\XN\rightarrow\{0,1\}$ by
\[\psi_{\mathfrak{S}}(\mathfrak{M}):=\mathbf{1}\{\mathfrak{M}|_{S}:=\mathfrak{S}\}=\left\{\begin{array}{cc} 1,& \mathfrak{M}|_{S}=\mathfrak{S},\\ 0,& \text{otherwise},\end{array}\right.\quad \mathfrak{M}\in\XN,\]
which is bounded and continuous.
To establish the projective property, we must prove that
\[\mathbb{P}\{\mathfrak{X}_t|_{S}=\mathfrak{S}\mid \mathfrak{X}_0=\mathfrak{M}\}=\mathbb{P}\{\mathfrak{X}_t|_{S}=\mathfrak{S}\mid \mathfrak{X}_0|_{S}=\mathfrak{M}|_{S}\}\]
for every $t\geq0$, $\mathfrak{S}\in\XS$, $\mathfrak{M}\in\XN$, and $S\subset\Nb$, which amounts to showing that 
\begin{equation}\label{eq:cond}\mathbf{P}_t\psi_{\mathfrak{S}}(\mathfrak{M})=\mathbf{P}_t\psi_{\mathfrak{S}}(\mathfrak{M}')\end{equation}
 for all $\mathfrak{M},\mathfrak{M}'\in\XN$ for which $\mathfrak{M}|_{S}=\mathfrak{M}'|_{S}$.

By the Feller property, it is enough to establish \eqref{eq:cond} on a dense subset of $\{\mathfrak{M}'\in\XN:\ \mathfrak{M}'|_{S}=\mathfrak{M}|_{S}\}$.
Exchangeability of $\mathbf{X}$ implies $\mathbf{P}_t\psi_{\mathfrak{S}}(\mathfrak{M})=\mathbf{P}_t\psi_{\mathfrak{S}}(\mathfrak{M}^{\sigma})$ for every $\sigma:\Nb\rightarrow\Nb$ that coincides with the identity $S\rightarrow S$.
To obtain \eqref{eq:cond}, we let $\mathfrak{F}$ be any Fra\"iss\'e limit of $\XN$ with $\mathfrak{F}|_{S}=\mathfrak{M}|_{S}$, as given by Theorem \ref{thm:Fraisse} and our assumption that $\XN$ is a Fra\"{i}ss\'e space (Definition \ref{defn:Fraisse space}).
By ultrahomogeneity and universality of $\mathfrak{F}$, 
\[\{\mathfrak{F}^{\sigma}: \sigma:\Nb\to\Nb\text{ coincides with the identity }\dom\mathfrak{S}\to\dom\mathfrak{S}\}\]
is dense in $\{\mathfrak{M}'\in\XN: \mathfrak{M}'|_{S}=\mathfrak{M}|_{S}\}$.
Since $\mathbf{P}_t\psi_{\mathfrak{S}}(\mathfrak{M})$ is invariant with respect to permutations that fix $S$, we have \eqref{eq:cond} on a dense subset of $\{\mathfrak{M}'\in\XN: \mathfrak{M}'|_{S}=\mathfrak{M}|_{S}\}$.
Continuity of $\mathbf{P}_t$ implies that \eqref{eq:cond} must hold on the closure of this set and, therefore, $\mathbf{P}_t\psi_{\mathfrak{S}}(\mathfrak{M})$ depends only on $\mathfrak{M}|_{S}$.

By the second part of the Feller property, we must have $\mathbf{P}_t \psi_{\mathfrak{S}}(\mathfrak{M})\to\psi_{\mathfrak{S}}(\mathfrak{M})$ as $t\downarrow0$, proving that $\mathbf{X}_{\mathfrak{M}}$ has c\`adl\`ag sample paths and the finite restrictions $\mathbf{X}^{S}_{\mathfrak{M}}$ are consistent for every finite subset $S\subset\Nb$.
\end{proof}

\end{proof}

\section{Relative exchangeability}\label{section:relative exchangeability}

In \cite{CraneTowsner2015}, we refined the notion of exchangeability to structures whose distributions are invariant with respect to the symmetries of another structure; see also \cite{Ackerman2015} for related work in a slightly different setting.
Vis-\`a-vis Definition \ref{defn:exchangeable CMP}, the following notion of relative exchangeability occurs naturally in our study of exchangeable combinatorial Feller processes.

\begin{definition}[Relative exchangeability]\label{defn:rel-exch}
Let $\mathcal{L},\mathcal{L}'$ be signatures with $\mathfrak{M}\in\lN$.
A countable random $\mathcal{L}'$-structure $\mathfrak{X}$ is {\em relatively exchangeable with respect to $\mathfrak{M}$}, alternatively {\em $\mathfrak{M}$-exchangeable} or {\em exchangeable relative to $\mathfrak{M}$}, if $\mathfrak{X}|_T^{\phi}\equalinlaw\mathfrak{X}|_S$ for every injection $\phi:S\to T$ such that $\mathfrak{M}|_T^{\phi}=\mathfrak{M}|_S$.
\end{definition}

The projective Markov and exchangeability properties of any exchangeable combinatorial Feller process $\mathbf{X}=\{\mathbf{X}_{\mathfrak{M}}: \mathfrak{M}\in\XN\}$ imply that $\mathfrak{X}_{s+u}$ is relatively exchangeable with respect to $\mathfrak{X}_s$ for all $s,u\geq0$, for every $\mathbf{X}_{\mathfrak{M}}=(\mathfrak{X}_t)_{t\geq0}$.

\begin{prop}
Let $\mathbf{X}$ be an exchangeable Feller process on $\XN$.
For every $\mathfrak{M}\in\XN$, the conditional distribution of $\mathfrak{X}_{s+u}$ in $\mathbf{X}_{\mathfrak{M}}=(\mathfrak{X}_t)_{t\geq0}$, given $\mathfrak{X}_s=\mathfrak{N}$, is relatively exchangeable with respect to $\mathfrak{N}$ for all $s,u\geq0$.
\end{prop}

\begin{proof}
For $s,u\geq0$, let $P_u(\mathfrak{N},\cdot)=\mathbb{P}\{\mathfrak{X}_{s+u}\in\cdot\mid \mathfrak{X}_s=\mathfrak{N}\}$ be the $u$-step transition probability measure for $\mathbf{X}$.

By Definitions \ref{defn:exchangeable CMP} and \ref{defn:projective}, it is clear that $\mathfrak{X}\sim P_s(\mathfrak{N},\cdot)$ satisfies $\mathfrak{X}\equalinlaw\mathfrak{X}^{\sigma}$ for all automorphisms $\sigma:\mathfrak{N}\to\mathfrak{N}$.
Furthermore, the projectivity assumption implies that $\mathfrak{X}|_{[n]}$ depends only on $\mathfrak{N}|_{[n]}$.
The combination of these properties implies relative exchangeability of $\mathfrak{X}$ with respect to $\mathfrak{N}$ \cite[Remark 1.4]{CraneTowsner2015}.
In $\mathbf{X}_{\mathfrak{M}}$, the Markov property and \eqref{eq:tps} imply that $\mathfrak{X}_{s+u}$ is relatively exchangeable with respect to $\mathfrak{N}$ on the event $\mathfrak{X}_s=\mathfrak{N}$, for all $s,u\geq0$.
\end{proof}

Our main theorems below require the general representation theorem for relatively exchangeable structures from \cite{CraneTowsner2015}.  We first state a simpler version which holds of symmetric structures and conveys the main idea.

\begin{theorem}[Crane \& Towsner \cite{CraneTowsner2015}]\label{thm:CT}
Let $\mathcal{L}=\{R_1,\ldots,R_r\}$ be a signature and $\XN\subseteq\lN$ be a symmetric Fra\"iss\'e space with Fra\"iss\'e limit $\mathfrak{F}$.
Suppose $\mathfrak{X}$ is relatively exchangeable with respect to $\mathfrak{F}$.
Then there exist Borel measurable functions $f_1,\ldots,f_r$ with range $\{0,1\}$ such that $\mathfrak{X}\equalinlaw\mathfrak{X}^*=(\Nb,\mathcal{X}_1^*,\ldots,\mathcal{X}_r^*)$ with
\begin{equation}\label{eq:X**}
\mathcal{X}^*_j(\vec x)=f_j(\mathfrak{F}|_{[\max\vec x]},(\xi_s)_{s\subseteq\rng\vec x}),\quad\vec x\in\Nb^{\ar(R_j)},
\end{equation}
for $(\xi_s)_{s\subseteq\Nb: |s|\leq\max_j\ar(R_j)}$ i.i.d.\ Uniform$[0,1]$ random variables.
\end{theorem}

\begin{remark}
Note that the first argument of $f_j$ in \eqref{eq:X**} depends on $\mathfrak{F}|_{[\max\vec x]}$, that is, the entire initial segment of $\mathfrak{F}$ that contains all elements of $\rng\vec x$.
Under stronger conditions on $\XN$ we achieve dependence on only $\mathfrak{F}|_{\rng\vec x}$, that is, the smallest substructure containing all elements of $\rng\vec x$.
The latter representation enables a more precise description of combinatorial Markov processes, as we discuss in Section \ref{section:Levy-Ito}.
\end{remark}

The corresponding representation is somewhat more complicated when $\XN$ is asymmetric.
For example, for $\mathcal{L}=\{R\}$ with $\ar(R)=2$ and a random structure $\mathfrak{X}^*=(\Nb,\mathcal{X}^*_1)$, $\mathcal{X}^*_1(1,2)$ and $\mathcal{X}^*_1(2,1)$ may be negatively correlated, so $f_1$ must be able to distinguish $(1,2)$ from $(2,1)$ and also coordinate the outcomes between the two.  
The presence of the random orderings $(\prec_{\vec y})_{\vec y\sqsubseteq\vec x}$ in Theorem \ref{thm:CT2} is needed to coordinate between possibly asymmetric outcomes.  (Recall from Section \ref{section:notation} that $\vec y\sqsubseteq\vec x$ means $\vec y$ is a subsequence of $\vec x$.)

\begin{definition}
    When $s\subseteq\Nb$ is a finite set, a \emph{uniform random ordering of $s$} is an ordering $\prec_{s}$ of $s$ chosen uniformly at random.  Given $\prec_{\rng\vec x}$, we write $\prec_{\vec x}$ for the ordering of $[|\rng\vec x|]$ induced by putting $i\prec_{\vec x}j$ if and only if $x_i\prec_{\rng\vec x}x_j$.  If $x_i=x_j$, then $i\not\prec_{\vec x}j$ and $j\not\prec_{\vec x}i$.
\end{definition}

The details of the upcoming representation theorem are not crucial to anything that follows.
We only need the existence of a suitable representation in order to ensure a construction of the random transition operators in our main theorems.
We state the theorem below and leave the details to \cite{CraneTowsner2015}.
\begin{theorem}[Crane \& Towsner \cite{CraneTowsner2015}]\label{thm:CTasym}
Let $\mathcal{L}=\{R_1,\ldots,R_r\}$ be a signature and $\XN\subseteq\lN$ be a Fra\"iss\'e space with Fra\"iss\'e limit $\mathfrak{F}$.
Suppose $\mathfrak{X}$ is relatively exchangeable with respect to $\mathfrak{F}$.
Then there exist Borel measurable functions $f_1,\ldots,f_r$ with range $\{0,1\}$ such that $\mathfrak{X}\equalinlaw\mathfrak{X}^*=(\Nb,\mathcal{X}_1^*,\ldots,\mathcal{X}_r^*)$ with
\begin{equation}\label{eq:X*}
\mathcal{X}^*_j(\vec x)=f_j(\mathfrak{F}|_{[\max\vec x]},(\xi_s)_{s\subseteq\rng\vec x},(\prec_{\vec y})_{\vec y\sqsubseteq\vec x}),\quad\vec x\in\Nb^{\ar(R_j)},
\end{equation}
for i.i.d.\ Uniform$[0,1]$ random variables $(\xi_s)_{s\subseteq\Nb:|s|\leq\max_j\ar(R_j)}$ and independent uniform random orderings $(\prec_{s})_{s\subseteq\Nb:|s|\leq \max_j\ar(R_j)}$.
\end{theorem}

\subsection{Lipschitz continuous functions}

Theorems \ref{thm:CT} and \ref{thm:CTasym} are general Aldous--Hoover-type theorems for relatively exchangeable structures.
Exchangeability arises as a special case by taking $\mathfrak{F}$ to be either the complete or empty $\mathcal{L}$-structure so that the dependence of each $f_j$ on the first argument is moot.
The representation in \eqref{eq:X*} factors into our proofs below by allowing us to construct random Lipschitz continuous functions from which we build standard versions of $\mathbf{X}$ by either an i.i.d.\ sequence or a Poisson point process, depending on whether time is discrete or continuous, respectively.

In our main theorems for discrete time chains, $\mathbf{X}$ is an exchangeable, time homogeneous combinatorial Feller process on a Fra\"iss\'e space and, thus, all of its transitions can be represented by a single exchangeable transition probability measure
\[P(\mathfrak{N},\cdot)=\mathbb{P}\{\mathfrak{X}_1\in\cdot\mid\mathfrak{X}_0=\mathfrak{N}\},\quad\mathfrak{N}\in\XN,\]
which satisfies \eqref{eq:exch-tps} and for which
\begin{equation}\label{eq:consistent tps}
\mathbb{P}\{\mathfrak{X}_1|_{[n]}=\mathfrak{S}\mid\mathfrak{X}_0=\mathfrak{N}\}=\mathbb{P}\{\mathfrak{X}_1|_{[n]}=\mathfrak{S}\mid\mathfrak{X}_0=\mathfrak{N}'\}
\end{equation}
for all $\mathfrak{N},\mathfrak{N}'\in\XN$ with $\mathfrak{N}|_{[n]}=\mathfrak{N}'|_{[n]}$, for all $\mathfrak{S}\in\Xn$, for all $n\in\Nb$.
By \eqref{eq:consistent tps}, we define the finite state space transition probabilities $P^{[n]}$ on $\Xn$ by
\begin{equation}\label{eq:fidi tps}
P^{[n]}(\mathfrak{S},\mathfrak{S}'):=P(\mathfrak{N},\{\mathfrak{N}'\in\XN: \mathfrak{N}'|_{[n]}=\mathfrak{S}'\}),\quad\mathfrak{S},\mathfrak{S}'\in\Xn,\end{equation}
for any choice of $\mathfrak{N}\in\XN$ with $\mathfrak{N}|_{[n]}=\mathfrak{S}$.

Among our main objectives is to show that any random structure $\mathfrak{X}\sim P(\mathfrak{N},\cdot)$ satisfies $\mathfrak{X}\equalinlaw\Phi(\mathfrak{N})$ for some exchangeable random Lipschitz continuous function $\Phi\in\Lip(\XN)$ whose distribution does not depend on $\mathfrak{N}$.
To this end, we define a canonical embedding of each $\mathfrak{N}\in\XN$ into the Fra\"iss\'e limit of $\XN$.

Let $\XN$ be a Fra\"iss\'e space with Fra\"iss\'e limit $\mathfrak{F}$, and for any injection $\phi:[n]\to\Nb$ recall the notation $\mathfrak{F}^{\phi}$ defined in \eqref{eq:phi-image}.
Given $\mathfrak{F}$ and $\mathfrak{S}\in\age(\XN)$, we define $\rho_{\mathfrak{S}}=\rho_{\mathfrak{S},\mathfrak{F}}:[|\dom\mathfrak{S}|]\to\Nb$ (suppressing the dependence on $\mathfrak{F}$ for convenience) as follows.
We begin by choosing $\rho_{\mathfrak{S}}(1)$ to be the smallest positive integer such that $\mathfrak{F}^{\rho_{\mathfrak{S}}\upharpoonright\{1\}}=\mathfrak{S}|_{\{1\}}$, where in general $\rho_{\mathfrak{S}}\upharpoonright S$ denotes the domain restriction of $\rho_{\mathfrak{S}}$ to a function $S\to\Nb$.
Given $\rho_{\mathfrak{S}}(1),\ldots,\rho_{\mathfrak{S}}(m)$, with $m<|\dom\mathfrak{S}|$, we then choose $\rho_{\mathfrak{S}}(m+1)$ to be the smallest integer greater than $\rho_{\mathfrak{S}}(m)$ such that $\rho_{\mathfrak{S}}\upharpoonright[m+1]$ is an embedding $\mathfrak{S}|_{[m+1]}\to\mathfrak{F}$.
We continue until $\rho_{\mathfrak{S}}$ defines an embedding $\mathfrak{S}\to\mathfrak{F}$.
The existence of such an embedding for every $\mathfrak{S}$ is guaranteed by the status of $\mathfrak{F}$ as a Fra\"iss\'e limit.

We have defined $\rho_{\mathfrak{S}}$ for finite structures, but our construction extends to countably infinite structures $\mathfrak{N}\in\XN$ by noting that $\rho_{\mathfrak{S'}}\upharpoonright\dom\mathfrak{S}=\rho_{\mathfrak{S}}$ for $\mathfrak{S},\mathfrak{S}'\in\age(\XN)$ with $\mathfrak{S}'|_{[|\dom\mathfrak{S}|]}=\mathfrak{S}$.  (Here our convention that the domain of $\mathfrak{S}$ is not merely a subset of $\mathfrak{S}'$, but actually an initial segment, plays a crucial role.)

\begin{example}
For a concrete description of the canonical embedding, let $\mathcal{L}=(1)$ and let $\XN$ be the set of all subsets of $\Nb$, as in Example \ref{ex:universal-set}.
Suppose $\mathfrak{F}=(\Nb,R^{\mathfrak{F}})$ is a Fra\"{i}ss\'e limit with $R^{\mathfrak{F}}=\{1,2,4,8,\ldots\}$.
For $\mathfrak{S}=([1],\{1\})$, we define $\rho_{\mathfrak{S},\mathfrak{F}}(1)=1$; for $\mathfrak{S}=([1],\emptyset)$, we define $\rho_{\mathfrak{S},\mathfrak{F}}(1)=3$; for $\mathfrak{S}=([2],\{1\})$, we define $\rho_{\mathfrak{S},\mathfrak{F}}(1)=1$ and $\rho_{\mathfrak{S},\mathfrak{F}}(2)=3$; for $\mathfrak{S}=([2],\{2\})$, we define $\rho_{\mathfrak{S},\mathfrak{F}}(1)=3$ and $\rho_{\mathfrak{S},\mathfrak{F}}(2)=4$; and so on.
In this way, $\mathfrak{F}^{\rho_{\mathfrak{S},\mathfrak{F}}}=\mathfrak{S}$ for every $\mathfrak{S}$ and $\rho_{\mathfrak{S},\mathfrak{F}}$ extends $\rho_{\mathfrak{S}',\mathfrak{F}}$ whenever $\mathfrak{S}$ extends $\mathfrak{S}'$.
\end{example}


In what follows, we always assume $\mathfrak{F}$ is an exchangeable Fra\"iss\'e limit for the Fra\"{i}ss\'e space $\XN$, as guaranteed by Theorem \ref{thm:AFP}.
Given an exchangeable, consistent transition probability measure $P$ and an exchangeable Fra\"iss\'e limit $\mathfrak{F}$, we construct an exchangeable Lipschitz continuous function $\Phi\in\Lip(\XN)$ by first taking $\mathfrak{Y}\sim P(\mathfrak{M},\cdot)$ conditional on the event $\mathfrak{F}=\mathfrak{M}$.
Given $\mathfrak{Y}$, we define the random function $\Phi=\Phi_{\mathfrak{Y}}:\XN\to\XN$ by putting $\Phi(\mathfrak{N})=\mathfrak{Y}^{\rho_{\mathfrak{N}}}$ for each $\mathfrak{N}\in\XN$, where $\rho_{\mathfrak{N}}:=\rho_{\mathfrak{N},\mathfrak{M}}$ is the random canonical embedding $\mathfrak{N}\to\mathfrak{M}$ as defined above, on the event $\mathfrak{F}=\mathfrak{M}$.
(Note that the randomness in $\rho_{\mathfrak{N},\mathfrak{F}}$ is derived from the randomness in the construction of $\mathfrak{F}$.)

\begin{theorem}\label{thm:Lip-construct}
Let $\mathcal{L}$ be a signature, $P$ be an exchangeable, consistent transition probability on a Fra\"{i}ss\'e space $\XN\subseteq\lN$, $\mathfrak{F}$ be an exchangeable Fra\"iss\'e limit of $\XN$, and, given $\mathfrak{F}=\mathfrak{M}$, $\mathfrak{Y}\sim P(\mathfrak{M},\cdot)$.
The above function $\Phi=\Phi_{\mathfrak{Y}}$ is well defined and Lipschitz continuous with probability 1 and exchangeable in the sense of Definition \ref{defn:exch-measure}.
\end{theorem}

\begin{proof}
In general, $\XN$ is uncountable, so establishing these properties with probability 1 for each $\mathfrak{N}\in\XN$ is not enough.
However, since the signature $\mathcal{L}=\{R_1,\ldots,R_r\}$ contains only finitely many relations, $\ln$ is finite for every $n\geq1$ and, therefore, so is $\Xn\subseteq\ln$.
It follows that $\age(\XN)$ is at most countable, and so it is enough to establish that the relevant properties of $\Phi$ hold with probability 1 for each $\mathfrak{S}\in\age(\XN)$.
In this way, we show that the sequence of finite restrictions $(\Phi^{[n]})_{n\geq1}$ determines a projective limit $\Phi$ that is well defined and Lipschitz continuous with probability 1.

For $\Phi$ to be well defined, there must exist an embedding $\rho_{\mathfrak{S}}:\mathfrak{S}\to\mathfrak{F}$ of the type we describe for every $\mathfrak{S}\in\age(\XN)$.
Fix $\mathfrak{S}\in\age(\XN)$.
By Theorem \ref{thm:AFP}, $\mathfrak{F}$ is isomorphic to the Fra\"iss\'e limit of $\XN$ with probability 1 and, therefore, $\mathfrak{S}$ embeds into $\mathfrak{F}$ with probability 1.
The fact that $\rho_{\mathfrak{S}}:\mathfrak{S}\to\mathfrak{F}$ can be chosen so that $\rho_{\mathfrak{S}}(m)<\rho_{\mathfrak{S}}(m+1)$ is a consequence of ultrahomogeneity, induction, and countable additivity of probability measures.
Since this holds for every $\mathfrak{S}\in\age(\XN)$ with probability 1, the embedding $\rho_{\mathfrak{S}}$ can be defined for all $\mathfrak{S}\in\age(\XN)$ with probability 1.
For $\mathfrak{N}\in\XN$, we define $\rho_{\mathfrak{N}}$ by $\rho_{\mathfrak{N}}(n)=\rho_{\mathfrak{N}|_{[n]}}(n)$ for every $n\geq1$.
Since $\rho_{\mathfrak{N}|_{[n]}}$ is well defined with probability 1 for every $n\geq1$, $\rho_{\mathfrak{N}}$ is also well defined for every $\mathfrak{N}\in\XN$ with probability 1.
Lipschitz continuity of $\Phi$ is plain by noting that $\rho_{\mathfrak{S}'}\upharpoonright\dom\mathfrak{S}=\rho_{\mathfrak{S}}$ for all $\mathfrak{S},\mathfrak{S}'\in\age(\XN)$ for which $\mathfrak{S}'|_{\dom\mathfrak{S}}=\mathfrak{S}$.

For exchangeability, we need $\mathbb{P}\{\Phi(\mathfrak{N})\in A\}=\mathbb{P}\{\Phi(\mathfrak{N}^{\sigma})\in A^{\sigma}\}$ for all $\mathfrak{N}\in\XN$, $A\Borel\XN$, and permutations $\sigma:\Nb\to\Nb$, where $A^{\sigma}=\{x^{\sigma}: x\in A\}$ is the set of relabeled structures in $A$.
We have defined $\Phi(\mathfrak{N})=\mathfrak{Y}^{\rho_{\mathfrak{N}}}$, where $\rho_{\mathfrak{N}}=\rho_{\mathfrak{N},\mathfrak{F}}$ for an exchangeable Fra\"iss\'e limit $\mathfrak{F}$  and $\mathfrak{Y}\sim P(\mathfrak{M},\cdot)$ conditional on $\mathfrak{F}=\mathfrak{M}$.
For any $\mathfrak{N}\in\XN$ and $A\Borel\XN$,
\begin{eqnarray*}
\mathbb{P}\{\Phi(\mathfrak{N})\in A\}&=&\mathbb{P}\{\mathfrak{Y}^{\rho_{\mathfrak{N}}}\in A\}\\
&=&\mathbb{E}\mathbb{P}\{\mathfrak{Y}^{\rho_{\mathfrak{N},\mathfrak{F}}}\in A\mid\mathfrak{F}\}\\
&=&\mathbb{E}\mathbb{P}\{\mathfrak{Y}^{\phi}\in A\mid\mathfrak{F}, \rho_{\mathfrak{N},\mathfrak{F}}=\phi\}\\
&=&\mathbb{E}P(\mathfrak{N},A)\\
&=&\mathbb{E} P(\mathfrak{N}^{\sigma},A^{\sigma})\\
&=&\mathbb{E} \mathbb{P}\{\mathfrak{Y}^{\phi}\in A^{\sigma}\mid\mathfrak{F},\rho_{\mathfrak{N}^{\sigma},\mathfrak{F}}=\phi\}\\
&=&\mathbb{E} \mathbb{P}\{\mathfrak{Y}^{\rho_{\mathfrak{N}^{\sigma},\mathfrak{F}}}\in A^{\sigma}\mid\mathfrak{F}\}\\
&=&\mathbb{P}\{\mathfrak{Y}^{\rho_{\mathfrak{N}^{\sigma}}}\in A^{\sigma}\}\\
&=&\mathbb{P}\{\Phi(\mathfrak{N}^{\sigma})\in A^{\sigma}\}
\end{eqnarray*}
for all permutations $\sigma:\Nb\to\Nb$; thus, $\Phi$ is exchangeable.
(In the above string of equalities, the first line follows by the definition of $\Phi$; the second line is a consequence of the tower property for conditional expectation; the third line holds because the event $\{\rho_{\mathfrak{N},\mathfrak{F}}=\phi\}$ is measurable with respect to $\mathfrak{F}$; the fourth line is the projective Markov property; the fifth line follows by exchangeability of the transition measure; and the rest follows by applying each of the above operations in reverse.)  This completes the proof.
\end{proof}

\subsection{Conjugation invariant functions}

Recall that any permutation $\sigma:\Nb\rightarrow\Nb$ acts on $F\in\Lip(\XN)$ by {\em conjugation}, that is, $(\sigma^{-1}F\sigma)(\mathfrak{M})=F(\mathfrak{M}^{\sigma})^{\sigma^{-1}}$, and recall Definition \ref{defn:conjugation invariant} of a conjugation invariant function.
Conjugation invariant functions exist, for example, the identity $\XN\rightarrow\XN$, but not every $F\in\Lip(\XN)$ is conjugation invariant, for example, the coagulation operator in Example \ref{example:coalescent} below.

\begin{prop}\label{prop:conj-inv}
 Let $F\in\Lip(\XN)$.  Then $F$ is conjugation invariant if and only if for all finite $S\subset\Nb$ and all $\mathfrak{M},\mathfrak{M}'\in\XN$
\begin{equation}\label{eq:conj-inv-cond}
\mathfrak{M}|_S=\mathfrak{M}'|_S\quad\Longrightarrow\quad F(\mathfrak{M})|_S=F(\mathfrak{M}')|_S.\end{equation}
\end{prop}

\begin{proof}
Suppose $F\in\Lip(\XN)$ is conjugation invariant and, for finite $S\subset\Nb$ with $|S|=n$, assume $\mathfrak{M},\mathfrak{M}'\in\XN$ satisfy $\mathfrak{M}|_S=\mathfrak{M}'|_S$.
With the elements of $S$ ordered $s_1<\cdots<s_n$, we define the transpositions $\tau_i=(is_i)$ and put $\tau=\tau_n\cdots\tau_1$, so that $\tau^2=\text{id}$ and $\mathfrak{M}^{\tau}|_{[n]}=\mathfrak{M}'^{\tau}|_{[n]}$.

For any map $\sigma:S\to S'$, we define $\sigma[S]=\{\sigma(s): s\in S\}\subseteq S'$ as the image of $S$ under $\sigma$.
By conjugation invariance, $\sigma^{-1}F\sigma\in\Lip(\XN)$ for all permutations $\sigma:\Nb\rightarrow\Nb$.
Taking $\sigma=\tau^{-1}$, we have
\[\tau F\tau^{-1}(\mathfrak{M}^{\tau})|_{[n]}=\tau F\tau^{-1}(\mathfrak{M}'^{\tau})|_{[n]}.\]
Furthermore,
\begin{eqnarray*}
\tau F\tau^{-1}(\mathfrak{M}^{\tau})|_{[n]}&=&(F(\mathfrak{M}))^{\tau}|_{[n]}\\
&=&F(\mathfrak{M})|_{\tau^{-1}[n]}\\
&=&F(\mathfrak{M})|_{S},
\end{eqnarray*}
and likewise for $\tau F\tau^{-1}(\mathfrak{M}'^{\tau})|_{[n]}$, establishing the conclusion.

For the converse, suppose $F$ satisfies \eqref{eq:conj-inv-cond}, $\mathfrak{M}|_{[n]}=\mathfrak{M}'|_{[n]}$, and $\sigma:\Nb\to\Nb$ is a permutation.
Then $\mathfrak{M}^{\sigma}$ and $\mathfrak{M}'^{\sigma}$ agree on $\sigma[n]$.
By \eqref{eq:conj-inv-cond}, $F(\mathfrak{M}^{\sigma})$ and $F(\mathfrak{M}'^{\sigma})$ agree on $\sigma[n]$ and, thus, $F(\mathfrak{M}^{\sigma})^{\sigma^{-1}}|_{[n]}=F(\mathfrak{M}'^{\sigma})^{\sigma^{-1}}|_{[n]}$, so that $\sigma F\sigma^{-1}$ is Lipschitz continuous.
Conjugation invariance readily follows.

\end{proof}

For an injection $\phi:S\rightarrow\Nb$, we write $F^{\phi}$ to denote the image of $F$ under $\phi$, that is, $F^{\phi}(\mathfrak{M})=[F(\mathfrak{M})]^{\phi}$ for $\mathfrak{M}\in\XN$.
Note that $F^{\phi}$ is not Lipschitz continuous in general.

\begin{cor}\label{cor:Lip}
Let $\phi:S\to\Nb$ be an injection and let $F\in\Lip(\XN)$ be conjugation invariant.
Then $F^{\phi}\in\Lip(\mathcal{X}_S)$.
\end{cor}

\begin{proof}
Given conjugation invariant $F\in\Lip(\XN)$, we define $F':\mathcal{X}_S\rightarrow\mathcal{X}_S$ by $F'(\mathfrak{S})=F(\mathfrak{M})^{\phi}$ for any $\mathfrak{M}\in\XN$ such that $\mathfrak{M}|_S=\mathfrak{S}$.
This definition is well defined and Lipschitz continuous by Proposition \ref{prop:conj-inv}, and $F'$ coincides with our definition of $F^{\phi}$ above.
\end{proof}

In the following theorem, recall the definition of $\odap$-DAP from Definition \ref{defn:n-DAP}.

\begin{theorem}[Crane \& Towsner \cite{CraneTowsner2015}]\label{thm:CT2}
Let $\mathcal{L}=\{R_1,\ldots,R_r\}$ be a signature and $\XN\subseteq\lN$ be a Fra\"iss\'e space that has Fra\"iss\'e limit $\mathfrak{F}$ and which satisfies $\odap$-DAP.
Suppose $\mathfrak{X}$ is relatively exchangeable with respect to $\mathfrak{F}$.
Then there exist Borel measurable functions $f_1,\ldots,f_r$ with range $\{0,1\}$ such that $\mathfrak{X}\equalinlaw\mathfrak{X}^*=(\Nb,\mathcal{X}_1^*,\ldots,\mathcal{X}_r^*)$ with
\begin{equation}\label{eq:X*-nice}
\mathcal{X}^*_j(\vec x)=f_j(\mathfrak{F}|_{\rng\vec x},(\xi_s)_{s\subseteq\rng\vec x},(\prec_{\vec y})_{\vec y\sqsubseteq\vec x}),\quad\vec x\in\Nb^{\ar(R_j)},
\end{equation}
for i.i.d.\ Uniform$[0,1]$ random variables $(\xi_s)_{s\subseteq\Nb: |s|\leq\max_j\ar(R_j)}$ and independent uniform random orderings $(\prec_s)_{s\subseteq\Nb: |s|\leq\max_j\ar(R_j)}$.
\end{theorem}

\begin{remark}
The difference between the representations in Theorems \ref{thm:CTasym} and \ref{thm:CT2} is that the assignment of $\vec x$ in $\mathcal{X}^*_j$ depends on the entire initial substructure $\mathfrak{F}|_{[\max\vec x]}$ in \eqref{eq:X*} but only on $\mathfrak{F}|_{\rng\vec x}$ in \eqref{eq:X*-nice}.
The intuition behind the weaker representation in \eqref{eq:X*} is that in general the structure of $\mathfrak{F}$ is such that statements about $\vec x$ have implications beyond the substructure $\mathfrak{F}|_{\rng\vec x}$. 
A prime example is the case of an equivalence relation $\sim_\pi$, where the transitivity axiom implies that $i\sim_{\pi} k$ whenever $i\sim_{\pi} j$ and $j\sim_{\pi} k$.
Therefore, when building a relatively exchangeable random partition from another, it is important to account for the possible non-local restrictions imposed by the transitivity axiom.
The condition of $\odap$-DAP implies that such non-local restrictions are absent, leading to the difference between \eqref{eq:X*} and \eqref{eq:X*-nice}.
\end{remark}

Let $\mathfrak{F}$ be an exchangeable Fra\"iss\'e limit for a space $\XN$ with $\odap$-DAP, $\mu$ be an $\mathfrak{F}$-exchangeable probability measure on $\XN$, and $f=(f_1,\ldots,f_r)$ be any Borel measurable functions as in \eqref{eq:X*-nice}.  Let $(\xi_{s})_{s\subseteq\Nb: |s|\leq\max_j\ar(R_j)}$ be i.i.d.\ Uniform$[0,1]$ random variables and $(\prec_{s})_{s\subseteq\Nb: |s|\leq\max_j\ar(R_j)}$ be independent uniform random orderings.
 We construct a random function $\Psi:\XN\to\XN$ by $\mathfrak{M}\mapsto\Psi(\mathfrak{M})=(\Nb,\mathcal{R}_1',\ldots,\mathcal{R}'_r)$, where
\begin{equation}\label{eq:Psi}
\mathcal{R}_j'(\vec x)=f_j(\mathfrak{M}|_{\rng\vec x},(\xi_{s})_{s\subseteq\rng\vec x},(\prec_{\vec y})_{\vec y\sqsubseteq\vec x}),\quad\vec x\in\Nb^{\ar(R'_j)}.\end{equation}
This map is clearly Lipschitz continuous and has the further property that $\mathfrak{N}|_{\rng\vec x}=\mathfrak{N}'|_{\rng\vec x}$ implies $\Psi(\mathfrak{N})|_{\rng\vec x}=\Psi(\mathfrak{N}')|_{\rng\vec x}$.
Note further that $\mathfrak{M}^{\sigma}|_{\rng\vec x}=\mathfrak{M}|_{\rng\sigma^{-1}(\vec x)}$ for every $\vec x$ and every permutation $\sigma:\Nb\to\Nb$.

The following is an immediate consequence of Theorem \ref{thm:Lip-construct} and our construction of $\Psi$.

\begin{prop}\label{prop:conj-construct}
Let $\mathcal{L}$ be a signature, $P$ be an exchangeable, consistent transition probability on a Fra\"iss\'e space $\XN$ that satisfies $\odap$-DAP, and $\mathfrak{F}$ be an exchangeable Fra\"iss\'e limit of $\XN$.
Given $\mathfrak{F}=\mathfrak{M}$, let $f=(f_1,\ldots,f_r)$ be a collection of Borel measurable functions that determine the law of the $\mathfrak{M}$-exchangeable probability distribution $P(\mathfrak{M},\cdot)$ as in \eqref{eq:X*-nice}.
The function $\Psi$ defined from $f$ as in \eqref{eq:Psi} is well defined and conjugation invariant with probability 1 and exchangeable in the sense of Definition \ref{defn:exch-measure}.
\end{prop}

\section{Discrete time chains}\label{section:discrete}

Throughout this section, $T=\mathbb{Z}_+$ and $\mathbf{X}$ is a discrete time Markov chain on a Fra\"iss\'e space $\XN$.
For any probability measure $\mu$ on $\Lip(\XN)$, recall the definition of the standard $\mu$-process $\mathbf{X}_{\mu}^*$ from Section \ref{section:discrete time-summary}.

\begin{prop}\label{prop:easy way}
For any exchangeable probability measure $\mu$ on $\Lip(\XN)$, the standard $\mu$-process is an exchangeable Feller chain on $\XN$.
\end{prop}

\begin{proof}
Let $\mu$ be any exchangeable probability measure on $\Lip(\XN)$ so that $F\sim\mu$ implies $F(\mathfrak{M}^{\sigma})\equalinlaw F(\mathfrak{M})^{\sigma}$ for every $\mathfrak{M}\in\XN$ and every permutation $\sigma:\Nb\to\Nb$.
The standard $\mu$-process $\mathbf{X}_{\mu}^*$ is constructed as in \eqref{eq:mu-chain} from an i.i.d.\ sequence $F_1,F_2,\ldots$ from $\mu$.
Each $\mathbf{X}_{\mathfrak{M},\mu}^*$ clearly satisfies the Markov property on $\XN$ by its construction from the i.i.d.\ sequence $F_1,F_2,\ldots$.
Exchangeability of $\mu$ implies 
\[\mathbf{X}_{\mathfrak{M},\mu}^{*\sigma}=(F^{(m)}(\mathfrak{M})^{\sigma})_{m\in\mathbb{Z}_+}\equalinlaw(F^{(m)}(\mathfrak{M}^{\sigma}))_{m\in\mathbb{Z}_+}=\mathbf{X}_{\mathfrak{M}^{\sigma},\mu}^{*}\]
 for every permutation $\sigma:\Nb\to\Nb$ and every $\mathfrak{M}\in\XN$, where $F^{(m)}(\cdot)=(F_m\circ\cdots\circ F_1)(\cdot)$ and $F^{(0)}(\cdot)$ is defined as the identity.
The transition probabilities of $\mathbf{X}^*_{\mu}$ therefore satisfy
\begin{eqnarray*}
\mathbb{P}\{\mathfrak{X}_{m+1}^{*}|_{[n]}=\mathfrak{S}'\mid\mathfrak{X}_{m}^{*}|_{[n]}=\mathfrak{S}\}&=&\mathbb{P}\{F_{m+1}^{[n]}(\mathfrak{S})=\mathfrak{S}'\}\\
&=&\mathbb{P}\{F_{m+1}^{[n]}(\mathfrak{S}^{\sigma})=\mathfrak{S}'^{\sigma}\}\\
&=&\mathbb{P}\{\mathfrak{X}_{m+1}^{*}|_{[n]}=\mathfrak{S}'^{\sigma}\mid\mathfrak{X}_m^{*}|_{[n]}=\mathfrak{S}^{\sigma}\},
\end{eqnarray*}
and $\mathbf{X}^*_{\mu}$ is exchangeable in the sense of Definition \ref{defn:exchangeable CMP}.
By construction from a common sequence of Lipschitz continuous functions, $\mathbf{X}^*_{\mu}$ is consistent (Definition \ref{defn:projective}).
The Feller property follows from Theorem \ref{prop:Feller}.
\end{proof}

Our main theorems in discrete time establish the converse based on the concept of relative exchangeability from Section \ref{section:relative exchangeability}.

\begin{proof}[{Proof of Theorem \ref{thm:discrete}}]

The `if' direction follows from Proposition \ref{prop:easy way}.
The `only if' direction goes as follows.

By assumption, $\XN$ is a Fra\"iss\'e space (Definition \ref{defn:Fraisse space}), which by Theorem \ref{thm:Fraisse} possesses a unique (up to isomorphism) ultrahomogeneous structure $\mathfrak{F}\in\XN$ into which every $\mathfrak{M}\in\XN$ embeds.
By Theorem \ref{thm:AFP}, we can assume $\mathfrak{F}$ results from an exchangeable construction.
Let $P(\cdot,\cdot)$ be the transition probability for $\mathbf{X}$ as in \eqref{eq:tps} and let $P^{[n]}(\cdot,\cdot)$ be the induced transition probabilities of $\mathbf{X}^{[n]}$.
By Definition \ref{defn:rel-exch}, $\mathfrak{Y}\sim P(\mathfrak{M},\cdot)$ is relatively exchangeable with respect to $\mathfrak{M}$, for every $\mathfrak{M}\in\XN$.
By Theorem \ref{thm:Fraisse}, $\mathfrak{F}$ satisfies the conditions of Theorem \ref{thm:CTasym} with probability 1 and Theorem \ref{thm:Lip-construct} implies that $\Phi=\Phi_{\mathfrak{Y}}$ defined from $\mathfrak{Y}\sim P(\mathfrak{M}, \cdot)$, given $\mathfrak{F}=\mathfrak{M}$, is an exchangeable Lipschitz continuous function $\XN\to\XN$.

Let $\mu$ be the probability measure governing $\Phi$.
By relative exchangeability and Theorem \ref{thm:CTasym}, we have
\[\Phi(\mathfrak{N})|_{[n]}=\mathfrak{Y}^{\rho_{\mathfrak{N}}}|_{[n]}\sim P^{[n]}(\mathfrak{N}|_{[n]},\cdot).\]
Thus, we can construct $\mathbf{X}^*_{\mathfrak{M}}\equalinlaw\mathbf{X}_{\mathfrak{M}}$ by taking $\mathfrak{X}_0^*=\mathfrak{M}$, generating $\Phi_1,\Phi_2,\ldots$ i.i.d.\ from $\mu$, and putting $\mathfrak{X}_m^{*}=\Phi_{m}(\mathfrak{X}_{m-1}^*)$ for each $m\geq1$.
Relative exchangeability guarantees that $\mathfrak{X}^*_m$ obeys the appropriate transition probabilities for each $m\geq1$.
We conclude $\mathbf{X}^*_{\mathfrak{M}}\equalinlaw\mathbf{X}_{\mathfrak{M}}$ for every $\mathfrak{M}\in\XN$ by induction on $m\geq1$; thus, $\mathbf{X}^*_{\mu}\equalinlaw\mathbf{X}$.
This completes the proof.
\end{proof}

There are a few examples of Theorem \ref{thm:discrete} scattered throughout the literature.
Perhaps the most well known is the discrete time coalescent chain.
\begin{example}[Coalescent chain]\label{example:coalescent}
Let $\mathcal{P}_{\Nb}$ be the set of partitions of $\Nb$.
For each $\pi\in\partitionsN$, the {\em coagulation operator} is a Lipschitz continuous map $\text{Coag}(\cdot,\pi):\mathcal{P}_{\Nb}\rightarrow\mathcal{P}_{\Nb}$ that acts by putting  $\text{Coag}(\pi'',\pi)=\{B_1',B_2',\ldots\}$, where
\[B'_j=\bigcup_{i\in B_j}B''_i,\quad j\geq1,\]
for $\pi=\{B_1,B_2,\ldots\}$ and $\pi''=\{B''_1,B''_2,\ldots\}$.
The operation $\text{Coag}(\pi'',\pi)$ is called the {\em coagulation of $\pi''$ by $\pi$}.

An {\em exchangeable coalescent chain} $\mathbf{X}=(\mathfrak{X}_m)_{m\in\mathbb{Z}_+}$ evolves on $\mathcal{P}_{\Nb}$ as follows.
Let $\mu$ be an exchangeable probability distribution on $\mathcal{P}_{\Nb}$, put $\mathfrak{X}_0=\mathbf{0}_{\Nb}=\{\{1\},\{2\},\ldots\}$ (the partition of $\Nb$ into singletons), and take $\Pi_1,\Pi_2,\ldots$ i.i.d.\ from $\mu$.
For each $m\geq1$, we define $\mathfrak{X}_m=\text{Coag}(\mathfrak{X}_{m-1},\Pi_m)$.
Exchangeability of $\mathbf{X}$ is endowed by exchangeability of the partitions $\Pi_1,\Pi_2,\ldots$ and the action of the coagulation operator.

In typical treatments, for example, \cite{Bertoin2006}, the coalescent chain is defined by its transitions in terms of the coagulation operator, which is later shown to imply the projectivity property.
Theorem \ref{thm:discrete} establishes a general converse, which guarantees the existence of some random Lipschitz continuous function such that the above iterative description is possible.
Theorem \ref{thm:discrete}, however, does not establish any canonical form for the associated class of Lipschitz continuous operators.

While $\text{Coag}(\cdot,\pi)$ is Lipschitz continuous, it is not conjugation invariant.
For example, take $\pi=\{\{1,2\},\{3\}\}$, $\pi'^{(1)}=\{\{1,4\},\{2,3,5\}\}$, and $\pi'^{(2)}=\{\{1,4\},\{2\},\{3,5\}\}$, so that $\pi'^{(1)}|_{\{3,4,5\}}=\pi'^{(2)}|_{\{3,4,5\}}=\{\{3,5\},\{4\}\}$.
In this case, $\text{Coag}(\pi'^{(1)},\pi)=\{\{1,2,3,4,5\}\}$ and $\text{Coag}(\pi'^{(2)},\pi)=\{\{1,2,4\},\{3,5\}\}$, so that $\text{Coag}(\pi'^{(1)},\pi)|_{\{3,4,5\}}\neq\text{Coag}(\pi'^{(2)},\pi)|_{\{3,4,5\}}$.
By Proposition \ref{prop:conj-inv}, $\text{Coag}(\cdot,\pi)$ is not conjugation invariant.
This observation goes hand-in-hand with Example \ref{ex:partition}, in which we show that partitions do not have the stronger $n$-disjoint amalgamation property for $n=3$, and Theorem \ref{thm:conjugation-discrete}, which associates conjugation invariant functions with processes on a Fra\"iss\'e space with $\odap$-DAP.

\end{example}

\subsection{Refining Theorem \ref{thm:discrete}}\label{section:refining}

When $\XN$ also has $\odap$-DAP, we can refine Theorem \ref{thm:discrete} so that $\mathbf{X}^{*}_{\mu}$ is constructed from i.i.d.\ conjugation invariant functions.

\begin{theorem}\label{thm:conjugation-discrete}
In addition to the hypotheses of Theorem \ref{thm:discrete}, suppose that $\XN$ is a Fra\"iss\'e space with $\odap$-DAP.
Then the characteristic measure $\mu$  from Theorem \ref{thm:discrete} can be defined so that $\mu$-almost every $F\in\Lip(\XN)$ is conjugation invariant.
\end{theorem}

\begin{proof}

The proof follows the same recipe as that of Theorem \ref{thm:discrete}, except now we use the stronger representation in \eqref{eq:X*-nice}.
Let $\Psi$ be the conjugation invariant function constructed in Proposition \ref{prop:conj-construct}.
For any $\mathfrak{N}\in\XN$, $\Psi(\mathfrak{N})\sim P(\mathfrak{N},\cdot)$ by definition.
The rest follows by combining Theorem \ref{thm:discrete} with Proposition \ref{prop:conj-inv}, Corollary \ref{cor:Lip}, and Theorem \ref{thm:CT2}. \end{proof}

Unlike Theorem \ref{thm:discrete}, Theorem \ref{thm:conjugation-discrete} also provides the form of those Lipschitz continuous functions that determine the evolution of exchangeable Feller chains.
The space of partitions does not have $3$-DAP, see Example \ref{ex:partition}, and therefore Theorem \ref{thm:conjugation-discrete} does not apply to the coalescent chain above or in general to Markov chains on $\partitionsN$.
Exchangeable Feller chains on $\{0,1\}^{\Nb}$ give an easy illustration of Theorem \ref{thm:conjugation-discrete}, as we now show.

\begin{example}\label{ex:cut-paste-discrete}
Here we let $\mathcal{L}=\{R\}$ be a signature with $\ar(R)=1$, so that any $\{0,1\}$-valued sequence $x=x^1x^2\cdots$ corresponds to an $\mathcal{L}$-structure $\mathfrak{N}=(\Nb,\mathcal{R})$ with $\mathcal{R}=\{n\in\Nb: x^n=1\}$.
The space $\{0,1\}^{\Nb}$ corresponds to the set of all subsets of $\Nb$, which is plainly a Fra\"iss\'e class and has $\odap$-DAP; see Example \ref{ex:universal-set}.
 Theorem \ref{thm:conjugation-discrete} applies to Feller chains on $\lN=\{0,1\}^{\Nb}$.

We construct $\mathbf{X}$ by specifying $\mu$ on $\Lip(\{0,1\}^{\Nb})$ as follows.
Let $\Theta$ be a probability measure on $[0,1]\times[0,1]$, take $(\theta_0,\theta_1)\sim\Theta$, and, given $(\theta_0,\theta_1)$, define $Y_0=Y_0^1Y_0^2\cdots$ and $Y_1=Y_1^1Y_1^2\cdots$ to be conditionally independent with $Y_0$ a sequence of i.i.d.\ Bernoulli$(\theta_0)$ random variables and $Y_1$ a sequence of i.i.d.\ Bernoulli$(\theta_1)$ random variables, that is, for every $j=1,2,\ldots$,
\begin{align*}
\mathbb{P}\{Y_0^j=1\mid\theta_0\}&=1-\mathbb{P}\{Y_0^j=0\mid\theta_0\}=\theta_0\quad\text{and}\\
\mathbb{P}\{Y_1^j=1\mid\theta_1\}&=1-\mathbb{P}\{Y_1^j=0\mid\theta_1\}=\theta_1.
\end{align*}
Based on $Y=(Y_0,Y_1)$, we define $F_Y:\{0,1\}^{\Nb}\to\{0,1\}^{\Nb}$ by $x\mapsto x'=F_ Y(x)$ with
\[x'^n:=\left\{\begin{array}{cc} Y_0^n,& x^n=0,\\ Y_1^n,& x^n=1.\end{array}\right.\]
Defined in this way, $F_Y$ is exchangeable and conjugation invariant.
The measure $\mu$ governing $F_Y$ corresponds to the measure governing the i.i.d.\ sequence $F_1,F_2,\ldots$ of Lipschitz continuous functions in Theorem \ref{thm:conjugation-discrete}.

\end{example}

\section{Continuous time processes}\label{section:continuous time}

We now let $\mathbf{X}=\{\mathbf{X}_{\mathfrak{M}}: \mathfrak{M}\in\XN\}$ be a continuous time exchangeable Feller process on a Fra\"{i}ss\'e space $\XN$.
By the Feller property, the behavior of $\mathbf{X}$ is determined by its infinitesimal jump rates
\begin{equation}\label{eq:rates}
Q(\mathfrak{M},A):=\lim_{t\downarrow0}\frac{1}{t}\mathbb{P}\{\mathfrak{X}_t\in A\mid \mathfrak{X}_0=\mathfrak{M}\},\quad \mathfrak{M}\in\XN, A\Borel\XN\setminus\{\mathfrak{M}\}.\end{equation}
Equation \eqref{eq:rates} determines the transition rates of the finite state space processes $\mathbf{X}^{[n]}$ through
\begin{equation}\label{eq:Qn}
Q^{[n]}(\mathfrak{S},\mathfrak{S}'):=Q(\mathfrak{M},\,\{\mathfrak{M}'\in\XN:\,\mathfrak{M}'|_{[n]}=\mathfrak{S}'\}),\quad \mathfrak{S},\mathfrak{S}'\in\Xn,\,\mathfrak{S}\neq\mathfrak{S}',\end{equation}
where $\mathfrak{M}$ is any element of $\{\mathfrak{M}^*\in\XN:\,\mathfrak{M}^*|_{[n]}=\mathfrak{S}\}$.
The projective Markov property implies $Q^{[n]}(\mathfrak{S},\Xn\setminus\{\mathfrak{S}\})<\infty$ for all $\mathfrak{S}\in\Xn$ and all $n\in\mathbb{N}$, and exchangeability \eqref{eq:exch-tps} guarantees that $Q^{[n]}(\mathfrak{S}^{\sigma},\mathfrak{S}'^{\sigma})=Q^{[n]}(\mathfrak{S},\mathfrak{S}')$ for all permutations $\sigma:[n]\rightarrow[n]$.

The continuous time analog to the i.i.d.\ sequence of Lipschitz continuous functions in Theorem \ref{thm:discrete} is a construction by a time homogeneous Poisson point process on the space of Lipschitz continuous functions.

\begin{prop}\label{prop:Poisson-cts}
Let $\Lambda$ be an exchangeable measure on $\Lip(\XN)$ satisfying \eqref{eq:regularity-Lambda} and let $\mathbf{X}^*_{\Lambda}$ be the standard $\Lambda$-process constructed from Poisson point process $\mathbf{\Phi}=\{(t,F_t)\}\subseteq[0,\infty)\times\Lip(\XN)$ with intensity $dt\otimes\Lambda$ as in Section \ref{section:continuous-summary}.
Then $\mathbf{X}^*_{\Lambda}$ is an exchangeable Feller process on $\XN$.
\end{prop}

\begin{proof}
We need to establish that each of the finite processes $\mathbf{X}^{*[n]}_{\Lambda}$ on $\Xn$ is an exchangeable Markov process.
This is a consequence of the thinning property for Poisson point processes, exchangeability of $\Lambda$, and almost sure Lipschitz continuity of the functions that determine the jumps of $\mathbf{X}^{*}_{\Lambda}$.

To see this explicitly, let $\mathbf{\Phi}=\{(t,F_t)\}\subseteq[0,\infty)\times\Lip(\XN)$ be a Poisson point process with intensity $dt\otimes\Lambda$.
We define $\mathbf{\Phi}^{[n]}:=\{(t,F_t^{[n]})\}\subseteq[0,\infty)\times\Lip(\Xn)\setminus\{\idn\}$ from the atoms of $\mathbf{\Phi}$ for which $F_t^{[n]}\neq\idn$.
By the thinning property of Poisson processes, see \cite{KallenbergRM}, $\mathbf{\Phi}^{[n]}$ is also a Poisson process with intensity $dt\otimes\Lambda^{[n]}$, where 
\[\Lambda^{[n]}(F'):=\Lambda(\{F\in\Lip(\XN):\ F^{[n]}=F'\}),\quad F'\in\Lip(\Xn)\setminus\{\idn\},\]
and $\Lambda^{[n]}(\idn)=0$.
By \eqref{eq:regularity-Lambda}, $\Lambda^{[n]}(\Lip(\Xn))=\Lambda(\{F: F^{[n]}\neq\idn\})<\infty$ and the rate at which $\mathbf{X}^{*[n]}_{\Lambda}$ jumps out of each $\mathfrak{S}\in\Xn$ satisfies
\[Q^{[n]}(\mathfrak{S},\cdot)=\Lambda^{[n]}(\{F\in\Lip(\Xn): F(\mathfrak{S})\in\cdot\}).\]
Exchangeability of $\mathbf{X}^{*[n]}_{\Lambda}$ follows from exchangeability of $\Lambda$, as defined in \eqref{eq:exch-mu}.
The family of processes $(\mathbf{X}^{*[n]}_{\Lambda})_{n\in\Nb}$ is {\em compatible} by their construction from the same Poisson point process $\mathbf{\Phi}$, that is, $\mathbf{X}^{*[n+1]}_{\Lambda}|_{[n]}=\mathbf{X}^{*[n]}_{\Lambda}$ for all $n\geq1$.
Compatibility implies the existence of a Markov process $\mathbf{X}^*_{\Lambda}$ on $\XN$ for which the projective Markov property holds.
The Feller property follows by Theorem \ref{prop:Feller}.
\end{proof}

\subsection{Existence of exchangeable jump measure}
Proposition \ref{prop:Poisson-cts} shows that each $\mathbf{X}^{*}_{\Lambda}$ is exchangeable and has the Feller property.
Theorem \ref{thm:continuous-Lambda} asserts the converse.
To see that such a construction is always possible, let $\mathbf{X}$ be an exchangeable Feller process on a Fra\"{i}ss\'e space $\XN$.
For every $t>0$, Theorems \ref{thm:discrete} and \ref{thm:Lip-construct} guarantee the existence of an exchangeable probability measure $\Lambda_t$ on $\Lip(\XN)$ such that, for every $\mathfrak{M},\mathfrak{N}\in\XN$,
\begin{equation}\label{eq:varphi-t}
\Lambda_t(\{F\in\Lip(\XN): F(\mathfrak{N})\in\cdot\})=\mathbb{P}\{\mathfrak{X}_{t+s}\in\cdot\mid \mathfrak{X}_s=\mathfrak{N}, \mathfrak{X}_0=\mathfrak{M}\},\end{equation}
for all $s\geq0$.
The Chapman--Kolmogorov theorem implies that $(\Lambda_t)_{t\geq0}$ satisfies
\begin{eqnarray}\label{eq:Chapman-Kolmogorov}
\lefteqn{\Lambda_{t+s}(\{F\in\Lip(\XN): F(\mathfrak{M})\in\cdot\})=}\\
\notag&=&\int_{\Lip(\XN)}\Lambda_t(\{F''\in\Lip(\XN): (F''\circ F')(\mathfrak{M})\in\cdot\})\Lambda_s(dF'),
\end{eqnarray}
for all $s,t\geq0$ and all $\mathfrak{M}\in\XN$.
The family $(\Lambda_t)_{t\geq0}$ is not determined by these conditions, but nevertheless time homogeneity and \eqref{eq:varphi-t} implies the existence of an exchangeable rate measure $\Lambda$ on $\Lip(\XN)$ that satisfies \eqref{eq:regularity-Lambda} and determines the jump rates of $\mathbf{X}$.

\begin{theorem}\label{thm:existence}
Let $\mathbf{X}=\{\mathbf{X}_{\mathfrak{M}}: \mathfrak{M}\in\XN\}$ be an exchangeable Feller process on a Fra\"{i}ss\'e space $\XN$.
Then there exists an exchangeable measure $\Lambda$ satisfying \eqref{eq:regularity-Lambda} such that $\mathbf{X}\equalinlaw\mathbf{X}^*_{\Lambda}$ as constructed in Section \ref{section:continuous-summary}.
\end{theorem}

\begin{proof}
Let $\mathbf{X}$ be a continuous time exchangeable Feller process on $\XN$.
For every $t>0$, Theorem \ref{thm:discrete} guarantees that there exists a probability measure $\Lambda_t$ on $\Lip(\XN)$ for which \eqref{eq:varphi-t} holds.
For each $n\geq1$ and $t>0$, we define $\Lambda_t^{(n)}$ on $\Lip(\Xn)$ by
\begin{equation}\label{eq:varphi-t-n}
\Lambda_t^{(n)}(F):=\left\{\begin{array}{cc}
t^{-1}\Lambda_t(\{H\in\Lip(\XN): H^{[n]}=F\}),& F\neq\idn,\\ 0,& \text{otherwise.}\end{array}\right.
\end{equation}
For every $n\geq1$, we define $\lambda_n:=\max_{\mathfrak{S}\in\Xn}Q^{[n]}(\mathfrak{S},\Xn\setminus\{\mathfrak{S}\})<\infty$ for every $n\geq1$, which satisfies $\lambda_n<\infty$ by finiteness of $\Xn$ and the projective Markov property (Theorem \ref{prop:Feller}).

For every $F\in\Lip(\Xn)\setminus\{\idn\}$, there is some $\mathfrak{S}^*\in\Xn$ such that $F(\mathfrak{S}^*)\neq\mathfrak{S}^*$.
Let $\mathfrak{S}$ be any such $\mathfrak{S}^*$.
Then
\begin{eqnarray*}
\Lambda_t^{(n)}(F)&\leq& \Lambda_t^{(n)}(\{H\in\Lip(\Xn): H(\mathfrak{S})\neq\mathfrak{S}\})\\
&=& t^{-1} \mathbb{P}\{\mathfrak{X}_t^{[n]}\neq\mathfrak{S}\mid \mathfrak{X}_0^{[n]}=\mathfrak{S}\}\\
&\leq& t^{-1}\mathbb{P}\{(\mathfrak{X}_s^{[n]})_{s\geq0}\text{ is discontinuous on }[0,t]\mid \mathfrak{X}_0^{[n]}=\mathfrak{S}\}\\
&=&t^{-1}(1-\exp\{-t Q^{[n]}(\mathfrak{S},\Xn\setminus\{\mathfrak{S}\})\})\\
&\leq& t^{-1}(1-\exp\{-t\lambda_n\})\\
&\leq& t^{-1}(\lambda_n t)\\
&=&\lambda_n\\
&<&\infty.
\end{eqnarray*}
For a signature $\mathcal{L}=\{R_1,\ldots,R_r\}$ with $\max_{1\leq j\leq r}\ar(R_j)=m^*<\infty$, there can be no more than $2^{rn^{m^*}}$ elements in $\Xn$ and, therefore, there are at most $4^{r^2n^{2m^*}}<\infty$ elements in $\Lip(\Xn)$.
Thus, $\Lambda_t^{(n)}$ is a finite measure bounded by $4^{r^2n^{2m^*}}\lambda_n$ for all $t>0$ and $n\geq1$, and $(\Lambda_t^{(n)})_{t\downarrow0}$ is a bounded sequence of exchangeable measures for each $n\geq1$.

Finiteness of $\Lip(\Xn)$ and the Bolzano--Weierstrass theorem allows us to choose a subsequence $t_{n,1}>t_{n,2}>\cdots>0$ with $t_{n,k}\downarrow0$ such that $\Lambda_{t_{n,k}}^{(n)}\to\Lambda^{(n)}$, a finite measure on $\Lip(\Xn)$ with
\begin{equation}\label{eq:varphi-n}
\Lambda^{(n)}(F)=\lim_{k\to\infty}\Lambda^{(n)}_{t_{n,k}}(F),\quad F\in\Lip(\Xn).\end{equation}
Each $\Lambda_t^{(n)}$ is exchangeable and, therefore, so is $\Lambda^{(n)}$ for any choice of subsequence.
We have also defined $\Lambda_t^{(n)}(\idn)=0$ for all $t>0$ so that $\Lambda^{(n)}(\idn)=0$ for any choice of subsequence.

We ensure that $(\Lambda^{(n)})_{n\geq1}$ determines a measure $\Lambda$ on $\Lip(\XN)$ by defining the collection $(\Lambda^{(n)})_{n\geq1}$ from refining subsequences $\{\{t_{n,k}\}_{k\geq1}\}_{n\geq1}$.
We first choose $\{t_{1,k}\}_{k\geq1}$ so that $\Lambda_{t_{1,k}}^{(1)}$ converges to $\Lambda^{(1)}$ on $\Lip(\mathcal{X}_{[1]})$.
Given $\{t_{n,k}\}_{k\geq1}$ for any $n\geq1$, we then choose $\{t_{n+1,k}\}_{k\geq1}$ as a subsequence of $\{t_{n,k}\}_{k\geq1}$ so that $\Lambda_{t_{n+1,k}}$ converges to a measure $\Lambda^{(n+1)}$ on $\Lip(\mathcal{X}_{[n+1]})$.
We can always choose such a subsequence by the Bolzano--Weierstrass theorem.

By our construction from a refining collection of subsequences, we have
\begin{eqnarray*}
\lefteqn{\Lambda^{(n+1)}(\{F'\in\Lip(\mathcal{X}_{[n+1]}): F'^{[n]}=F\})=}\\
&=&\sum_{F'\in\Lip(\mathcal{X}_{[n+1]}): F'^{[n]}=F}\Lambda^{(n+1)}(F')\\
&=&\sum_{F'\in\Lip(\mathcal{X}_{[n+1]}): F'^{[n]}=F}\lim_{k\to\infty}\Lambda_{t_{n+1,k}}^{(n+1)}(F')\\
&=&\lim_{k\to\infty}t_{n+1,k}^{-1}\sum_{F'\in\Lip(\mathcal{X}_{[n+1]}): F'^{[n]}=F}\Lambda_{t_{n+1,k}}(\{F^*\in\Lip(\XN): F^{*[n+1]}=F'\})\\
&=&\lim_{k\to\infty}t_{n+1,k}^{-1}\Lambda_{t_{n+1,k}}(\{F^*\in\Lip(\XN): F^{*[n]}=F\})\\
&=&\Lambda^{(n)}(F),
\end{eqnarray*}
where the interchange of sum and limit is justified by the bounded convergence theorem.

The Borel $\sigma$-field on $\Lip(\XN)$ is generated by the $\pi$-system of events of the form
\[\{F^*\in\Lip(\XN): F^{*[n]}=F\},\quad F\in\Lip(\Xn),\quad n\in\Nb;\]
therefore, we can define a set function $\Lambda$ on $\Lip(\XN)$ on sets of the form $\{F^*\in\Lip(\XN): F^{*[n]}=F\}$ for $F\in\Lip(\Xn)\setminus\{\idn\}$ by
\[\Lambda(\{F^*\in\Lip(\XN): F^{*[n]}=F\})=\Lambda^{(n)}(F),\quad F\in\Lip(\Xn)\setminus\{\idn\}.\]
By construction, $(\Lambda^{(n)})_{n\in\Nb}$ satisfies
\[\Lambda^{(m)}(F)=\sum_{F'\in\Lip(\Xn): F'^{[m]}=F}\Lambda^{(n)}(F')\]
for every $m\leq n$ and $F\in\Lip(\mathcal{X}_{[m]})\setminus\{\text{id}_{[m]}\}$.
Carath\'eodory's extension theorem guarantees an extension to a measure $\Lambda$ on $\Lip(\XN)\setminus\{\idN\}$.
For each $n\in\Nb$, $\Lambda^{(n)}$ determines the jump rates of an exchangeable Markov chain on $\Xn$, giving
\[\Lambda(\{F^*\in\Lip(\XN): F^{*[n]}\neq\idn\})=\Lambda^{(n)}(\Xn\setminus\{\idn\})<\infty\]
for all $n\in\Nb$.
Putting $\Lambda(\{\idN\})=0$ gives \eqref{eq:regularity-Lambda}.
Exchangeability of every $\Lambda_t$, $t>0$, makes each $\Lambda^{(n)}$, $n\in\Nb$, exchangeable and, hence, implies that $\Lambda$ is exchangeable.
Finally, we must show that such $\Lambda$ gives $\mathbf{X}^{*}_{\Lambda}\equalinlaw\mathbf{X}$.

Let $\Lambda$ be the exchangeable measure from above and let $\mathbf{\Phi}=\{(t,\Phi_t)\}\subseteq[0,\infty)\times\Lip(\XN)$ be a Poisson point process with intensity $dt\otimes\Lambda$.
Since $\Lambda$ satisfies \eqref{eq:regularity-Lambda}, Proposition \ref{prop:Poisson-cts} allows us to construct $\mathbf{X}^*_{\Lambda}$ from $\mathbf{\Phi}$ as in Section \ref{section:continuous-summary}.
The jump rates of each $\mathbf{X}_{\Lambda}^{*[n]}$ are determined by a thinned version $\mathbf{\Phi}^{[n]}$ of $\mathbf{\Phi}$ that only keeps the atoms $(t,\Phi_t)$ for which $\Phi_t^{[n]}\neq\idn$.
By the thinning property of Poisson random measures and our construction of $\Lambda$, the intensity of $\mathbf{\Phi}^{[n]}$ is $\Lambda^{(n)}$ as defined in \eqref{eq:varphi-n}, and it follows immediately that the jump rate from $\mathfrak{S}$ to $\mathfrak{S}'\neq\mathfrak{S}$ in $\mathbf{X}^{*[n]}_{\Lambda}$ is
\[\Lambda^{(n)}(\{F\in\Lip(\Xn): F(\mathfrak{S})=\mathfrak{S}'\})=Q^{[n]}(\mathfrak{S},\mathfrak{S}'),\]
for $Q^{[n]}(\cdot,\cdot)$ as in \eqref{eq:Qn}.
By construction, $(\mathbf{X}^{*[n]}_{\Lambda})_{n\in\Nb}$ is a compatible collection of c\`adl\`ag exchangeable Markov chains governed by the finite-dimensional transition law of $\mathbf{X}$.
The proof is complete.
\end{proof}

\begin{proof}[Proof of Theorem \ref{thm:continuous-Lambda}]
The `if' direction follows directly from Proposition \ref{prop:Poisson-cts}.
The `only if' direction follows from Theorem \ref{thm:existence}.
\end{proof}

\begin{example}[Coalescent process]\label{example:coag}
The coalescent chain in Example \ref{example:coalescent} has a continuous time analog  defined as follows.
Let $(\mathbf{P}_t)_{t\geq0}$ be the semigroup of a Markov process on $\partitionsN$ such that, for every $t\geq0$ and every bounded continuous $g:\partitionsN\to\mathbb{R}$,
\[\mathbf{P}_t g(\pi)=\mathbb{E}g(\text{Coag}(\pi,\Pi_t))\]
for some exchangeable random partition $\Pi_t$.
By definition, this process is defined in terms of the Lipschitz continuous coagulation operator.
This is a special case of Theorem \ref{thm:continuous-Lambda}.

By analyzing the interplay between the coagulation operator and the exchangeability condition, the transition rates enjoy a special structure with $\Lambda=\mathbf{c}\kappa+\varrho_{\nu}$ for a unique constant $\mathbf{c}\geq0$ and measures $\kappa$ and $\nu$ defined as follows.

Let $\nu$ be a measure on the ranked-simplex
\[\Delta^{\downarrow}:=\{(s_1,s_2,\ldots): s_1\geq s_2\geq \cdots\geq0,\ \sum_{i\geq1}s_i\leq 1\}\]
satisfying
\begin{equation}\label{eq:reg-paintbox} \nu(\{\mathbf{0}\})=0\quad\text{and}\quad\int_{\Delta^{\downarrow}}\left(\sum_{i=1}^{\infty}s_i^2\right)\nu(ds)<\infty,\end{equation}
where $\mathbf{0}:=(0,0,\ldots)$.
We write $\varrho_{\nu}$ to denote {\em Kingman's paintbox measure} \cite{Kingman1978b} directed by $\nu$ on the space of partitions of $\Nb$, that is,
\begin{equation*}\label{eq:regularity-paintbox}
\varrho_{\nu}(\cdot):=\int_{\Delta^{\downarrow}}\varrho_s(\cdot)\nu(ds),\end{equation*}
where $\varrho_{s}(\cdot)$ is the paintbox measure for a fixed $s\in\Delta^{\downarrow}$; see Bertoin \cite[Chapters 2 and 4]{Bertoin2006} for more details.
By \eqref{eq:reg-paintbox}, the image measure of $\varrho_{\nu}$ on $\Lip(\XN)$ defined by $\pi\mapsto\coag(\cdot,\pi)$ satisfies \eqref{eq:regularity-Lambda} and is exchangeable.

For $i<j$ let $e_{i,j}$ be the partition of $\Nb$ with all blocks singletons except one block of size two given by $\{i,j\}$.
The {\em Kingman measure} $\kappa(\cdot)=\sum_{i<j}\delta_{e_{i,j}}(\cdot)$ puts mass 1 on each of these partitions.
The image of $\kappa$ by $\pi\mapsto\coag(\cdot,\pi)$ also satisfies \eqref{eq:regularity-Lambda} and is exchangeable.

\end{example}

\begin{example}[Fragmentation process]\label{example:frag}

The class of homogeneous fragmentation processes evolves in the opposite direction of the above coalescent process.
Instead of merging blocks together by the coagulation operator, a fragmentation process breaks them apart using the {\em fragmentation operator}.
For a partition $\pi=\{B_1,B_2,\ldots\}$ with blocks listed in increasing order of their smallest element, we define $\text{Frag}:\partitionsN\times\partitionsN\times\Nb\to\partitionsN$ by $\pi\mapsto\pi'=\text{Frag}(\pi,\pi'',k)$, where $\pi'$ has blocks $B_1,\ldots,B_{k-1},B_{k+1},\ldots$ just as in $\pi$ along with the blocks $B_k\cap B''_1,B_k\cap B''_2,\ldots$, that is, the blocks obtained by restricting $\pi''$ to $B_k$.
We call $\text{Frag}(\pi,\pi'',k)$ the {\em fragmentation of the $k$th block of $\pi$ by $\pi''$}.

Heuristically, the homogeneous fragmentation process $(\mathfrak{X}_t)_{t\in[0,\infty)}$ is an exchangeable Feller process on $\partitionsN$ constructed from initial state $\mathbf{1}_{\Nb}=\{\Nb\}$ and a Poisson point process $\{(t,\Pi_t,K_t)\}\subset[0,\infty)\times\partitionsN\times\Nb$ by putting $\mathfrak{X}_t=\text{Frag}(\mathfrak{X}_{t-},\Pi_t,K_t)$ if $t\geq0$ is an atom time.
(The rigorous construction follows the same program as in Section \ref{section:continuous-summary}; see also Bertoin \cite[Chapter 3]{Bertoin2006}.)
Here the measure $\Lambda$ decomposes as $\Lambda=\mathbf{c}\epsilon+\varrho_{\nu}$ for a unique constant $\mathbf{c}\geq0$, $\varrho_{\nu}$ as defined in Example \ref{example:coag} with regularity constraint
\begin{equation}\label{eq:regularity-frag}
\nu(\{(1,0,\ldots)\})=0\quad\text{and}\quad\int_{\Delta^{\downarrow}}(1-s_1)\nu(ds)<\infty,
\end{equation}
and {\em erosion measure} $\epsilon$ defined as follows.
For each $n\in\Nb$, let $\mathbf{e}_{n}:=\{\Nb\setminus\{n\},\{n\}\}$ be the partition with element $n$ isolated from the rest and define $\epsilon(\cdot):=\sum_{n\in\Nb}\delta_{\mathbf{e}_n}(\cdot)$.
The construction proceeds by putting $\Lambda=\mathbf{c}\epsilon+\varrho_{\nu}$ and generating a Poisson point process on $[0,\infty)\times\partitionsN\times\Nb$ with intensity $dt\otimes\Lambda\otimes K$, where $K$ is counting measure on $\Nb$.
Bertoin \cite{Bertoin2001a} showed that every homogeneous fragmentation process can be constructed from this procedure for unique $\mathbf{c}\geq0$ and $\nu$ satisfying \eqref{eq:regularity-frag}.

\end{example}

Berestycki \cite{Berestycki2004} combined the dynamics from Examples \ref{example:coag} and \ref{example:frag} into {\em exchangeable fragmentation-coalescent} (EFC) processes.
We omit those details here.

\section{L\'evy--It\^o--Khintchine structure}\label{section:Levy-Ito}
When $\XN$ has $\odap$-DAP, we can refine Theorem \ref{thm:continuous-Lambda} by decomposing the characteristic measure $\Lambda$ into mutually singular pieces that capture different qualitative features of $\mathbf{X}$.
We call this the {\em L\'evy--It\^o--Khintchine representation} for combinatorial Markov processes because of its resemblance to the eponymous representation for classical L\'evy processes.
Examples \ref{example:coag} and \ref{example:frag} show that analogous decompositions may also hold in the absence of $\odap$-DAP, but in those cases the representation hinges on specific assumptions about the transition behavior of the given process.
Notice that both the coalescent and fragmentation processes evolve on the space of partitions of $\Nb$, but their constructions by the coagulation and fragmentation operators lead to fundamentally different descriptions as processes generated by iteratively applying random Lipschitz continuous functions.
By contrast, when $\mathbf{X}$ evolves on a space with $\odap$-DAP, we achieve a generic decomposition that applies in all cases, without any prior assumption about the transition behavior.

We set the stage for Theorem \ref{thm:Levy-Ito} by demonstrating its implications in specific cases.

\begin{example}\label{example:graph-valued}
  Consider a continuous time exchangeable Feller process $\mathbf{X}=\{\mathbf{X}_{\mathfrak{G}}: \mathfrak{G}\in\graphsN\}$ on the Fra\"iss\'e space of countable graphs $\graphsN$ from Example \ref{ex:graph}.
(The directed and undirected cases are similar, and so we discuss here the undirected case.)
The transition behavior of each $\mathbf{X}_{\mathfrak{G}}=(\mathfrak{X}_t)_{t\in[0,\infty)}$ is determined by an exchangeable measure $\Lambda$ on $\Lip(\XN)$ which decomposes into mutually singular components as follows.
At the atom times $t$ of a Poisson point process with intensity $dt\otimes\Lambda$ either
\begin{itemize}
	\item[($\emptyset$)] $\mathfrak{X}_{t-}$ undergoes a global change governed by some conjugation invariant Lipschitz continuous function just as in the discrete time setting of Theorem \ref{thm:conjugation-discrete},
	\item[(1)] a single vertex $i^*\in\Nb$ is chosen and the collection of edges $(\{i^*,j\})_{j\neq i}$ changes according to some exchangeable transition rule on $\{0,1\}$-valued sequences while all edges not involving $i^*$ stay fixed, or
	\item[(1,1)] a pair $i^*,j^*$, $i^*\neq j^*$, is chosen and the edge $\{i^*,j^*\}$ changes status while the rest of the graph stays fixed.
\end{itemize}


Here we see the first feature of our representation: there may be times at which we select some set of vertices and allow only changes in which all such vertices are included.  In case $(1)$, we select a single vertex $i^*$, while in the case $(1,1)$ we select two vertices $i^*$ and $j^*$.
\end{example}

Our indexing scheme ($\emptyset$), ($1$), and ($1,1$) conveys the structure of the jump behavior.
For example, we understand the case ($1,1$) to mean ``choose two distinct vertices, each of multiplicity $1$.''
Most generally, we classify jumps of $(\mathfrak{X}_t)_{t\in[0,\infty)}$ by a sequence ($\alpha_1,\ldots,\alpha_k$) with $\alpha_1\geq\cdots\geq\alpha_k\geq1$, which indicates that the transition is determined by first choosing a multiset with $k$ distinct elements $i_1,\ldots,i_k$ in which $i_1$ appears with multiplicity $\alpha_1$, $i_2$ appears with multiplicity $\alpha_2$, and so on and then modifying $\mathfrak{X}_{t-}$ in an exchangeable way such that only relations involving the chosen multiset of elements can change.
The rationale behind this indexing scheme becomes clearer with more elaborate examples and the formal exposition to follow.

\begin{example}\label{example:hypergraph-valued}
  Consider a continuous time exchangeable Feller process $\mathbf{X}=\{\mathbf{X}_{\mathfrak{G}}: \mathfrak{G}\in\graphsN\}$ on the Fra\"iss\'e space of anti-reflexive, countable $n$-ary hypergraphs from Example \ref{ex:hypergraph}.
The transition behavior of each $\mathbf{X}_{\mathfrak{G}}=(\mathfrak{X}_t)_{t\in[0,\infty)}$ is determined by an exchangeable measure $\Lambda$ on $\Lip(\XN)$ which decomposes into mutually singular components as follows.
At the atom times $t$ of a Poisson point process with intensity $dt\otimes\Lambda$ either
\begin{itemize}
	\item[($\emptyset$)] $\mathfrak{X}_{t-}$ undergoes a global change governed by some conjugation invariant Lipschitz continuous function just as in the discrete time setting of Theorem \ref{thm:conjugation-discrete},
	\item[(1)] a single vertex $i^*_1\in\Nb$ is chosen and $(\{i^*_1,j_2,\ldots,j_n\})_{j_l\neq i^*_1}$, with $j_2\neq j_3\neq\dots\neq j_n$, changes according to some exchangeable transition rule for $(n-1)$-ary hypergraphs while all edges not involving $i^*_1$ stay fixed,
	\item[(1,1)] a pair $i^*_1,i^*_2$, $i^*_1\neq i^*_2$, is chosen and $(\{i^*_i,i^*_2,j_3,\ldots,j_n\})_{j_l\neq i^*_1,i^*_2}$, with $j_3\neq\cdots\neq j_n$, changes according to some exchangeable transition rule for $(n-2)$-ary hypergraphs while all edges not involving \emph{both} $i^*_1$ and $i^*_2$ stay fixed,
        \item [$\vdots$]
        \item [(1,\ldots,1)] $n$ distinct vertices $i^*_1,\ldots,i^*_n$ are chosen and the edge $\{i^*_1,\ldots,i^*_n\}$ changes status while the rest of $\mathfrak{X}_{t-}$ stays fixed.
\end{itemize}

\end{example}

The above examples involve anti-reflexive structures, meaning each vertex can appear at most once in any given edge.
On spaces without this restriction, our representation captures the feature that transitions can be indexed by collections of vertices in which some vertices appear with multiplicity.

\begin{example}\label{example:PR}
Let $\XN$ be the Fra\"{i}ss\'e space of undirected graphs permitting self-loops.  
The transition behavior of any $(\mathfrak{X}_t)_{t\in[0,\infty)}$ on $\XN$ is determined by an exchangeable measure $\Lambda$ on $\Lip(\XN)$ which decomposes into mutually singular pieces indexed by ($\emptyset$), (1), (1,1), and (2) as follows.
At the atom times $t$ of a Poisson point process with intensity $dt\otimes\Lambda$ either
\begin{itemize}
	\item[($\emptyset$)] $\mathfrak{X}_{t-}$ undergoes a global change described by some conjugation invariant Lipschitz continuous function just as in the discrete time setting of Theorem \ref{thm:conjugation-discrete},
	\item[(1)] a single vertex $i^*\in\Nb$ is chosen and the collection of edges $(\{i^*,j\})_{j\neq i}$ changes according to some exchangeable transition rule on $\{0,1\}$-valued sequences, as in the discrete time case for $(1)$-structures, with the self-loop at $i^*$ possibly changing and any edge not involving $i^*$ remaining fixed,
	\item[(2)] a single $i^*\in\Nb$ is chosen and the self-loop at $i^*$ changes while the rest of $\mathfrak{X}_{t-}$ stays fixed, or
	\item[(1,1)] a pair $i^*,j^*\in\Nb$, $i^*\neq j^*$, is chosen and the edge $\{i^*,j^*\}$ changes while the rest of $\mathfrak{X}_{t-}$ stays fixed.
\end{itemize}

The new feature here is the distinction between the (1) case and the (2) case.  Both involve only a single vertex $i^*$, but in the (2) case the only relations considered are those in which the vertex $i^*$ appears at least twice.
We interpret the indexing of the jump labeled (2) to mean ``choose one vertex $i^*$ and change any edge in which $i^*$ appears with multiplicity 2.''

\end{example}

One last example further illustrates the distinction.

\begin{example}\label{example:PR_color}
Let $\XN$ be the Fra\"{i}ss\'e space of \emph{colored} undirected graphs with loops as in Example \ref{ex:colorgraph}.
Specifically, we take $\mathcal{L}=\{P,R\}$ with binary edge relation $R$ and unary relation $P$ which records the coloring.  
The transition behavior of any exchangeable $(\mathfrak{X}_t)_{t\in[0,\infty)}$ on $\XN$ is determined by an exchangeable measure $\Lambda$ on $\Lip(\XN)$ which decomposes into mutually singular components as follows.
At the atom times $t$ of a Poisson point process with intensity $dt\otimes\Lambda$ either
\begin{itemize}
	\item[(0)] $\mathfrak{X}_{t-}$ undergoes a global change described by some conjugation invariant Lipschitz continuous function just as in the discrete time setting of Theorem \ref{thm:conjugation-discrete}.
	\item[(1)] a single $i^*\in\Nb$ is chosen and the relations $((R^{\mathfrak{X}_{t-}}(i^*,j),R^{\mathfrak{X}_{t-}}(j,i^*))_{j\neq i^*}$ undergo an exchangeable transition as in the discrete time case unary relations, with the self-loop $R^{\mathfrak{X}_{t-}}(i^*,i^*)$ and coloring $P^{\mathfrak{X}_{t-}}(i^*)$ possibly changing while the rest of $\mathfrak{X}_{t-}$ stays fixed,
	\item[(2)] a single $i^*\in\Nb$ is chosen and the self-loop $R^{\mathfrak{X}_{t-}}(i^*,i^*)$ changes while the rest of $\mathfrak{X}_{t-}$, including $P^{\mathfrak{X}_{t-}}(i^*)$, stays fixed, or
	\item[(1,1)] a pair $i^*,j^*\in\Nb$, $i^*\neq j^*$, is chosen and $R^{\mathfrak{X}_{t-}}(i^*,j^*)$ changes while the rest of $\mathfrak{X}_{t-}$ stays fixed.
\end{itemize}

\end{example}

\subsection{Proof of L\'evy--It\^o--Khintchine decomposition}  

We fix a signature $\mathcal{L}=\{R_1,\ldots,R_r\}$ and and let $\XN\subseteq\lN$ be a Fra\"{i}ss\'e space.
We write $s\finite\Nb$ to indicate that $s$ is a finite multiset of $\Nb$, which we identify by its collection of multiplicities $\mathbf{m}(s)=(m_i(s))_{i\geq1}$, with each $m_i(s)\geq0$ indicating the multiplicity of $i$ in $s$.
In this way, we express $s=\{i^{m_i}\}_{i\geq1}$, usually with elements of multiplicity 0 omitted in specific examples.
We write $|s|=\sum_{i\geq1}m_i(s)$ to denote the cardinality of $s$ and $\rng s=\{i: m_i(s)\geq1\}$ to denote the set of distinct elements that appear at least once in $s$ (without multiplicity).

With this notation, the usual set theoretic relations and operations extend to multisets $s,s'\finite\Nb$, allowing us to write $s\subseteq s'$ to denote that $m_i(s)\leq m_i(s')$ for all $i\geq1$, $s\cap s':=\{i^{m_i(s\cap s')}\}_{i\geq1}$ with $m_i(s\cap s')=m_i(s)\wedge m_i(s')$ for every $i\geq1$, and $s\cup s':=\{i^{m_i(s\cup s')}\}_{i\geq1}$ with $m_i(s\cup s')=m_i(s)\vee m_i(s')$ for every $i\geq1$.
For example, the multiset $s=\{1,1,1,2,2,4\}$ has $\mathbf{m}=(3,2,0,1,0,0,\ldots)$, which we condense to $s=\{1^{3},2^2,4^1\}$.
For $s=\{1^3,2^2,4^1\}$ and $s'=\{1^2,2^3,3^2,5^3\}$, we have $s\cap s'=\{1^2,2^2\}$ and $s\cup s'=\{1^3,2^3,3^2,4^1,5^3\}$.


For any integer $k\geq1$, we write $\alpha\vdash k$ to denote that $\alpha$ is a {\em partition of the integer $k$}, meaning $\alpha=(\alpha_1,\alpha_2,\ldots)$ has $\alpha_1\geq\alpha_2\geq\cdots\geq0$ and $\sum_{i=1}^{\infty}\alpha_i=k$.
Every $\alpha\vdash k$ corresponds to a canonical multiset $s_{\alpha}:=\{i^{\alpha_i}\}_{i\geq1}$ for which $|s_{\alpha}|=k$.
(Note our convention to pad $\alpha$ with infinitely many trailing $0$s allows us to easily move between a partition $\alpha$ and its canonical multiset $s_{\alpha}$.  
We can and will omit the trailing 0s whenever convenient.)

For any $s=\{i^{m_i(s)}\}_{i\geq1}\finite\Nb$, the decreasing multiplicity vector, denoted $\mathbf{m}^{\downarrow}(s)=(m_i^{\downarrow}(s))_{i\geq1}$, lists the multiplicities of $\mathbf{m}(s)$ in decreasing order so that $m_i^{\downarrow}(s)\geq m_{i+1}^{\downarrow}(s)$.
This ranked reordering corresponds to a unique partition of the integer $|s|$, which we call the {\em partition type of $s$}.
We write $s^{\downarrow}=(s_1^{m_1^{\downarrow}(s)},s_2^{m_2^{\downarrow}(s)},\ldots)$ to denote the ordering of $s$ so that $m_i^{\downarrow}(s)\geq m_{i+1}^{\downarrow}(s)$ for all $i\geq1$ and $s_i>s_{i+1}$ whenever $m_i^{\downarrow}(s)=m_{i+1}^{\downarrow}(s)$.
For example, $s=\{1^2,2^3,4^1,5^3\}$ has $\mathbf{m}^{\downarrow}(s)=(3,3,2,1)\vdash9$ and $s^{\downarrow}=(5^3,2^3,1^2,4^1)$.


Every $\vec x=(x_1,\ldots,x_d)\in\Nb^d$ determines a multiset, written as $x:=\{i^{m_i(x)}\}_{i\geq1}$, by ignoring the order in which elements appear.
For example, $\vec x=(1,2,1,1,4,2,5)$ determines $x=\{1^3,2^2,4^1,5^1\}$ and $x^{\downarrow}=(1^3,2^2,5^1,4^1)$.


Theorems \ref{thm:discrete} and \ref{thm:continuous-Lambda} characterize the behavior of exchangeable Feller processes in terms of the action of randomly chosen Lipschitz continuous functions $F:\XN\to\XN$.
When $\XN$ has $\odap$-DAP, we further analyze the transitions induced by $F$ by considering the collection of tuples $\vec x$ for which $R_j^{F(\mathfrak{M})}(\vec x)\neq R_j^{\mathfrak{M}}(\vec x)$ for some $\mathfrak{M}\in\XN$.
In other words, we consider whether $F$ acts nontrivially in location $\vec x$ for some $\mathfrak{M}\in\XN$.
By Proposition \ref{prop:conj-inv}, conjugation invariant functions act locally on $\XN$.
Given any $F\in\Lip(\XN)$ and $\vec x\in\Nb^{\ar(R_j)}$, we write $R_j^{F(\bullet)}(\vec x)\neq R_j^{\bullet}(\vec x)$ to denote that $F$ does not fix the status of $R_j$ at $\vec x$ on all of $\XN$.
More precisely, we define the event
$\{R_j^{F(\bullet)}(\vec x)\neq R_j^{\bullet}(\vec x)\}=\bigcup_{\mathfrak{M}\in\XN}\{R_j^{F(\mathfrak{M})}(\vec x)\neq R_j^{\mathfrak{M}}(\vec x)\}$ for every $\vec x\in\Nb^{\ar(R_j)}$.

For any $j=1,\ldots,r$ and $s\finite\Nb$ with $0\leq |s|\leq \ar(R_j)$ and $0\leq i< \ar(R_j)-|s|$, we put
\begin{equation}\label{eq:L-j-i}
L_{j,i}(s,F):=\limsup_{n\to\infty}\frac{1}{n^{\ar(R_j)-|s|-i}}\sum_{\vec x\in[n]^{\ar(R_j)}: s\subseteq x,\ |\rng x\setminus \rng s|=\ar(R_j)-|s|-i}\mathbf{1}\{R_j^{F(\bullet)}(\vec x)\neq R_j^{\bullet}(\vec x)\}.\end{equation}
(Note that when $|s|=\ar(R_j)$, the condition $s\subseteq x$ immediately implies $s=x$ and, thus, $\rng x\setminus\rng s=\emptyset$.  In this case, we ignore the second condition on $|\rng x\setminus\rng s|$.  In all other cases, we require $|\rng x\setminus\rng s|>0$.)

From \eqref{eq:L-j-i}, we define
\begin{align}
\label{eq:L-j}L_j(s,F)&:=\max_{0\leq i<\ar(R_j)-|s|}L_{j,i}(s,F),\\
\label{eq:L}L(s,F)&:=\max_{1\leq j\leq r}L_j(s,F),\quad\text{and}\\
\label{eq:intersect}\Delta_F&:=\bigcap\{s\finite\Nb: L(s,F)>0\}.
\end{align}

The quantity in \eqref{eq:L-j-i} and the derivative objects in \eqref{eq:L-j}-\eqref{eq:intersect} detect the way in which $F$ acts on $\XN$.
As Examples \ref{ex:1}, \ref{ex:2}, and \ref{ex:3} illustrate, these quantities are designed so that our decomposition admits the interpretation described above in terms of first choosing a set of elements with multiplicity and then making an exchangeable transition by changing only relations involving the chosen multiset.

As we show in Proposition \ref{prop:key lemma}, exchangeability of the characteristic measure $\Lambda$ allows us to replace the limit superior in \eqref{eq:L-j-i} by a proper limit $\Lambda$-almost everywhere.
With the limit superior replaced by a proper limit, the quantity $L_{j,i}(s,F)$ in \eqref{eq:L-j-i} is the limiting fraction of locations $\vec x$ at which $F$ acts nontrivially on $\XN$, where $\vec x$ is any extension of $s$ such that $\rng x\setminus \rng s$ contains $\ar(R_j)-|s|-i$ distinct elements.
When $i=0$ and $s=\{s_1^{m_1},\ldots,s_k^{m_k}\}$, we choose $\vec x$ so that $x=s\cup\{y_1,\ldots,y_{\ar(R_j)-|s|}\}$, for $y_1,\ldots,y_{\ar(R_j)-|s|}$ all distinct and disjoint from $s$.
On the other extreme, with $i=\ar(R_j)-|s|-1$, we choose $\vec x$ so that $x=s\cup\{y\}$ for some $y\in\Nb\setminus s$.
For example, when $s=\{1^2,2^1\}$ and $\ar(R_j)=6$, the former considers all $\vec x$ of the form $(1,1,2,a,b,c)$ (and rearrangements thereof) for $a\neq b\neq c$ disjoint from $\{1,2\}$, while the latter considers all $\vec x$ of the form $(1,1,2,a,a,a)$ (and rearrangements thereof) for $a\not\in\{1,2\}$.

The quantity $L(s,F)$, therefore, records whether $F$ affects a limiting positive fraction of changes in locations $\vec x$ containing $s$.
For any $s$ and $0\leq i<\ar(R_j)-|s|$, there are $(n-|\rng(x)|)(n-|\rng(x)|-1)\cdots(n-|\rng(x)|-\ar(R_j)+|s|+i+1)$ choices of $y_1,\ldots,y_{\ar(R_j)-|s|-i}$.
We normalize by the asymptotically equivalent quantity $n^{\ar(R_j)-|s|-i}$ for a cleaner notation.

Some examples elucidate the interpretation of $\Delta_F$ as the common core of all locations at which $F$ acts nontrivially on $\XN$.
In general, we interpret $\Delta_F=s$ to mean that $F$ acts nontrivially on entries $\vec x$ containing $s$ and acts trivially elsewhere. 
As we show below, the behavior of exchangeable Feller processes on a space $\XN$ having $\odap$-DAP decomposes according to the structure of $\Delta_F$ over all $F\in\Lip(\XN)$.

\begin{example}\label{ex:1}
Consider a signature $\mathcal{L}=\{R_1\}$ with $\ar(R_1)=3$.
For this example we take $\XN=\lN$, so that $\XN$ consists of all $3$-ary relations on $\Nb$ and, thus, has $\odap$-DAP.
We define an exchangeable random conjugation invariant function $F:\XN\to\XN$ as follows.

Let $A=(A_{i,j})_{i,j\geq1}$ to be a random array with all entries independent Bernoulli random variables with success probability $1/2$.
Given $A$, we define $F\in\Lip(\XN)$ by
\[R_1^{F(\mathfrak{M})}((a,b,c))=\left\{\begin{array}{cc} A_{b,c},& a=1,\\
R_1^{\mathfrak{M}}((a,b,c)),& \text{otherwise}.\end{array}\right.\]
Therefore, $F$ replaces all entries $\vec x=(a,b,c)$ with $a=1$ by the outcome of a fair coin toss and leaves all entries with $a\neq1$ unchanged.
If we take $s=\{k\}$ for $k\neq1$, then $L_{1,0}(s,F)\leq \limsup_{n\to\infty}2n/n^2=0$ as $n\to\infty$ since the only changes can occur at locations of the form $(1,k,b)$ and $(1,b,k)$ for $b\in\Nb$.
With $s=\{1\}$, we have $L_{1,0}(\{1\},F)=1/3$ a.s.\ since $F$ acts nontrivially at any location of the form $\vec x=(1,b,c)$ but trivially at locations of form $(b,1,c)$ and $(b,c,1)$.
We likewise have that $L_{1,1}(\{1\},F)=1/3$ a.s.\ since $F$ acts nontrivially at locations of the form $(1,b,b)$ but trivially at $(b,1,b)$ and $(b,b,1)$.

We also have $L_{1,0}(\{1,a\},F)=1/3$ a.s., $L_{1,0}(\{a,b\},F)=0$ a.s.\ for $a,b\neq1$, and $L_{1,0}(\{a,b,c\},F)=1/3$ a.s.\ if and only if $1\in\{a,b,c\}$; whence,  $\Delta_F=\{1\}$.
\end{example}

\begin{example}\label{ex:2}
To see the need for considering all $0\leq i<\ar(R_j)-|s|$, we modify Example \ref{ex:1} as follows.
Once again, we have $\mathcal{L}=\{R_1\}$, $\ar(R_1)=3$, and $\XN=\lN$, but we now define $A=(A_i)_{i\geq1}$ as an i.i.d.\ sequence of Bernoulli random variables with success probability $1/2$.
Given $A$, we define a conjugation invariant $F\in\Lip(\XN)$ by
\[R_1^{F(\mathfrak{M})}((a,b,c))=\left\{\begin{array}{cc} A_c,& a=b=1,\\
R_1^{\mathfrak{M}}((a,b,c)), & \text{otherwise.}\end{array}\right.\]
Since $F$ only acts nontrivially on the diagonal strip $(1,1,c)$, we now have $L_{1,0}(\{1\},F)=0$ a.s.\ and $L_{1,1}(\{1\},F)=0$ a.s., but $L_{1,0}(\{1,1\},F)=1/3$ a.s.\ 
Thus, $\Delta_F=\{1,1\}$.
\end{example}

\begin{example}\label{ex:3}
With everything as defined in Example \ref{ex:2}, including the i.i.d.\ Bernoulli sequence $A=(A_i)_{i\geq1}$, we modify $F$ by
\[R_1^{F(\mathfrak{M})}((a,b,c))=\left\{\begin{array}{cc} A_b,& a=1,\ b=c,\\ R_1^{\mathfrak{M}}(a,b,c),& \text{otherwise.}\end{array}\right.\]
Here, $F$ only acts nontrivially on the diagonal strip $(1,b,b)$, giving $L_{1,0}(\{1\},F)=0$ a.s., $L_{1,1}(\{1\},F)=1/3$ a.s., and $\Delta_F=\{1\}$.
\end{example}

%
%
%

\begin{lemma}\label{lemma:Lambda-alpha-star}
For any $\alpha\vdash k$ with $1\leq k\leq\max_{1\leq j\leq r}\ar(R_j)$, let $\Lambda_{\alpha}$ be a measure on $\XN$ that satisfies \eqref{eq:regularity-Lambda}, is invariant with respect to all permutations $\sigma:\Nb\to\Nb$ that coincide with the identity on $s_{\alpha}$, satisfies 
\begin{equation}\label{eq:s-alpha}
\Delta_F=s_{\alpha}\quad\text{for }\Lambda_{\alpha}\text{-almost every }F\in\Lip(\XN),
\end{equation}
and, for all $j=1,\ldots,r$,
\begin{equation}\label{eq:substructures}
R_j^{F(\bullet)}(\vec x)=R_j^{\bullet}(\vec x)\quad\text{for all }\vec x\in\Nb^{\ar(R_j)}\text{ such that }s_{\alpha}\not\subseteq x
\end{equation}
for $\Lambda_{\alpha}$-almost every $F\in\Lip(\XN)$.
Then
\begin{equation}\label{eq:mu-alpha-star}
\Lambda_{\alpha}^*(\cdot):=\sum_{s\finite\Nb: \mathbf{m}^{\downarrow}(s)=\alpha}\Lambda_{\alpha}(\{F\in\Lip(\XN): F^{\phi_{s,\alpha}^{-1}}\in\cdot\}),
\end{equation}
satisfies \eqref{eq:regularity-Lambda} and is exchangeable, where we define $\phi_{s,\alpha}:\Nb\to\Nb$ as the bijection that sends $i\mapsto s_i^{\downarrow}$, for $(s_i^{\downarrow})_{i\geq1}$ the ordering of elements in $s^{\downarrow}$ defined above, and $F^{\phi^{-1}_{s,\alpha}}(\mathfrak{M})=F(\mathfrak{M}^{\phi_{s,\alpha}})$.

\end{lemma}

\begin{proof}
For any permutation $\sigma:\Nb\to\Nb$, $F^{\sigma}=\idN$ if and only if $F=\idN$.
Since $\Lambda_{\alpha}$ fulfills the lefthand side of \eqref{eq:regularity-Lambda} and there are countably many multisets $s\finite\Nb$ with $\mathbf{m}^{\downarrow}(s)=\alpha$, $\Lambda_{\alpha}^*$ also satisfies the lefthand side of \eqref{eq:regularity-Lambda}.

For every $n\in\Nb$, \eqref{eq:s-alpha} and \eqref{eq:substructures} together imply that 
\begin{eqnarray*}
\lefteqn{\Lambda_{\alpha}^*(\{F\in\Lip(\XN): F^{[n]}\neq\idn\})=}\\
&=&\sum_{s\finite\Nb: \mathbf{m}^{\downarrow}(s)=\alpha}\Lambda_{\alpha}(\{F\in\Lip(\XN): F^{\phi^{-1}_{s,\alpha}}|_{[n]}\neq\idn\})\\
&=&\sum_{s\finite[n]: \mathbf{m}^{\downarrow}(s)=\alpha}\Lambda_{\alpha}(\{F\in\Lip(\XN): F^{\phi^{-1}_{s,\alpha}}|_{[n]}\neq\idn\})\\
&\leq&n^{|s_{\alpha}|}\Lambda_{\alpha}(\{F\in\Lip(\XN): F^{[n]}\neq\idn\})\\
&<&\infty,
\end{eqnarray*}
since $\{F^{[n]}\neq\idn\}$ has positive measure under $\Lambda_{\alpha}^*$ only if $\Delta_F\finite[n]$.
Exchangeability is plain since $\Delta_{\sigma F\sigma^{-1}}$ has partition type $\alpha$ for every permutation $\sigma:\Nb\to\Nb$ and every $s\finite\Nb$ with $\mathbf{m}^{\downarrow}(s)=\alpha$ and $\Delta_F=s$.

\end{proof}


\begin{prop}\label{prop:key lemma}
Let $\Lambda$ be an exchangeable measure that satisfies \eqref{eq:regularity-Lambda} and for which $\Lambda$-almost every $F\in\Lip(\XN)$ is conjugation invariant and has $\Delta_F\neq\emptyset$.
Then $\Lambda=\sum_{k=1}^{\max_j\ar(R_j)}\sum_{\alpha\vdash k}\Lambda_{\alpha}^*$, with $\Lambda_{\alpha}^*$ as defined in \eqref{eq:mu-alpha-star} for each $\alpha\vdash k$, for some unique measure $\Lambda_{\alpha}$ which satisfies \eqref{eq:regularity-Lambda}, \eqref{eq:s-alpha}, and \eqref{eq:substructures} and is invariant with respect to all permutations that coincide with the identity on $s_{\alpha}$.
\end{prop}

\begin{proof}
By assumption that $\Delta_F\neq\emptyset$ for $\Lambda$-almost every $F\in\XN$ and the fact that $\Delta_F$ is well defined for $\Lambda$-almost every $F\in\Lip(\XN)$,  we can decompose $\Lambda$ according to
\[\Lambda=\sum_{k=1}^{\max_j\ar(R_j)}\sum_{\alpha\vdash k}\Lambda\mathbf{1}_{\{F\in\Lip(\XN): \mathbf{m}^{\downarrow}(\Delta_F)=\alpha\}}.\]
For each $\alpha\vdash k$, we claim that $\Lambda_{\alpha}:=\Lambda\mathbf{1}_{\{F\in\Lip(\XN): \Delta_F=s_{\alpha}\}}$ is the unique measure satisfying the above conditions for which $\Lambda\mathbf{1}_{\{F\in\Lip(\XN): \mathbf{m}^{\downarrow}(\Delta_F)=\alpha\}}=\Lambda_{\alpha}^*$.

First, we note that any $F\in\Lip(\XN)$ for which $\mathbf{m}^{\downarrow}(\Delta_F)=\alpha\vdash k$ implies that $L(s',F)=0$ for every $s'\finite\Nb$ with $|s'|<k$, since the definition of $\Delta_F$ in \eqref{eq:intersect} implies that $\Delta_F\subseteq s'$ and, thus, $|\Delta_F|<k$ and $\mathbf{m}^{\downarrow}(\Delta_F)$ partitions $|\Delta_F|<k$, a contradiction.

\begin{claim}
For any multiset $s\finite\Nb$ and $F\in\Lip(\XN)$, $\Delta_F\not\subseteq s$ implies $L(s,F)=0$.
Conversely, $L(s,F)>0$ implies $\Delta_F\subseteq s$.
\end{claim}

\begin{claimproof}
This is clear since $\Delta_F\not\subseteq s$ implies $R_j^{F(\bullet)}(\vec x)=R_j^{\bullet}(\vec x)$ for all $\vec x$ with $s\subseteq x$ and all $j=1,\ldots,r$.
\end{claimproof}

For $F\in\XN$, $\varepsilon,\delta>0$, $j=1,\ldots,r$, and $|\Delta_F|\leq m\leq\max_{1\leq i\leq r}\ar(R_i)$, we define
\begin{align*}
A_{F}^{j,m}(\varepsilon)&:=\{s'\finite\Nb: |s'|=m \text{ and }L_j(s',F)\geq\varepsilon\},\\
|A_F^{j,m}(\varepsilon)|&:=\limsup_{n\to\infty}{n^{-m}}\sum_{s'\finite[n]: |s'|=m}\mathbf{1}\{s'\in A_F^{j,m}(\varepsilon)\},\quad\text{and}\\
V^{j,m}(\varepsilon,\delta)&:=\{F\in\Lip(\XN): |A^{j,m}_F(\varepsilon)|\geq\delta\}.
\end{align*}


The function of the above quantities is to establish that if $|A_F^{j,m}(\varepsilon)|>0$ then the collection of $s\in A_F^{j,m}(\varepsilon)$ must have a nonempty intersection $s^*$ such that $L_j(s^*,F)>0$. 
If $|A_F^{j,m}(\varepsilon)|=0$, then we show that $A_F^{j,m}(\varepsilon)$ must be finite, in which case the intersection $s^*$ corresponds to the common core of all changes caused by $F$.
We rule out other cases as a consequence of exchangeability and the righthand side of \eqref{eq:regularity-Lambda}.

\begin{claim}
For $\Lambda$-almost every $F\in\Lip(\XN)$, we can replace limits superior by proper limits in our definitions of $L_{j,i}(s',F)$ and $|A_F^{j,m}(\varepsilon)|$ above.
\end{claim}
\begin{claimproof}
  To see this, we let $\Lambda_n:=\Lambda\mathbf{1}_{\{F\in\Lip(\XN): F^{[n]}\neq\idn\}}$, $n\in\Nb$, denote the restriction of $\Lambda$ to the event $F^{[n]}\neq\idn$.
By the righthand side of \eqref{eq:regularity-Lambda} and exchangeability of $\Lambda$, $\Lambda_n$ is a finite measure which is invariant with respect to all permutations that coincide with the identity on $[n]$.
For every $n\geq1$, we define the {\em $n$-shift} as the injection $\overrightarrow{\sigma}_n:\Nb\to\Nb$ for which $\overrightarrow{\sigma}_n(k)=k+n$ for every $k\geq1$.
For any conjugation invariant $F\in\Lip(\XN)$, we define the {\em image of $F$ by $\overrightarrow{\sigma}_n$}, alternatively the {\em $n$-shift of $F$}, by $\overleftarrow{F}_n:\XN\to\XN$, which satisfies $\overleftarrow{F}_n(\mathfrak{N}):=F(\mathfrak{M})^{\overrightarrow{\sigma}_n}$ for any $\mathfrak{M}\in\XN$ satisfying $\mathfrak{M}^{\overrightarrow{\sigma}_n}=\mathfrak{N}$.
This is well defined $\Lambda$-almost everywhere by Proposition \ref{prop:conj-inv} and the assumption that $\Lambda$-almost every $F\in\Lip(\XN)$ is conjugation invariant.

We write $\overleftarrow{\Lambda}_n$ to denote the image measure of $\Lambda_n$ by the $n$-shift, that is,
\[\overleftarrow{\Lambda}_n(\cdot):=\Lambda_n(\{F\in\Lip(\XN): \overleftarrow{F}_n\in\cdot\}),\]
which is an exchangeable, finite measure on $\Lip(\XN)$ and, therefore, is proportional to an exchangeable probability measure $\mu_n$ on $\Lip(\XN)$.
Let $\Phi\sim\mu_n$ be a random function distributed according to $\mu_n$.
For any $j=1,\ldots,r$ and $s'\finite\Nb$, with $|s'|<\ar(R_j)$ and $0\leq i< \ar(R_j)-|s'|$, we consider the random variables
\[Z_m^i:=\frac{1}{(m-|\rng s'|)^{\downarrow(\ar(R_j)-|s'|-i)}}\sum_{\vec x\in[m]^{\ar(R_j)}: s'\subseteq x,\ |\rng x\setminus \rng s'|=\ar(R_j)-|s'|-i}\mathbf{1}\{R_j^{\Phi(\bullet)}(\vec x)\neq R_j^{\bullet}(\vec x)\},\]
for $m\geq1$,
where $m^{\downarrow j}:=m(m-1)\cdots(m-j+1)$ is the falling factorial and recall that $x$ denotes the multiset determined by forgetting the ordering of elements in $\vec x$.

We define the reverse filtration $(\mathcal{F}_m^i)_{m\geq1}$ by $\mathcal{F}_m^i:=\sigma\langle Z_{m+1}^i,Z_{m+2}^i,\ldots\rangle$ so that $(Z_m^i)_{m\geq1}$ is a reverse martingale with respect to $(\mathcal{F}_m^i)_{m\geq1}$.
By the reverse martingale convergence theorem, $Z_m^i$ converges $\mu_n$-almost surely as $m\to\infty$, allowing us to replace the limit superior with a proper limit in \eqref{eq:L-j-i} for $\overleftarrow{\Lambda}_n$-almost every $F\in\Lip(\XN)$.
Since this limit depends on $F$ only through its $n$-shift, for every $n\geq1$, it follows that this limit exists for $\Lambda_n$-almost every $F\in\Lip(\XN)$, for all $n\geq1$.
Moreover, we note that the events $\{F\in\Lip(\XN): F^{[n]}\neq\idn\}$ increase to $\{F\in\Lip(\XN): F\neq\idN\}$ as $n\to\infty$ in the sense that
\[\{F\in\Lip(\XN): F\neq\idN\}=\bigcup_{n\geq1}\{F\in\Lip(\XN): F^{[n]}\neq\idn\},\]
implying that $\Lambda_n\uparrow\Lambda\mathbf{1}_{\{F\in\Lip(\XN): F\neq\idN\}}=\Lambda$ by the monotone convergence theorem and the lefthand side of \eqref{eq:regularity-Lambda}.
It follows that $\lim_{m\to\infty}Z_m^i=Z^i$ 
exists $\Lambda$-almost everywhere for all $s'\finite\Nb$ and all $j=1,\ldots,r$.
We can replace the upper limit with a proper limit in the definition of $L_{j,i}(s,F)$ by noting that $(m-|\rng s'|)^{\downarrow \ar(R_j)-|s'|-i}\sim m^{\ar(R_j)-|s'|-i}$ and, thus, $Z^i=L_{j,i}(s,F)$ for $\Lambda$-almost every $F$.

An analogous argument allows us to replace the limit superior by a proper limit in our definition of $|A_F^{j,m}(\varepsilon)|$ for $\Lambda$-almost every $F$.
\end{claimproof}

\begin{claim}\label{claim:3}
$\Lambda(V^{j,m}(\varepsilon,\delta))=0$ for all $\varepsilon,\delta>0$.  In other words, $|A^{j,m}_F(\varepsilon)|=0$ for $\Lambda$-almost every $F\in\Lip(\XN)$.
\end{claim}
\begin{claimproof}
By definition, $\mathbf{m}^{\downarrow}(\Delta_F)=\alpha\vdash k$ implies that, for every $s'\subsetneq\Delta_F$ with $|s'|=l<k$ and every $0\leq i<\ar(R_j)-|s'|$,
\begin{eqnarray*}
\lefteqn{0\leq\liminf_{n\to\infty}\frac{1}{n^{\ar(R_j)-l-i}}\sum_{\vec x\in[n]^{\ar(R_j)}: s'\subseteq x,\ |\rng x\setminus\rng s'|=\ar(R_j)-|s'|-i}\mathbf{1}\{R_j^{F(\bullet)}(\vec x)\neq R_j^{\bullet}(\vec x)\}\leq}\\
&\leq &\limsup_{n\to\infty}\frac{1}{n^{\ar(R_j)-l-i}}\sum_{\vec x\in[n]^{\ar(R_j)}: s'\subseteq x,\ |\rng x\setminus \rng s'|=\ar(R_j)-|s'|-i}\mathbf{1}\{R_j^{F(\bullet)}(\vec x)\neq R_j^{\bullet}(\vec x)\}\\
&=&0.
\end{eqnarray*}
By assumption, $\Lambda$-almost every $F$ has $\mathbf{m}^{\downarrow}(\Delta_F)=\alpha\vdash k$; whence, $L(s',F)=0$ for all $s'\subsetneq \Delta_F$ for $\Lambda$-almost every $F\in\Lip(\XN)$.

Also, for any $F\in\Lip(\XN)$ and  $l<|\Delta_F|$, we can write
\[A_F^{j,m}(\varepsilon)=\bigcup_{\tilde{s}\subseteq \Delta_F: |\tilde{s}|=l}\{s'\finite\Nb: \tilde{s}\subseteq s'\text{ and }s'\in A_F^{j,m}(\varepsilon)\};\]
whence, for every $\varepsilon,\delta>0$ and $\Lambda$-almost every $F\in V^{j,m}(\varepsilon,\delta)$,
\begin{eqnarray*}
\lefteqn{\delta\leq\limsup_{n\to\infty}n^{-m}\sum_{s'\finite[n]: |s'|=m}\mathbf{1}\{s'\in A_F^{j,m}(\varepsilon)\}=}\\
&=&\lim_{n\to\infty}n^{-m}\sum_{s'\finite[n]: |s'|=m}\mathbf{1}\{s'\in A_F^{j,m}(\varepsilon)\}=\\
&=&\lim_{n\to\infty}\lim_{p\to\infty}n^{-l}p^{l-m}\sum_{\tilde{s}\finite[n]: |\tilde{s}|=l}\sum_{s'\finite[p]: |s'|=m-l}\mathbf{1}\{s'\oplus \tilde{s}\in A_F^{j,m}(\varepsilon)\}\\
&=&\lim_{n\to\infty}n^{-l}\sum_{\tilde{s}\finite[n]: |\tilde{s}|=l}\lim_{p\to\infty}p^{l-m}\sum_{s'\finite[p]: |s'|=m-l}\mathbf{1}\{s'\oplus \tilde{s}\in A_F^{j,m}(\varepsilon)\},\\
\end{eqnarray*}
where for multisets $a=\{i^{m_i(a)}\}_{i\geq1}$ and $b=\{i^{m_i(b)}\}_{i\geq1}$ we write $a\oplus b:=\{i^{m_i(a)+m_i(b)}\}_{i\geq1}$.
The above calculation implies
\begin{equation}\label{eq:tilde-s}\limsup_{p\to\infty}\frac{1}{p^{m-l}}\sum_{s'\finite[p]: |s'|=m-l}\mathbf{1}\{s'\oplus \tilde{s}\in A_F^{j,m}(\varepsilon)\}\geq\delta>0\end{equation}
for some $\tilde{s}\finite\Nb$ with $|\tilde{s}|=l$.
By definition of $A_F^{j,m}(\varepsilon)$, for $\Lambda$-almost every $F\in V^{j,m}(\varepsilon,\delta)$ there is some $\tilde{s}$ with $|\tilde{s}|=l$ as in \eqref{eq:tilde-s} and some $0\leq i<\ar(R_j)-|\tilde{s}|$ such that (with $\widehat{\sum}$ indicating the sum over $\{\vec x\in[n]^{\ar(R_j)}: s'\subseteq x,\ |\rng x\setminus\rng\tilde{s}|=\ar(R_j)-|\tilde{s}|-i\}$)
\begin{eqnarray*}
\lefteqn{\limsup_{n\to\infty}\frac{1}{n^{\ar(R_j)-l-i}}\sum_{\vec x\in[n]^{\ar(R_j)}: \tilde{s}\subseteq x,\ |\rng x\setminus\rng\tilde{s}|=\ar(R_j)-|\tilde{s}|-i}\mathbf{1}\{R^{F(\bullet)}_j(\vec x)\neq R^{\bullet}_j(\vec x)\}\geq}\\
&\geq&\limsup_{n\to\infty}\frac{1}{n^{\ar(R_j)-l-i}}\sum_{s'\finite[n]: |s'|=m, \tilde{s}\subseteq s'}\frac{1}{\binom{\ar(R_j)}{m-|\tilde{s}|}}\widehat{\sum}\mathbf{1}\{R^{F(\bullet)}_j(\vec x)\neq R^{\bullet}_j(\vec x)\}\\
&=&\frac{1}{\binom{\ar(R_j)}{m-|\tilde{s}|}}\lim_{n\to\infty}\lim_{p\to\infty}\frac{1}{n^{m-l}}\frac{1}{p^{\ar(R_j)-m-i}}\sum_{s'\finite[n]: |s'|=m, \tilde{s}\subseteq s'}\widehat{\sum}\mathbf{1}\{R^{F(\bullet)}_j(\vec x)\neq R^{\bullet}_j(\vec x)\}\\
&=&\frac{1}{\binom{\ar(R_j)}{m-|\tilde{s}|}}\lim_{n\to\infty}\frac{1}{n^{m-l}}\sum_{s'\finite[n]: |s'|=m, \tilde{s}\subseteq s'}\lim_{p\to\infty}\frac{1}{p^{\ar(R_j)-m-i}}\widehat{\sum}\mathbf{1}\{R^{F(\bullet)}_j(\vec x)\neq R^{\bullet}_j(\vec x)\}\\
&\geq&\frac{1}{\binom{\ar(R_j)}{m-|\tilde{s}|}}\lim_{n\to\infty}\frac{1}{n^{m-l}}\sum_{s'\finite[n]: |s'|=m, \tilde{s}\subseteq s'}\varepsilon\mathbf{1}\{s'\in A_F^{j,m}(\varepsilon)\}\\
&=&\frac{1}{\binom{\ar(R_j)}{m-|\tilde{s}|}}\lim_{n\to\infty}\frac{1}{n^{m-l}}\sum_{s''\finite[n]: |s''|=m-l}\varepsilon\mathbf{1}\{s''\oplus\tilde{s}\in A_F^{j,m}(\varepsilon)\}\\
&\geq&\varepsilon{\delta}\frac{1}{\binom{\ar(R_j)}{m-|\tilde{s}|}}\\
&>&0,
\end{eqnarray*}
which contradicts the assumption $\tilde{s}\subsetneq \Delta_F$.
It follows that $|A_F^{j,m}(\varepsilon)|=0$ for all $\varepsilon>0$ for $\Lambda$-almost every on $F\in\Lip(\XN)$.
(In the above string, the first inequality follows by noting that for any $\vec x\in[n]^{\ar(R_j)}$ there are at most $\binom{\ar(R_j)}{m-|\tilde{s}|}$ unique multisets of $[n]$ embedded within $\vec x$, allowing for the possibility that $\vec x$ with $\tilde{s}\subseteq s'\subseteq x$ is counted $\binom{\ar(R_j)}{m-|\tilde{s}|}$ times instead of only once. 
The second line is permitted since we have shown that the limits exist for $\Lambda$-almost every $F\in\Lip(\XN)$.
The inequality in the fourth line is plain by our definition of the set $A_F^{j,m}(\varepsilon)$ above.
The fifth line is a rewriting of the fourth line in equivalent form.
The sixth line follows by our choice of $\tilde{s}$ that satisfies \eqref{eq:tilde-s}.)
\end{claimproof}



Claim \ref{claim:3} establishes that $|A_F^{j,m}(\varepsilon)|=0$ for $\Lambda$-almost every $F\in\Lip(\XN)$, for all $\varepsilon>0$, which allows the possibility that either $\#A_F^{j,m}(\varepsilon)=\infty$ or $\#A_F^{j,m}(\varepsilon)<\infty$.
We observe further that $\Lambda$-almost every $F\in\Lip(\XN)$ satisfies either $\#A_F^{j,m}(\varepsilon)<\infty$ or if $\#A_F^{j,m}(\varepsilon)=\infty$ then there is some $|\Delta_F|\leq m'<m$ such that $\#A_F^{j,m'}(\varepsilon)<\infty$.

Let $m^*=\max_{1\leq j\leq r}\ar(R_j)$ and, without loss of generality, suppose $\alpha\vdash k$ is such that $\rng s_{\alpha}\subseteq[m^*]$.
For any $s'\supseteq s_{\alpha}$ with $\rng s'\subseteq[m^*]$ and $|s'|\leq m^*$ and any $j=1,\ldots,r$, \eqref{eq:regularity-Lambda}  implies
\begin{eqnarray*}
\lefteqn{\infty>\Lambda(\{F\in\Lip(\XN): F^{[m^*]}\neq\id_{[m^*]}\})\geq}\\
&&\quad\geq\Lambda(\{F\in\Lip(\XN): F^{[m^*]}\neq\id_{[m^*]}, L_j(s',F)\geq\varepsilon\})\end{eqnarray*}
and, furthermore,
\begin{eqnarray*}
\lefteqn{\Lambda(\{F\in\Lip(\XN): F^{[m^*]}\neq\id_{[m^*]}\})\geq}\\
&\geq&\Lambda(\{F\in\Lip(\XN): F^{[m^*]}\neq\id_{[m^*]}, L_j(s',F)\geq\varepsilon\})\\
&=&\int_{\Lip(\XN)}\mathbf{1}\{F^{[m^*]}\neq\id_{[m^*]}\}\mathbf{1}\{L_j(s',F)\geq\varepsilon\}\Lambda(dF)\\
&=&\int_{\Lip(\XN)}\mathbf{1}\{F^{s'\cup s^*}\neq\id_{s'\cup s^*}\}\mathbf{1}\{L_j(s',F)\geq\varepsilon\}\Lambda(dF)
\end{eqnarray*}
for all $s^*\subseteq\Nb\setminus\rng s'$ with $|s^*|=m^*-|\rng s'|$ by exchangeability of $\Lambda$, where $F^{s'\cup s^*}$ is the restriction of $F$ to a function $\mathcal{X}_{s'\cup s^*}\to\mathcal{X}_{s'\cup s^*}$.
Thus,
\begin{eqnarray*}
\lefteqn{\Lambda(\{F\in\Lip(\XN): F^{[m^*]}\neq\id_{[m^*]}\})\geq}\\
&\geq&\Lambda(\{F\in\Lip(\XN): F^{[m^*]}\neq\id_{[m^*]}, L_j(s',F)\geq\varepsilon\})\\
&=&\frac{1}{\binom{n-|\rng s'|}{m^*-|\rng s'|}}\widehat{\sum}\int_{\Lip(\XN)}\mathbf{1}\{F^{s'\cup s^*}\neq\id_{s'\cup s^*}\}\mathbf{1}\{L_j(s',F)\geq\varepsilon\}\Lambda(dF)\\
&=&\int_{\Lip(\XN)}\mathbf{1}\{L_j(s',F)\geq\varepsilon\}\frac{1}{\binom{n-|\rng s'|}{m^*-|\rng s'|}}\widehat{\sum}\mathbf{1}\{F^{s'\cup s^*}\neq\id_{s'\cup s^*}\}\Lambda(dF),
\end{eqnarray*}
where we write $\widehat{\sum}$ for the sum over $\{s^*\subseteq[n]\setminus\rng s': |s^*|=m^*-|\rng s'|\}$.
Since the righthand side is constant we have
\begin{eqnarray*}
\lefteqn{\Lambda(\{F\in\Lip(\XN): F^{[m^*]}\neq\id_{[m^*]}\})\geq}\\
&\geq&\Lambda(\{F\in\Lip(\XN): F^{[m^*]}\neq\id_{[m^*]}, L_j(s',F)\geq\varepsilon\})\\
&=&\liminf_{n\to\infty}\int_{\Lip(\XN)}\mathbf{1}\{L_j(s',F)\geq\varepsilon\}\frac{1}{\binom{n-|\rng s'|}{m^*-|\rng s'|}}\widehat{\sum} \mathbf{1}\{F^{s'\cup s^*}\neq\id_{s'\cup s^*}\}\Lambda(dF)\\
&\geq&\int_{\Lip(\XN)}\mathbf{1}\{L_j(s',F)\geq\varepsilon\}\liminf_{n\to\infty}\frac{1}{\binom{n-|\rng s'|}{m^*-|\rng s'|}}\widehat{\sum}\mathbf{1}\{F^{s'\cup s^*}\neq\id_{s'\cup s^*}\}\Lambda(dF)\\
&\geq&\int_{\Lip(\XN)}\mathbf{1}\{L_j(s',F)\geq\varepsilon\}\mathbf{1}\{\lim_{n\to\infty}\frac{1}{\binom{n-|\rng s'|}{m^*-|\rng s'|}}\widehat{\sum}\mathbf{1}\{F^{s'\cup s^*}\neq\id_{s'\cup s^*}\}\text{ exists}\}\times\\
&&\quad\times\lim_{n\to\infty}\frac{1}{\binom{n-|\rng s'|}{m^*-|\rng s'|}}\widehat{\sum}\mathbf{1}\{F^{s'\cup s^*}\neq\id_{s'\cup s^*}\}\Lambda(dF)\\
&\geq&\varepsilon\Lambda(\{F\in\Lip(\XN): L_j(s',F)\geq\varepsilon\}).
\end{eqnarray*}
(The inequality in passing from the third line to the fourth line is a consequence of Fatou's lemma.)

The final inequality above together with \eqref{eq:regularity-Lambda} implies 
\[\Lambda(\{F\in\Lip(\XN): L_j(s',F)\geq\varepsilon\})<\infty\quad\text{for every }s'\finite\Nb.\]
For any $s'\finite\Nb$, countable additivity and \eqref{eq:regularity-Lambda} implies further that
\begin{eqnarray*}
\lefteqn{\Lambda(\{F\in\Lip(\XN): L_j(s',F)\geq\varepsilon\})\geq}\\
&\geq&\Lambda(\{F\in\Lip(\XN): L_j(s',F)\geq\varepsilon, \#A_F^{j,m}(\varepsilon)<\infty\})\\
&=&\sum_{k=1}^{\infty}\Lambda(\{F\in\Lip(\XN): L_j(s',F)\geq\varepsilon,\ \#A_F^{j,m}(\varepsilon)=k\}).
\end{eqnarray*}
We partition $A_F^{j,m}(\varepsilon)$ according to
\[A_F^{j,m}(\varepsilon)=\{s\in A_F^{j,m}(\varepsilon): \rng s\subseteq \rng s'\}\cup\{s\in A_F^{j,m}(\varepsilon): \rng s\not\subseteq\rng s'\}.\]
By exchangeability, $\Lambda$ assigns the same measure to the set of $F$ for which $L_j(s,F)\geq\varepsilon$ and $\rng s\not\subseteq\rng s'$ and the set of $F$ for which $L_j(s^{\sigma},F)\geq\varepsilon$ for any permutation $\sigma:\Nb\to\Nb$ that coincides with the identity on $\rng s'$, where $s^{\sigma}$ is the image of $s$ under the permutation $\sigma$.
As there are infinitely many such $s^{\sigma}$ whenever $\rng s\not\subseteq\rng s'$, condition \eqref{eq:regularity-Lambda} and $L_j(s',F)\geq\varepsilon$ force $\rng s\subseteq\rng s'$ for all $s\in A_F^{j,m}(\varepsilon)$; whence, the event $\#A_F^{j,m}(\varepsilon)<\infty$ implies that there exists $S\subseteq\Nb$ such that $\rng s=S$ for all $s\in A_F^{j,m}(\varepsilon)$.
Any such $S\subseteq\Nb$ necessarily satisfies $\rng\Delta_F\subseteq S$.

On the other hand, if $\#A_F^{j,m}(\varepsilon)=\infty$ for $m=|\Delta_F|$, then
\[\Delta_F\subseteq\bigcap\{s\finite\Nb: s\in A_F^{j,m}(\varepsilon)\}\]
implies $|\Delta_F|<|\Delta_F|$, a contradiction.
We can, therefore, assume $m>|\Delta_F|$ and define 
\[D_F:=\bigcap\{s'\finite\Nb: s'\in A_F^{j,m}\}\supseteq\Delta_F,\] 
which must satisfy $m'=|D_F|\geq|\Delta_F|>1$.
If we still have $\#A_F^{j,m'}=\infty$, then we define 
\[D'_F:=\bigcap\{s''\finite\Nb: s''\in A_F^{j,m'}\}\supseteq\Delta_F\]
and proceed as above.
If $\#A_F^{j,m''}<\infty$, then the above argument implies that $\rng s'=\rng s$ for all $s,s'\in A_F^{j,m''}$.



Now, for any $n\geq\max_j\ar(R_j)$, we define $\Lambda_n$ to be the restriction of $\Lambda$ to the event $\{F\in\Lip(\XN): F^{[n]}\neq\idn\}$ and we write $\overleftarrow{\Lambda}_n$ to denote the image of $\Lambda_n$ by the $n$-shift defined above, that is,
\[\overleftarrow{\Lambda}_n(\{F\in\Lip(\XN): F\in\cdot\})=\Lambda(\{F\in\Lip(\XN): F^{[n]}\neq\idn, \overleftarrow{F}_n\in\cdot\}),\]
which is finite and exchangeable by \eqref{eq:regularity-Lambda}.
We can regard $\overleftarrow{\Lambda}_n$ as a constant multiple of an exchangeable probability measure $\mu_n$ on $\Lip(\XN)$.  
By the Aldous--Hoover theorem, we have
\[|A_F^{j,m}(\varepsilon)|=\limsup_{k\to\infty}k^{-m}\sum_{s'\finite[k]: |s'|=m}\mathbf{1}\{L_j(s',F)\geq\varepsilon\}=0\quad\mu_n\text{-a.s.},\]
implying that $L_j(s',F)=0$ for all $s'\finite\Nb$ $\mu_n$-a.s.

Finally, the righthand side of condition \eqref{eq:regularity-Lambda} implies that, for every $j=1,\ldots,r$, the set 
\[\Delta_F^{(j)} = \bigcap\{s\finite\Nb: L_j(s,F)\geq\varepsilon\}\]
either has $\rng\Delta_F^{(j)}=\Delta_F$ or else $L_j(s,F)=0$ for all $s\finite\Nb$, for $\Lambda$-almost every $F\in\Lip(\XN)$.
For if there were positive measure assigned to the event $\Delta_F^{(j)},\Delta_F^{(j')}\neq\emptyset$ and $\rng\Delta_F^{(j)}\neq\rng\Delta_F^{(j')}$ then exchangeability would assign infinite measure to an event of the form $\{F^{[n]}\neq\idn\}$, contradicting the righthand side of \eqref{eq:regularity-Lambda}.

Exchangeability allows us to proceed as in the above argument under the assumption that $\rng\Delta_F=[n]$, from which we deduce that $R_j^{F(\bullet)}(\vec x)=R_j^{\bullet}(\vec x)$ for all $\vec x\in\Nb^{\ar(R_j)}$ such that $\Delta_F\not\subseteq x$ for $\Lambda$-almost every $F\in\Lip(\XN)$.

We now define $\Lambda_{\alpha}:=\Lambda\mathbf{1}_{\{F\in\Lip(\XN): \Delta_F=s_{\alpha}\}}$ and $\Lambda_{\alpha}^*$ as in \eqref{eq:mu-alpha-star}.
For any $s\finite\Nb$ with $\mathbf{m}^{\downarrow}(s)=\alpha$, let $\phi_{s,\alpha}$ be the bijection defined in Lemma \ref{lemma:Lambda-alpha-star}.
By definition, $\Lambda_{\alpha}$ satisfies \eqref{eq:regularity-Lambda}, \eqref{eq:s-alpha}, and \eqref{eq:substructures}.
Moreover, for $\Lambda$-almost every $F\in\Lip(\XN)$, $F^{\phi_{s,\alpha}^{-1}}$ also satisfies \eqref{eq:s-alpha} and \eqref{eq:substructures}.
Exchangeability of $\Lambda$ implies that $\Lambda\mathbf{1}_{\{F\in\Lip(\XN): \mathbf{m}^{\downarrow}(\Delta_F)=\alpha\}}=\Lambda_{\alpha}^*$.
The proof is complete.

\end{proof}

\begin{theorem}\label{thm:Levy-Ito}
Let $\mathbf{X}$ be a continuous time, exchangeable Markov process on a Fra\"iss\'e space $\XN$ that has $\odap$-DAP.
Then the characteristic measure $\Lambda$ from Theorem \ref{thm:continuous-Lambda} can be chosen so that it concentrates on conjugation invariant functions $\XN\to\XN$ and decomposes uniquely as
\begin{equation}\label{eq:decomp}
\Lambda=\Lambda_{\emptyset}^*+\sum_{k=1}^{\max_j\ar(R_j)}\sum_{\alpha\vdash k}\Lambda_{\alpha}^*,\end{equation}
where $\Lambda_{\emptyset}^*$ is an exchangeable measure on $\Lip(\XN)$ satisfying \eqref{eq:regularity-Lambda} and $\Lambda_{\emptyset}^*$-almost every $F\in\Lip(\XN)$ has $\Delta_F=\emptyset$ and for every $\alpha\vdash k$, $1\leq k\leq \max_j\ar(R_j)$, $\Lambda_{\alpha}$ is invariant with respect to permutations that coincide with the identity on $s_{\alpha}$ and satisfies \eqref{eq:regularity-Lambda}, \eqref{eq:s-alpha}, and \eqref{eq:substructures}, and $\Lambda_{\alpha}^*$ is as defined in \eqref{eq:mu-alpha-star}.
\end{theorem}

\begin{proof}
We can decompose $\Lambda$ into mutually singular measures by
\[\Lambda=\Lambda\mathbf{1}_{\{F\in\Lip(\XN): \Delta_F=\emptyset\}}+\Lambda\mathbf{1}_{\{F\in\Lip(\XN): \Delta_F\neq\emptyset\}}.\]
We write $\Lambda_{\emptyset}$ to denote the lefthand term above, which is exchangeable and satisfies \eqref{eq:regularity-Lambda} by our assumptions on $\Lambda$.
The righthand term decomposes further as
\[\Lambda\mathbf{1}_{\{F\in\Lip(\XN): \Delta_F\neq\emptyset\}}=\sum_{k=1}^{\max_j\ar(R_j)}\sum_{\alpha\vdash k}\Lambda\mathbf{1}_{\{F\in\Lip(\XN): \mathbf{m}^{\downarrow}(\Delta_F)=\alpha\}}.\]
Under these circumstances, each component $\Lambda\mathbf{1}_{\{F\in\Lip(\XN): \mathbf{m}^{\downarrow}(\Delta_F)=\alpha\}}$ is exchangeable and satisfies the conditions of Proposition \ref{prop:key lemma}.
The decomposition of $\Lambda$ in \eqref{eq:decomp} follows.
\end{proof}

The representation simplifies when $\XN$ has additional structure.

\begin{cor}
Let $\mathbf{X}$ be a continuous time, exchangeable Markov process on a Fra\"iss\'e space $\XN$ that has $\odap$-DAP and is anti-reflexive, that is, $\vec x\in R_j^{\mathfrak{M}}$ implies $\mathbf{m}^{\downarrow}(x)=(1,\ldots,1)$ for every $\mathfrak{M}\in\XN$.
Then the characteristic measure $\Lambda$ from Theorem \ref{thm:Levy-Ito} decomposes uniquely as
\begin{equation}\label{eq:decomp}
\Lambda=\Lambda_{\emptyset}^*+\sum_{k=1}^{\max_j\ar(R_j)}\Lambda_{k}^*,\end{equation}
where $\Lambda_{\emptyset}^*$ is an exchangeable measure on $\Lip(\XN)$ satisfying \eqref{eq:regularity-Lambda} and $\Lambda_{\emptyset}^*$-almost every $F\in\Lip(\XN)$ has $\Delta_F=\emptyset$ and for every $1\leq k\leq \max_j\ar(R_j)$, $\Lambda_{k}$ is invariant with respect to permutations that coincide with the identity on $s_{(1,\ldots,1)}$ and satisfies \eqref{eq:regularity-Lambda}, \eqref{eq:s-alpha}, and \eqref{eq:substructures}, and $\Lambda_{k}^*$ is as defined in \eqref{eq:mu-alpha-star}.
\end{cor}

\section{Projection to limit structures}\label{section:limits}

In addition to the precise structural properties exhibited by the processes in Examples \ref{example:coag} and \ref{example:frag} and all those covered by Theorem \ref{thm:Levy-Ito}, the sample paths of exchangeable combinatorial Feller processes are well behaved when projected into an appropriate space of limit objects, regardless of whether $\XN$ satisfies $\odap$-DAP.

In the case of partition-valued processes, as in Examples \ref{example:coag} and \ref{example:frag}, the appropriate limit space is the ranked-simplex $\Delta^{\downarrow}$.
On the space of undirected graphs, the space of graph limits, as studied in \cite{Lovasz2013,LovaszSzegedy2006}, is appropriate.
Most generally, we consider the projection into the space of dense limits of relational structures, as studied recently in \cite{AroskarCummings2014}.

Let $\mathcal{L}$ be any signature and $\XN\subseteq\lN$ be a Fra\"{i}ss\'e space.
For any $\mathfrak{S},\mathfrak{S}'\in\age(\XN)$ with $\dom{\mathfrak{S}}=[m]$ and $\dom{\mathfrak{S}'}=[n]$, we define the {\em density of $\mathfrak{S}$ in $\mathfrak{S}'$} by
\[
\delta(\mathfrak{S},\mathfrak{S}'):=\frac{1}{n^{\downarrow m}}\sum_{\phi:[m]\to[n]}\mathbf{1}\{\mathfrak{S}'^{\phi}=\mathfrak{S}\},
\]
where each $\phi:[m]\to[n]$ is an injection and $n^{\downarrow m}:=n(n-1)\cdots(n-m+1)$.
For any $\mathfrak{M}\in\XN$, we define the {\em limiting density of $\mathfrak{S}$ in $\mathfrak{M}$} by
\begin{equation}\label{eq:density}
\delta(\mathfrak{S},\mathfrak{M}):=\lim_{n\to\infty}\delta(\mathfrak{S},\mathfrak{M}|_{[n]}),\quad\text{if it exists}.
\end{equation}

For any $\mathfrak{M}\in\XN$, we define the {\em combinatorial limit} of $\mathfrak{M}$ by $\xnorm{\mathfrak{M}}:=(\delta(\mathfrak{S},\mathfrak{M}))_{\mathfrak{S}\in\age(\XN)}$, provided $\delta(\mathfrak{S},\mathfrak{M})$ exists for all $\mathfrak{S}\in\age(\XN)$.

\begin{definition}
A random structure $\mathfrak{X}$ is {\em dissociated} if $\mathfrak{X}|_S$ and $\mathfrak{X}|_{S'}$ are independent whenever $S,S'\subseteq\Nb$ are disjoint.
\end{definition}

\begin{prop}\label{prop:dissociated}
Suppose $\xnorm{\mathfrak{M}}$ exists for some $\mathfrak{M}\in\XN$.
Then $\xnorm{\mathfrak{M}}$ determines a unique exchangeable, dissociated probability measure on $\XN$.
\end{prop}

\begin{proof}
Let $\xnorm{\mathfrak{M}}=(\delta(\mathfrak{S},\mathfrak{M}))_{\mathfrak{S}\in\age(\XN)}$ be the limiting densities of $\mathfrak{M}$.
For each $n\in\Nb$, we define 
\[\nu_n(\mathfrak{S}):=\delta(\mathfrak{S},\mathfrak{M}),\quad\mathfrak{S}\in\Xn.\]
By definition \eqref{eq:density}, $\nu_n(\mathfrak{S})\geq0$ for all $\mathfrak{S}\in\Xn$ and
\begin{eqnarray*}
\sum_{\mathfrak{S}\in\Xn}\nu_n(\mathfrak{S})&=&\sum_{\mathfrak{S}\in\Xn}\lim_{n\to\infty}\frac{1}{n^{\downarrow m}}\sum_{\phi:[m]\to[n]}\mathbf{1}\{\mathfrak{M}^{\phi}=\mathfrak{S}\}\\
&=&\lim_{n\to\infty}\sum_{\mathfrak{S}\in\Xn}\frac{1}{n^{\downarrow m}}\sum_{\phi:[m]\to[n]}\mathbf{1}\{\mathfrak{M}^{\phi}=\mathfrak{S}\}\\
&=&1\end{eqnarray*}
by the bounded convergence theorem.
Since $\delta(\mathfrak{S},\mathfrak{M})=\delta(\mathfrak{S}^{\sigma},\mathfrak{M})$ for all permutations $\sigma:[n]\to[n]$, $\nu_n$ is an exchangeable probability measure on $\Xn$.

Furthermore, for $k\leq m$ and $\mathfrak{S}\in\mathcal{X}_{[k]}$,
\begin{eqnarray*}
\lefteqn{\sum_{\mathfrak{S}'\in\mathcal{X}_{[m]}:\mathfrak{S}'|_{[k]}=\mathfrak{S}}\delta(\mathfrak{S}',\mathfrak{M})=}\\
&=&
\sum_{\mathfrak{S}'\in\mathcal{X}_{[m]}:\mathfrak{S}'|_{[k]}=\mathfrak{S}}\lim_{n\to\infty}\frac{1}{n^{\downarrow m}}\sum_{\phi:[m]\to[n]}\mathbf{1}\{\mathfrak{M}^{\phi}=\mathfrak{S}'\}\\
&=&\lim_{n\to\infty}\frac{1}{n^{\downarrow k}}\sum_{\phi':[k]\to[n]}\frac{1}{(n-k)^{\downarrow(m-k)}}\sum_{\text{extensions }\phi:[m]\to[n]\text{ of }\phi'}\sum_{\mathfrak{S}'\in\mathcal{X}_{[m]}:\mathfrak{S}'|_{[k]}=\mathfrak{S}}\mathbf{1}\{\mathfrak{M}^{\phi}=\mathfrak{S}'\}\\
&=&\lim_{n\to\infty}\frac{1}{n^{\downarrow k}}\sum_{\phi':[k]\to[n]}\mathbf{1}\{\mathfrak{M}^{\phi'}=\mathfrak{S}\}\\
&=&\delta(\mathfrak{S},\mathfrak{M}).
\end{eqnarray*}
Thus, the set function
\[\nu(\{\mathfrak{N}\in\XN: \mathfrak{N}|_{[k]}=\mathfrak{S}\})=\nu_k(\mathfrak{S}),\quad\mathfrak{S}\in\mathcal{X}_{[k]},\quad k\in\Nb,\]
is an additive pre-measure on $\XN$, which extends to a unique probability measure on $\XN$ by Carath\`eodory's extension theorem.
Since each $\nu_n$ measure is exchangeable and the sets of the form $\{\mathfrak{M}\in\XN: \mathfrak{M}|_{[k]}=\mathfrak{S}\}$, for $\mathfrak{S}\in\mathcal{X}_{[k]}$, $k\in\Nb$, constitute a generating $\pi$-system of the Borel $\sigma$-field on $\XN$, it follows that $\nu$ is exchangeable.

Finally, we observe that for $S,T\subseteq\Nb$ with $S\cap T=\emptyset$, $\mathfrak{S}\in\mathcal{X}_S$ and $\mathfrak{T}\in\mathcal{X}_T$,
\begin{eqnarray*}
\lefteqn{\nu(\{\mathfrak{N}\in\XN: \mathfrak{N}|_{S}=\mathfrak{S}\text{ and }\mathfrak{N}|_T=\mathfrak{T}\})=}\\
&=&\lim_{n\to\infty}\frac{1}{n^{\downarrow(|S|+|T|)}}\sum_{\phi:S\cup T\to[n]}\mathbf{1}\{\mathfrak{M}^{\phi}|_S=\mathfrak{S}\}\mathbf{1}\{\mathfrak{M}^{\phi}|_T=\mathfrak{T}\}\\
&=&\lim_{n\to\infty}\lim_{m\to\infty}\frac{1}{n^{\downarrow|S|}(m-|S|)^{\downarrow|T|}}\sum_{\varphi:S\to[n]}\sum_{\psi:T\to[m]\setminus\cod\varphi}\mathbf{1}\{\mathfrak{M}^{\varphi}=\mathfrak{S}\}\mathbf{1}\{\mathfrak{M}^{\psi}=\mathfrak{T}\}\\
&=&\lim_{n\to\infty}\frac{1}{n^{\downarrow|S|}}\sum_{\varphi:S\to[n]}\mathbf{1}\{\mathfrak{M}^{\varphi}=\mathfrak{S}\}\lim_{m\to\infty}\frac{1}{(m-|S|)^{\downarrow|T|}}\sum_{\psi:T\to[m]\setminus\cod\varphi}\mathbf{1}\{\mathfrak{M}^{\psi}=\mathfrak{T}\}\\
&=&\delta(\mathfrak{S},\mathfrak{M})\delta(\mathfrak{T},\mathfrak{M})\\
&=&\nu(\{\mathfrak{N}\in\XN: \mathfrak{N}|_S=\mathfrak{S}\})\nu(\{\mathfrak{N}\in\XN: \mathfrak{N}|_T=\mathfrak{T}\}),
\end{eqnarray*}
where $\cod\varphi$ denotes the codomain of $\varphi$.
Thus,  $\nu$ is dissociated.
\end{proof}

\begin{prop}\label{prop:limit-exists}
Let $\mathfrak{X}$ be an exchangeable random $\mathcal{L}$-structure.
Then $\xnorm{\mathfrak{X}}$ exists with probability 1.
\end{prop}

\begin{proof}
Since $\age(\XN)$ is countable it suffices to show that $\delta(\mathfrak{S},\mathfrak{X})$ exists with probability 1 for every $\mathfrak{S}\in\age(\XN)$.
Let $\mathfrak{S}\in\mathcal{X}_{[m]}$, $m\in\Nb$, and suppose $\mathfrak{X}$ is dissociated.
For $n\geq1$, we define
\[Z_n:=\delta(\mathfrak{S},\mathfrak{M}|_{[n]})=\frac{1}{n^{\downarrow m}}\sum_{\phi:[m]\to[n]}\mathbf{1}\{\mathfrak{X}^{\phi}=\mathfrak{S}\}\]
and $\mathcal{F}_n:=\sigma\langle Z_{n+1},Z_{n+2},\ldots\rangle$, the reverse filtration generated by $(Z_n)_{n\geq1}$.
By exchangeability of $\mathfrak{X}$, we observe that $\mathbb{E}Z_n=\mathbb{P}\{\mathfrak{X}^{\phi}=\mathfrak{S}\}$ and $\mathbb{E}(Z_n\mid\mathcal{F}_n)=Z_{n+1}$ for every $n\geq1$.
Thus, $(Z_n)_{n\geq1}$ is a reverse martingale with respect to $(\mathcal{F}_n)_{n\geq1}$.
It follows that $Z_n\to Z_{\infty}=\delta(\mathfrak{S},\mathfrak{X})$ a.s.
Since $\mathfrak{X}$ is dissociated, $\delta(\mathfrak{S},\mathfrak{X})$ depends only on the tail $\sigma$-field generated by $(Z_n)_{n\geq1}$ and, therefore, is deterministic by the 0-1 law.
By the Aldous--Hoover--Kallenberg theorem, every exchangeable random structure is conditionally dissociated given its tail $\sigma$-field.
It follows that $\delta(\mathfrak{S},\mathfrak{X})$ exists almost surely whenever $\mathfrak{X}$ is exchangeable.
This completes the proof. 

\end{proof}

We write $\E(\XN)$ to denote the space of exchangeable probability measures on $\XN$.
Let $\mathbf{X}=\{\mathbf{X}_{\mathfrak{M}}: \mathfrak{M}\in\XN\}$ be an exchangeable Feller process on $\XN$.
For $D\in\E(\XN)$, we define $\mathbf{X}_{D}:=(\mathfrak{X}_t)_{t\geq0}$ as the process obtained by first taking $\mathfrak{X}_0\sim D$ and then putting $\mathbf{X}_D=\mathbf{X}_{\mathfrak{M}}$ on the event $\mathfrak{X}_0=\mathfrak{M}$.
We define $\xnorm{\mathbf{X}_{D}}:=(\xnorm{\mathfrak{X}_t})_{t\geq0}$, provided $\xnorm{\mathfrak{X}_t}$ exists for all $t\geq0$.
We now show that $\xnorm{\mathbf{X}_D}$ exists for every $D\in\E(\XN)$ and determines a Feller process $\xnorm{\mathbf{X}_{\E(\XN)}}:=\{\xnorm{\mathbf{X}_D}: D\in\E(\XN)\}$ on $\E(\XN)$, which we furnish with the topology and Borel $\sigma$-field induced by the metric
\begin{equation}\label{eq:E-metric}
\rho(\gamma,\gamma'):=\sum_{n\in\Nb}2^{-n}\sum_{\mathfrak{S}\in\XN}|\gamma(\{\mathfrak{M}: \mathfrak{M}|_{[n]}=\mathfrak{S}\})-\gamma'(\{\mathfrak{M}: \mathfrak{M}|_{[n]}=\mathfrak{S}\})|
\end{equation}
for any $\gamma,\gamma'\in\E(\XN)$.
(The metric \eqref{eq:E-metric} is analogous to the metric used by Diaconis and Janson \cite{MR2463439} in their discussion of graph limits.)

\begin{theorem}\label{thm:limit process}
Let $\XN$ be a Fra\"{i}ss\'e space and let $\mathbf{X}=\{\mathbf{X}_{\mathfrak{M}}: \mathfrak{M}\in\XN\}$ be an exchangeable Feller process on $\XN$.
Then $\xnorm{\mathbf{X}_{\E(\XN)}}$ exists almost surely and is a Feller process on $\E(\XN)$.
\end{theorem}

\begin{proof}

Let $\mathcal{L}=\{R_1,\ldots,R_r\}$ be any signature with $m^*:=\max_{1\leq j\leq r}\ar(R_j)$ and let $\XN\subseteq\lN$ be a Fra\"{i}ss\'e space.
Below we fix $D\in\E(\XN)$ and let $\mathbf{X}_D=(\mathfrak{X}_t)_{t\in T}$ be the process with initial distribution $D$ as defined above.
The proof covers both discrete time and continuous time processes, but we only state the proof in the continuous time case, which is harder.  The discrete time case follows by analogous, often much simpler arguments.

\begin{proof}[Proof of Existence]
Since we construct $\mathbf{X}_D$ to have an exchangeable initial distribution $D\in\E(\XN)$ and an exchangeable transition probability measure, the marginal distribution of $\mathfrak{X}_t$ is exchangeable for every $t\geq0$ and, by Proposition \ref{prop:limit-exists}, the limit $\xnorm{\mathfrak{X}_t}$ exists with probability 1 for all fixed times $t\geq0$.
Countable additivity of probability measures immediately implies the simultaneous existence of $\xnorm{\mathfrak{X}_t}$ for any countable collection of times $\mathcal{C}\subset[0,\infty)$.
In particular, $\xnorm{\mathbf{X}_D}$ exists for a discrete time process.
Technical difficulty arises when considering simultaneous existence of $\xnorm{\mathfrak{X}_t}$ at the uncountable collection of times $t\in[0,\infty)$.

The proof of existence follows a standard recipe whereby we show that the upper and lower densities,
\begin{align*}
	\delta^+(\mathfrak{S},\mathfrak{X}_t)&:=\limsup_{n\to\infty}\frac{1}{n^{\downarrow m}}\sum_{\phi:[m]\to[n]}\mathbf{1}\{\mathfrak{X}^{\phi}_t=\mathfrak{S}\}\quad\text{and}\\
	\delta^-(\mathfrak{S},\mathfrak{X}_t)&:=\liminf_{n\to\infty}\frac{1}{n^{\downarrow m}}\sum_{\phi:[m]\to[n]}\mathbf{1}\{\mathfrak{X}^{\phi}_t=\mathfrak{S}\},
\end{align*}
respectively, coincide at all times $t\in[0,\infty)$ with probability 1.

We write $\mathbf{X}_{[0,1]}:=(\mathfrak{X}_t)_{t\in[0,1]}$ to denote the evolution of $\mathbf{X}_D$ at times $t\in[0,1]$.
Since the processes we study are time homogeneous, it is sufficient to first prove existence of $\xnorm{\mathbf{X}_{[0,1]}}:=(\xnorm{\mathfrak{X}_t})_{t\in[0,1]}$ with probability 1 and then deduce existence of $\xnorm{\mathbf{X}_D}$ by countable additivity.
By the c\`adl\`ag paths property of $\mathbf{X}_D$, we can regard $\mathbf{X}_{[0,1]}=(\mathfrak{X}_{[0,1]}|_{\rng x})_{\vec x\in\Nb^{m^*}}$ as an $m^*$-dimensional array taking values in the Polish space $\mathcal{D}=\mathcal{D}([0,1]\to\mathcal{X}_{[m^*]})$ of c\`adl\`ag functions $[0,1]\to\mathcal{X}_{[m^*]}$.
Viewed in this way, $\mathbf{X}_{[0,1]}$ is an exchangeable, symmetric array valued in a Polish space and, thus, the Aldous--Hoover--Kallenberg theorem \cite{Aldous1981,Hoover1979,KallenbergSymmetries} for exchangeable random arrays characterizes the law of $\mathbf{X}_{[0,1]}$ in terms of a Borel measurable function $g:[0,1]^{2^{m^*}}\to\mathcal{X}_{[m^*]}$ such that $\mathbf{X}_{[0,1]}\equalinlaw\mathbf{Y}:=(\mathfrak{Y}_{\vec x})_{\vec x\in\Nb^{m^*}}$ for
\[\mathfrak{Y}_{\vec x}:=g((\xi_{s})_{s\subseteq\rng x}),\quad \vec x\in\Nb^{m^*},\]
where $(\xi_{s})_{s\subset\Nb: |s|\leq m^*}$ are i.i.d.\ Uniform$[0,1]$ random variables.
More specifically, $\mathbf{Y}$ is conditionally dissociated given its tail $\sigma$-field, and so we may proceed under the assumption that $\mathbf{X}_{[0,1]}$ is dissociated.

To show that $\xnorm{\mathfrak{X}_{t}}$ exists simultaneously for all $t\in[0,1]$, we show that the upper and lower densities $\delta^+(\mathfrak{S},\mathfrak{X}_t)$ and $\delta^-(\mathfrak{S},\mathfrak{X}_t)$ coincide for all $t\in[0,1]$ with probability 1 for all $\mathfrak{S}\in\age(\XN)$.
In particular, we show that for all $\mathfrak{S}\in\age(\XN)$,
\[\mathbb{P}\{\sup_{t\in[0,1]}|\delta^+(\mathfrak{S},\mathfrak{X}_t)-\delta^{-}(\mathfrak{S},\mathfrak{X}_t)|=0\}=1.\]

By the projective Markov property of $\mathbf{X}$, each entry $\mathbf{X}_{[0,1]}|_{\rng\vec x}$ of $\mathbf{X}_{[0,1]}$ has finitely many discontinuities for all $\vec x\in\Nb^{m^*}$ and, therefore, so must the restriction $\mathbf{X}^{[n]}_{[0,1]}:=(\mathfrak{X}_t|_{\rng \vec x})_{\vec x\in[n]^{m^*}}$ for every $n\in\Nb$.
For every $\varepsilon>0$, there is a finite subset $T_{\varepsilon}\subset[0,1]$ and an at most countable partition $I_1,I_2,\ldots$ of $[0,1]\setminus T_{\varepsilon}$ such that, for every $\vec x\in\Nb^{m^*}$,
\begin{align*}
	&\mathbb{P}\{\mathbf{X}_{[0,1]}|_{\rng\vec x}\text{ is discontinuous at }t\in T_{\varepsilon}\}\geq\varepsilon\quad\text{and}\\
	&\mathbb{P}\{\mathbf{X}_{[0,1]}|_{\rng\vec x}\text{ is discontinuous on }I_j\}<\varepsilon,\quad \text{ for every }j=1,2,\ldots.
\end{align*}
For suppose there were no such partition.
Then there would be some $t\not\in T_{\varepsilon}$ such that 
\[\mathbb{P}\{\mathbf{X}_{[0,1]}|_{\rng\vec x}\text{ is discontinuous on }(t-1/n,t+1/n)\}\geq\varepsilon\]
for every $n\in\Nb$, implying
\[\mathbb{P}\{\mathbf{X}_{[0,1]}|_{\rng\vec x}\text{ is discontinuous at }t\not\in T_{\varepsilon}\}\geq\varepsilon,\]
a contradiction.

The action of relabeling by any permutation $\sigma:\Nb\to\Nb$ is ergodic for exchangeable, dissociated $\mathcal{L}$-structures, implying
\begin{eqnarray*}
\lefteqn{\mathbb{P}\{\mathbf{X}_{[0,1]}|_{[m^*]}\text{ is discontinuous on }I_j\}=}\\
&=&\lim_{n\to\infty}\binom{n}{m^*}^{-1}\sum_{s\subseteq[n]: |s|=m^*}\mathbf{1}\{\mathbf{X}_{[0,1]}|_{s}\text{ is discontinuous on }I_j\}<\varepsilon\quad\text{a.s.}
\end{eqnarray*}
for every subinterval $I_1,I_2,\ldots$; thus, $(\delta^+(\mathfrak{S},\mathfrak{X}_t))_{t\in I_j}$ and $(\delta^-(\mathfrak{S},\mathfrak{X}_t))_{t\in I_j}$ cannot vary by more than $\varepsilon|\dom\mathfrak{S}|^{m^*}$ for any $j=1,2,\ldots$.
Since $\delta^+(\mathfrak{S},\mathfrak{X}_t)=\delta^-(\mathfrak{S},\mathfrak{X}_t)$ for the endpoints of $I_j$, we must have
\[\mathbb{P}\{\sup_{t\in I_j}|\delta^+(\mathfrak{S},\mathfrak{X}_t)-\delta^-(\mathfrak{S},\mathfrak{X}_t)|\leq 2\varepsilon|\dom\mathfrak{S}|^{m^*}\}=1\]
for all $\varepsilon>0$.

We can cover $[0,1]$ by the countable set of intervals $I_1,I_2,\ldots$ and a countable nonrandom set of times $S=\bigcup_{\varepsilon>0} T_{\varepsilon}$; whence,
\[\mathbb{P}\{\sup_{t\in[0,1]}|\delta^+(\mathfrak{S},\mathfrak{X}_t)-\delta^-(\mathfrak{S},\mathfrak{X}_t)|\leq 2\varepsilon|\dom\mathfrak{S}|^{m^*}\}=1\]
for all $\varepsilon>0$, for every $\mathfrak{S}\in\age(\XN)$, from which continuity from above implies
\begin{eqnarray*}
\lefteqn{\mathbb{P}\{\sup_{t\in[0,1]}|\delta^+(\mathfrak{S},\mathfrak{X}_t)-\delta^-(\mathfrak{S},\mathfrak{X}_t)|=0\}=}\\
&=&\lim_{\varepsilon\downarrow0}\mathbb{P}\{\sup_{t\in[0,1]}|\delta^+(\mathfrak{S},\mathfrak{X}_t)-\delta^-(\mathfrak{S},\mathfrak{X}_t)|\leq 2\varepsilon|\dom\mathfrak{S}|^{m^*}\}=1.\end{eqnarray*}
It follows that $\delta(\mathfrak{S},\mathfrak{X}_t)$ exists simultaneously for all $t\in[0,1]$ with probability 1 for every $\mathfrak{S}\in\age(\XN)$.
Since $\age(\XN)$ is countable, we conclude that $\xnorm{\mathbf{X}_{[0,1]}}$ exists with probability 1 and, thus, $\xnorm{\mathbf{X}_D}$ exists with probability 1 and determines a process on $\E(\XN)$.

\end{proof}

\begin{proof}[Proof of Feller property]
By Theorem \ref{thm:continuous-Lambda}, we can assume $\mathbf{X}$ is constructed from a Poisson point process $\Phi=\{(t,\Phi_t)\}\subseteq[0,\infty)\times\Lip(\XN)$ with intensity $dt\otimes\Lambda$ for some measure $\Lambda$ satisfying \eqref{eq:regularity-Lambda}.
As every $F\in\Lip(\XN)$ restricts to a unique $F^{[n]}\in\Lip(\Xn)$ as in \eqref{eq:restrict-Lip}, we can construct a process $\mathbf{\Psi}=(\Psi_t)_{t\in[0,\infty)}$ on $\Lip(\XN)$ from $\mathbf{\Phi}$ as follows.
For every $n\in\Nb$, we define $\mathbf{\Psi}^{[n]}=(\Psi_t^{[n]})_{t\in[0,\infty)}$ on $\Lip(\Xn)$ by putting 
\begin{itemize}
	\item $\Psi^{[n]}_0=\idn$, the identity $\Xn\to\Xn$, 
	\item $\Psi^{[n]}_t=\Phi_t^{[n]}\circ\Psi_{t-}^{[n]}$, if $t$ is an atom time of $\mathbf{\Phi}$ with $\Phi_t^{[n]}\neq\idn$, and
	\item $\Psi_t^{[n]}=\Psi_{t-}^{[n]}$, otherwise,
\end{itemize}
for $\Psi_{t-}^{[n]}:=\lim_{s\uparrow t}\Psi_s^{[n]}$ and $\Phi_t^{[n]}\circ\Psi_{t-}^{[n]}$ denoting the composition of Lipschitz continuous functions $\Xn\to\Xn$.
We define $\Psi$ as the projective limit of the collection $(\mathbf{\Psi}^{[n]})_{n\in\Nb}$, which is an exchangeable Feller process on $\Lip(\XN)$ under the topology induced by the metric $d_{\Lip(\XN)}$ defined in \eqref{eq:d-Lip}.
In this way, we can assume $\mathbf{X}=\{\mathbf{X}_{\mathfrak{M}}: \mathfrak{M}\in\XN\}$ is constructed from the same $\mathbf{\Psi}$ by $\mathbf{X}_{\mathfrak{M}}:=(\Psi_t(\mathfrak{M}))_{t\in[0,\infty)}$, for every $\mathfrak{M}\in\XN$.

Let $\Phi$ be an exchangeable Lipschitz continuous random function $\XN\to\XN$ governed by exchangeable measure $\varphi$.
Any such $\Phi\sim\varphi$ determines a unique exchangeable, consistent transition probability measure $P^{\Phi}$ on $\XN$ by
\begin{equation}\label{eq:P-Phi}
P^{\Phi}(\mathfrak{M},A):=\varphi(\{F\in\Lip(\XN): F(\mathfrak{M})\in A\}),\quad\mathfrak{M}\in\XN,\quad A\Borel\XN.
\end{equation}
Any exchangeable, consistent transition probability $P$ acts on $\E(\XN)$ by $\mu\mapsto P\mu$, where
\begin{equation}\label{eq:P-mu}P\mu(A):=\int_{\XN}P(\mathfrak{M},A)\mu(d\mathfrak{M}),\quad A\Borel\XN.\end{equation}
This map is Lipschitz continuous in the metric \eqref{eq:E-metric}.

If $\mathfrak{X}$ is an exchangeable random $\mathcal{L}$-structure and $\Phi\sim\varphi$ is an exchangeable Lipschitz continuous function $\XN\to\XN$, then $\Phi(\mathfrak{X})\equalinlaw\sigma\Phi\sigma^{-1}(\mathfrak{X}^{\sigma})=\Phi(\mathfrak{X})^{\sigma}$ for all permutations $\sigma:\Nb\to\Nb$ and Proposition \ref{prop:limit-exists} implies that $\xnorm{\Phi(\mathfrak{X})}$ exists with probability 1.

Consequently, for any $D\in\E(\XN)$, the marginal distribution of $\mathfrak{X}_t$ in $\mathbf{X}_{D}=(\mathfrak{X}_t)_{t\in[0,\infty)}$ coincides with that of $\Psi_t(\mathfrak{X}_0)$ for $\mathfrak{X}_0\sim D$ and, thus, 
\[\xnorm{\mathfrak{X}_t}\equalinlaw\xnorm{\Psi_t(\mathfrak{X}_0)}\equalinlaw P^{\Psi_t}\xnorm{\mathfrak{X}_0}=P^{\Psi_t}D.\]

To establish the Feller property for $\xnorm{\mathbf{X}_{\E(\XN)}}$, we let $\mathbf{P}=(\mathbf{P}_t)_{t\in[0,\infty)}$ be the semigroup of $\xnorm{\mathbf{X}_{\E(\XN)}}$, that is, each $\mathbf{P}_t$ acts on bounded, continuous functions $g:\E(\XN)\to\mathbb{R}$ by
\[\mathbf{P}_t g(D):=\mathbb{E}(g(\xnorm{\mathfrak{X}_t})\mid\xnorm{\mathfrak{X}_0}=D).\]
We need to show that, for all bounded, continuous $g:\E(\XN)\to\mathbb{R}$,
\begin{itemize}
	\item[(a)] $D\to\mathbf{P}_tg(D)$ is continuous for all $t>0$ and
	\item[(b)] $\lim_{t\downarrow0}\mathbf{P}_tg(D)=g(D)$ for all $D\in\E(\XN)$.
\end{itemize}
To establish (a), we let $g:\E(\XN)\to\mathbb{R}$ be any continuous function.
By compactness of $\E(\XN)$, any such $g$ is uniformly continuous and, therefore, bounded.
The dominated convergence theorem and Lipschitz continuity of the action defined in \eqref{eq:P-mu} implies (a).

For (b), we write $P^{\idN}$ to denote the transition probability on $\XN$ corresponding to the identity map $\idN:\XN\to\XN$, for which $P^{\idN}(\mathfrak{M},\mathfrak{M})\equiv1$.
Let $\psi_t$ be the law governing $\Psi_t$ and let $\mathbf{I}$ denote the measure that assigns probability 1 to $\idN\in\Lip(\XN)$.
We first show $\Psi_t\to_P\idN$ as $t\downarrow0$, where $\to_P$ denotes {\em convergence in probability}, that is,
\[\lim_{t\downarrow0}\mathbb{P}\{\Psi_t^{[n]}\neq\idn\}=0\quad\text{for all }n\geq1.\]

In the direction of a contradiction, we assume
\[\limsup_{t\downarrow0}\mathbb{P}\{\Psi_t^{[n]}\neq\idn \text{ for some }n\geq1\}>0,\]
by which there must be some $n\in\Nb$, $F\in\Lip(\Xn)\setminus\{\idn\}$, and $\varrho>0$ such that
\[\limsup_{t\downarrow0}\mathbb{P}\{\Psi_t^{[n]}=F\}\geq\varrho.\]
Given any such $F\in\Lip(\Xn)\setminus\{\idn\}$, we let $\mathfrak{S}\in\Xn$ be such that $F(\mathfrak{S})\neq\mathfrak{S}$.
For any $\mathfrak{M}\in\XN$ such that $\mathfrak{M}|_{[n]}=\mathfrak{S}$, we have
\begin{eqnarray}
\notag\mathbb{P}\{\mathfrak{X}_t^{[n]}\neq\mathfrak{S}\mid\mathfrak{X}_0=\mathfrak{M}\}&\leq&\mathbb{P}\{\mathbf{X}_{\mathfrak{M}}^{[n]}\text{ is discontinuous on }[0,t]\}\\
&\leq&1-\exp\{-t\Lambda(\{F\in\Lip(\XN): F^{[n]}\neq\idn\})\}\label{eq:leq}
\end{eqnarray}
and
\begin{eqnarray}
\mathbb{P}\{\mathfrak{X}_t^{[n]}\neq\mathfrak{S}\mid\mathfrak{X}_0=\mathfrak{M}\}&\geq&\mathbb{P}\{\Psi_t^{[n]}=F\}\label{eq:geq}
\end{eqnarray}
By  \eqref{eq:leq} and the righthand side of \eqref{eq:regularity-Lambda}, 
\begin{eqnarray*}
\lefteqn{\limsup_{t\downarrow0}\mathbb{P}\{\mathbf{X}_{\mathfrak{M}}^{[n]}\text{ is discontinuous on }[0,t]\}\leq}\\
&\leq&\limsup_{t\downarrow0}1-\exp\{-t\Lambda(\{F\in\Lip(\XN): F^{[n]}\neq\idn\})\\
&=&0.
\end{eqnarray*}
On the other hand, \eqref{eq:geq} implies
\[\limsup_{t\downarrow0}\mathbb{P}\{\mathbf{X}_{\mathfrak{M}}^{[n]}\text{ is discontinuous on }[0,t]\}\geq\limsup_{t\downarrow0}\mathbb{P}\{\Psi_t^{[n]}=F\}\geq\varrho>0,\]
establishing a contradiction.
We conclude that 
\[\limsup_{t\downarrow0}\mathbb{P}\{\Psi_t^{[n]}\neq\idn\text{ for some }n\geq1\}=0\]
and, therefore, $\Psi_t\to_P\idN$ as $t\downarrow0$.
It follows that 
\[\xnorm{\mathfrak{X}_t}\equalinlaw\xnorm{\Psi_t(\mathfrak{X}_0)}\equalinlaw P^{\Psi_t}\xnorm{\mathfrak{X}_0}=P^{\Psi_t}D\to_P P^{\idN} D=D\]
as $t\downarrow0$.
Part (b) of the Feller property follows, completing the proof.

\end{proof}

\end{proof}

\begin{theorem}\label{thm:disc}
Let $\mathbf{X}$ be a continuous time, exchangeable Feller process on a Fra\"{i}ss\'e space $\XN$ having $\odap$-DAP.
Then, for every $D\in\E(\XN)$, the sample paths of $\xnorm{\mathbf{X}_D}$ are continuous at all times except possibly those of jumps from the $\Lambda_{\emptyset}^*$ measure in Theorem \ref{thm:Levy-Ito}.
\end{theorem}

\begin{proof}
Let $D\in\E(\XN)$ and $\mathbf{X}_D=(\mathfrak{X}_t)_{t\in[0,\infty)}$ have initial distribution $\mathfrak{X}_0\sim D$.
By Theorem \ref{thm:Levy-Ito}, we can construct $\mathbf{X}_D$ from a Poisson point process with intensity $dt\otimes\Lambda$ for some exchangeable measure $\Lambda$ which decomposes as
\[\Lambda=\Lambda_{\emptyset}^{*}+\sum_{k=1}^{\max_j\ar(R_j)}\sum_{\alpha\vdash k}\Lambda_{\alpha}^*,\]
where $\Lambda_{\emptyset}^*$ is an exchangeable measure satisfying \eqref{eq:regularity-Lambda} and for which $\Lambda_{\emptyset}^*$-almost every $F\in\Lip(\XN)$ has $\Delta_F=\emptyset$, where $\Delta_F$ is defined in \eqref{eq:intersect}, each $\Lambda_{\alpha}$, $\alpha\vdash k$, $1\leq k\leq\max_{j}\ar(R_j)$, satisfies \eqref{eq:regularity-Lambda}, \eqref{eq:intersect}, and \eqref{eq:s-alpha}, and $\Lambda_{\alpha}^*$ is defined as in \eqref{eq:mu-alpha-star}.

By Theorem \ref{thm:limit process}, $\xnorm{\mathbf{X}_D}$ is a Feller process and, therefore, has a version with c\`adl\`ag sample paths.
For this c\`adl\`ag version, $\lim_{s\uparrow t}\xnorm{\mathfrak{X}_s}$ exists for all $t>0$ with probability 1.
Although the map $\xnorm{\cdot}:\XN\to\E(\XN)$ is not continuous, we still have $\lim_{s\uparrow t}\xnorm{\mathfrak{X}_s}=\xnorm{\mathfrak{X}_{t-}}$ for all $t>0$ with probability 1, as we now show.

Suppose $t>0$ is a discontinuity time for $\mathbf{X}_D$.
Exchangeability of $\mathbf{X}_D$ and the Aldous--Hoover theorem implies that $\mathfrak{X}_{t-}$ is exchangeable for every $t>0$ and, thus, $\xnorm{\mathfrak{X}_{t-}}$ exists with probability 1.
If $\lim_{s\uparrow t}\xnorm{\mathfrak{X}_s}\neq\xnorm{\mathfrak{X}_{t-}}$, then there exists $\varepsilon>0$ such that
\[\rho(\lim_{s\uparrow t}\xnorm{\mathfrak{X}_s},\xnorm{\mathfrak{X}_{t-}})>\varepsilon,\]
where $\rho$ is the metric defined in \eqref{eq:E-metric}.
In particular, there exists $m\in\Nb$ and $\mathfrak{S}\in\mathcal{X}_{[m]}$ such that 
\[|\lim_{s\uparrow t}\delta(\mathfrak{S},\mathfrak{X}_s)-\delta(\mathfrak{S},\mathfrak{X}_{t-})|>\varepsilon.\]
These limits exist with probability 1, allowing us to replace limits with limits inferior to get
\begin{eqnarray*}
\lefteqn{0\leq|\liminf_{s\uparrow t}\liminf_{n\to\infty}\frac{1}{n^{\downarrow m}}\sum_{\phi:[m]\to[n]}\mathbf{1}\{\mathfrak{X}_s^{\phi}=\mathfrak{S}\}-\mathbf{1}\{\mathfrak{X}_{t-}^{\phi}=\mathfrak{S}\}|\leq}\\
&\leq&\liminf_{s\uparrow t}\liminf_{n\to\infty}\frac{1}{n^{\downarrow m}}\sum_{\phi:[m]\to[n]}|\mathbf{1}\{\mathfrak{X}_s^{\phi}=\mathfrak{S}\}-\mathbf{1}\{\mathfrak{X}_{t-}^{\phi}=\mathfrak{S}\}|\\
&\leq&2\liminf_{s\uparrow t}\liminf_{n\to\infty}\frac{1}{n^{\downarrow m}}\sum_{\phi:[m]\to[n]}\mathbf{1}\{\mathbf{X}_D^{\phi}\text{ is discontinuous on }[s,t)\}.
\end{eqnarray*}
Combining the bounded convergence theorem and Fatou's lemma with the exchangeability and projectivity properties of $\mathbf{X}_D$ and its construction from the process $\mathbf{\Psi}$ on $\Lip(\XN)$, we see
\begin{eqnarray*}
\lefteqn{\mathbb{E}\left[\liminf_{s\uparrow t}\liminf_{n\to\infty}\frac{1}{n^{\downarrow m}}\sum_{\phi:[m]\to[n]}\mathbf{1}\{\mathbf{X}_D^{\phi}\text{ is discontinuous on }[s,t)\}\right]\leq}\\
&\leq&\liminf_{s\uparrow t}\liminf_{n\to\infty}\frac{1}{n^{\downarrow m}}\sum_{\phi:[m]\to[n]}\mathbb{E}\left[\mathbf{1}\{\mathbf{X}_D^{\phi}\text{ is discontinuous on }[s,t)\}\right]\\
&\leq&\liminf_{s\uparrow t}\liminf_{n\to\infty}1-\exp\{-(t-s)\Lambda(\{F\in\Lip(\XN): F^{[m]}\neq\text{id}_{[m]}\})\}\\
&=&0.
\end{eqnarray*}
By Markov's inequality,
\[\mathbb{P}\{|\lim_{s\uparrow t}\delta(\mathfrak{S},\mathfrak{X}_s)-\delta(\mathfrak{S},\mathfrak{X}_{t-})|>\varepsilon\}=0\quad\text{ for all }\varepsilon>0\text{ and all }\mathfrak{S}\in\age(\XN);\]
whence,
\[\mathbb{P}\{\rho(\lim_{s\uparrow t}\xnorm{\mathfrak{X}_s},\xnorm{\mathfrak{X}_{t-}})>\varepsilon\}=0\quad\text{ for all }\varepsilon>0.\]

To see that the discontinuities in $\xnorm{\mathbf{X}_D}$ occur only at the times of discontinuities from the $\Lambda_{\emptyset}^*$ measure, suppose $s>0$ is a discontinuity time for $\xnorm{\mathbf{X}_D}$.
Since $\xnorm{\mathbf{X}_D}$ has c\`adl\`ag sample paths and $\lim_{s\uparrow t}\xnorm{\mathfrak{X}_s}=\xnorm{\mathfrak{X}_{t-}}$ a.s., a continuity at $s>0$ implies that $|\delta(\mathfrak{S},\mathfrak{X}_{s-})-\delta(\mathfrak{S},\mathfrak{X}_{s})|>\varepsilon$ for some $\mathfrak{S}\in\age(\XN)$ and some $\varepsilon>0$.
By the strong law of large numbers,
\[\lim_{n\to\infty}\binom{n}{m^*}^{-1}\sum_{S\subseteq[n]: |S|=m^*}\mathbf{1}\{\mathfrak{X}_{s-}|_{S}\neq\mathfrak{X}_s|_{S}\}>0,\]
since otherwise exchangeability would imply $\delta(\mathfrak{S},\mathfrak{X}_{s-})=\delta(\mathfrak{S},\mathfrak{X}_s)$.
For any $1\leq k\leq m^*$ and $\alpha\vdash k$, condition \eqref{eq:substructures} implies that $R_j^{F(\bullet)}(\vec x)=R_j^{\bullet}(\vec x)$ for all $\vec x$ such that $s_{\alpha}\not\subseteq x$, for $\Lambda_{\alpha}$-almost every $F\in\Lip(\XN)$.
It follows that for $\Lambda$-almost every $F\in\Lip(\XN)$, if the discontinuity at time $s>0$ comes from the $\Lambda_{\alpha}^*$ measure, then there exists some $T\subseteq\Nb$ of partition type $\alpha\vdash k$ such that 
\[\lim_{n\to\infty}\binom{n}{m^*}^{-1}\sum_{S\subseteq[n]: |S|=m^*, T\subseteq S}\mathbf{1}\{\mathfrak{X}_{s-}|_S=\mathfrak{X}_s|_S\}\leq\lim_{n\to\infty}\binom{n}{m^*}^{-1}\binom{n}{m^*-k} =0.\]
Implying discontinuities can only occur at the time of atoms from the $\Lambda_{\emptyset}^*$ measure.

\end{proof}

\bibliography{refs}
\bibliographystyle{abbrv}

\end{document}